\RequirePackage{fix-cm}
\documentclass[12pt]{article}
 \usepackage{amsmath}
\usepackage{graphicx}
\usepackage{caption,subcaption}
\usepackage[utf8]{inputenc}
  \usepackage{nomencl}
 
 \usepackage{xcolor}
\usepackage{ctable}
\usepackage{datetime}
\usepackage{amsthm, amssymb}
\usepackage{lineno}

\usepackage{placeins} 
\usepackage{flafter}  
\usepackage{authblk}
\usepackage{blindtext}
\usepackage[misc,geometry]{ifsym} 

\usepackage{ifsym}
\newcommand{\ds}{\displaystyle}
\newcommand{\uu}{\vec{u}}
\newcommand{\qq}{\vec{q}}

\providecommand{\abs}[1]{\lvert#1\rvert}
\providecommand{\divv}[1]{\nabla\cdot#1}

\providecommand{\norm}[1]{\lVert#1\rVert}
\def\mathcolor#1#{\@mathcolor{#1}}
\def\@mathcolor#1#2#3{%
  \protect\leavevmode
  \begingroup
    \color#1{#2}#3%
  \endgroup
}

\providecommand{\keywords}[1]{\textbf{\textit{Index terms---}} #1}
\newtheorem{remark}{Remark}
\newtheorem{lemma}{Lemma}
\newtheorem{theorem}{Theorem}

\begin{document}

\title{Robust iterative schemes for non-linear poromechanics}

\author[1]{Manuel Borregales \thanks{Manuel.Borregales@uib.no (\Letter), $^\dagger$Florin.Radu@uib.no, \\$^\ddagger$Kundan.Kumar@uib.no,  	$^\circ$ Jan.Nordbotten@uib.no}}
\author[1]{Florin A. Radu$^\dagger$}
\author[1]{Kundan Kumar$^\ddagger$}
\author[1,2]{Jan M. Nordbotten$^\circ$}
\affil[1]{Department of Mathematics, University of Bergen, \\Norway.}
\affil[2]{Department of Civil and Environmental Engineering, Princeton University, Princeton, NJ, USA}
%



\hyphenation{li-near-pro-blem-rigu-rous Never-the-less poro-me-cha-nics appea-ring acce-le-ra-tion}

\maketitle
\begin{abstract}


We consider a  non-linear  extension  of  Biot's  model  for  poromechanics, wherein both the fluid flow and mechanical deformation are allowed to be non-linear. We perform an implicit discretization in time (backward  Euler) and propose  two iterative schemes for solving the non-linear problems appearing within each time step: a splitting algorithm extending the undrained split and fixed stress methods to non-linear problems, and a monolithic L-scheme. The convergence of both schemes is shown rigorously. Illustrative numerical examples are presented to confirm the applicability of the schemes and validate the theoretical results. 

\end{abstract}

\section{Introduction}
\label{intro}

Poromechanics, that is to say the coupled flow and mechanics of porous media, plays a crucial role in many societal relevant applications. These include geothermal energy extraction, energy storage in the subsurface, $CO_2$ sequestration and understanding of biological tissues. The increased role played by computers for the development and optimisation of (industrial) technologies for these applications enhances the need for improved mathematical models and robust numerical solvers for poromechanics.

The most common mathematical model for coupled flow and mechanics in porous media is the linear, quasi-stationary Biot model \cite{biot1941general,biot1954}. The model consists of two fully coupled partial differential equations, representing balance of forces for the mechanics and conservation of mass for (single-phase) flow in porous media. In terms of modelling, non-linear extensions are considered in \cite{coussy1995mechanics,Book_Lewis,Fullycouple4} or non-stationary Biot, i.e. Biot-Allard model in \cite{Mikelic3}. These coupled (linear or non-linear) equations are in practice impossible to solve analytically, and very challenging to solve numerically. 

It has been widely acknowledged that only simulation with fully coupled fluid potential and mechanical deformation is accurate for non-linear poro-mechanical processes \cite{FlorianFCP,Petersen2012,Prevost2013}. Nevertheless, due to the complexities associated with monolithic solvers for the full non-linear problem, industry standard remains to use so-called weakly, or iteratively, coupled approaches \cite{Settari_2001,Petersen2012}. These de-coupled approaches lead to significant numerical diffusion, which may also mask fundamental numerical incompatibilities between the spatial discretizations. 

As a response to this current status, the objective of this paper is to show how a simple linearization technique, i.e. the $L$-scheme (see \cite{List2016,RaduPopKnabner2004,RaduNPK15} for application of this method to Richards' equation or two-phase flow in porous media) can be combined with a splitting algorithm (known both as the undrained split or the fixed stress method \cite{Jakub_2016,Kim20112094,Kim20111591,RIS_0,Mikelic,Mikelic2,Settari}) to obtain a robust and efficient iterative scheme for solving a non-linear Biot model.

In this paper, we use for concreteness linear conformal Galerkin elements for the discretization in space of the mechanics equation and mixed finite elements for the flow equation \cite{GayX2,wheeler}. Precisely, the lowest order Raviart-Thomas elements are used \cite{brezzi2012mixed}. We expect, however, that the solution strategy discussed herein will be applicable to other combinations of spatial discretizations such as those discussed in \cite{Jan_2016,Carmen_2016} and the references therein. Backward Euler is used for the temporal discretization. Multirate time discretizations or higherorder space-time Galerkin method have been also proposed for the linear Biot model in \cite{Kundan_2016} and \cite{Uwe3}, respectively. We propose two new iterative methods for solving the resulting non-linear equations at the new time-level: a splitting algorithm based on a combination between the undrained split and the fixed stress methods (as mentioned above)  and a monolithic approach based on the same linearization technique. The existence and uniqueness of a solution for the both formulations, as well as their global, linear convergence are rigorously  shown.  To the best of our knowledge, these are the first rigorous convergence results for iterative (monolithic or splitting) schemes in the non-linear case. At the same time, we also acknowledge that while the non-linearities considered here are representative for compressible materials, they are simpler than those encountered in more complex applications such as multiphase flow. 

To summarise, the new contributions of this paper are
\begin{itemize}
\item[$\bullet$]{We propose splitting and a monolithic L-scheme for a non-linear Biot model.}
\item[$\bullet$]{The linear convergence of both schemes is rigorously shown in energy norms.}
\item[$\bullet$]{We provide a benchmark for the convergence of splitting algorithms for the non-linear Biot model, including a comprehensive comparison between the splitting and monolithic L-scheme.}
\end{itemize}


The paper is structured as follows. In the next section we present the mathematical model and the discretization, including the fully discrete computational schemes. In Section \ref{sec:convergence} we analyse rigorously the proposed schemes. Numerical results are presented in Section \ref{sec:numericalresults}, and we conclude the paper in Section \ref{sec:conclusions}.

\subsection{Notations}
In this paper we use typical notations from functional analysis.  Let $\Omega \subset \mathbb{R}^d$ be an open and bounded  domain with a Lipschitz continuous boundary $\partial \Omega$, with $d \in \{1, 2, 3\}$ being the dimension of the space. We denote by $L^2(\Omega)$ the space of square integrable functions and by $H^1(\Omega)$ the Sobolev space
$$H^1(\Omega) = \{v \in L^2(\Omega)\,;\, \nabla\,v \in L^2(\Omega)^d\}.$$

Furthermore,  $H^1_0(\Omega) $ will be the space of functions in $H^1(\Omega)$ vanishing on $\partial \Omega$ and $H({\rm div};\Omega)$  the space of vector valued function having all the components and the divergence in $L^2(\Omega)$. We will use bold face notation to specify when dealing with vectors. We denote by $ \langle \cdot , \cdot \rangle $ the inner product on $L^2(\Omega)$ and $ \norm{v} = \sqrt{\langle v,v \rangle} $ the associated norm. Let further $[0,T]$ be a time interval, with $T$  denoting the final computational time. The notations for the variables and parameters of Biot's model are summarized in Table \ref{table:nomenclature}.
 
 \begin{table}[t]
\centering
\caption{Nomenclature}\label{table:nomenclature}
\begin{tabular}{lrclr}
Parameters            &&&Variables &   \\
\hline
Lam\'e's first parameter & $\lambda$&&Displacement&$\uu$  \\
Lam\'e's second parameter & $\mu$ &&Pressure& $p$ \\
Dynamic fluid viscosity     & $\mu_f$&&Mass flux&$\qq$  \\
Kinematic fluid viscosity     & $\nu_f$&&& \\
Fluid density     & $\rho_f$&&&  \\
Source term & $S_f$&&& \\
Biot's constant     &$ \alpha$ &&&  \\
Biot's modulus & $M$&&& \\
Permeability scalar value           &$ k $ &&& \\
Gravity vector           &$ \vec{g} $ &&& \\
Body forces & $\vec{f}$&&& \\
Effective stress tensor & $\vec{\sigma}$&&& \\
Cauchy stress tensor & $\vec{\sigma}^{por}$&&& \\
Strain tensor & $\varepsilon$
\end{tabular}
\end{table}
  

\section{Mathematical model and discretization}
\label{sec:1}

We use Biot's consolidation model in the domain $\Omega \times [0,T]$ considering a non-linear elastic, homogeneous, isotropic, porous medium saturated with a compressible fluid.
The Cauchy stress tensor $\vec{\sigma}^{por}$ can be expressed in terms of the  fluid pressure $p$ and the displacement $\uu$ as
\begin{equation}
 \vec{\sigma}^{por}(\uu,p)= \vec{\sigma}(\uu)- \alpha p \vec{I},
 \label{sigma}
\end{equation}
where $\vec{I}$ is the identity tensor and $\alpha$ is the dimensionless Biot coefficient, see e.g. \cite{biot1941general,biot1954,detournay1993fundamentals,coussy1995mechanics},  and $\vec{\sigma}(\uu)$  the extended non-linear stress tensor, given by 
 \begin{equation}
\vec{\sigma}(\uu) =2 \mu \varepsilon (\uu) + h( \divv{\uu})\vec{I}.
\label{Consitutive}
\end{equation}
Above, $\mu>0$ is the constant shear modulus, $\varepsilon$ the strain (or symmetric gradient) tensor, i.e. $\varepsilon(\uu)=\frac{1}{2}\left( \nabla \uu+(\nabla \uu)^t\right)$. The non-linear term $h(\cdot)$ models the volumetric stress. Under quasi-static assumptions (neglecting the acceleration), the governing  equation for mechanical deformation of the solid-fluid system can be expressed as
\begin{equation}
 -\divv{\vec{\sigma}^{por}}=\vec{f},
 \label{Fuerza}
 \end{equation}
where $\vec{f} $ is a body force in $\Omega$. 
Substituting the constitutive relation \eqref{Consitutive} into \eqref{sigma} and expanding \eqref{Fuerza}, we get
 $$-\nabla \cdot [2 \mu \varepsilon (\uu) + h (\divv {\uu}) \vec{I}] + \alpha \divv p \vec{I}  =  \vec{f}.$$
The volumetric flux through porous medium $\Omega$ is modelled using Darcy's law
$$\qq_{v}= -\frac{K}{\mu_f} \left[ \nabla p -\rho_f \vec{g} \right],$$
where $\mu_f$, $\rho_f$ are the dynamic viscosity and density of the fluid respectively  and $\vec{g}$  is the gravity vector. 

\begin{remark}
For simplicity, we consider the permeability to be a scalar function, but the results of the paper can be extended without difficulties to the tensor case.
\end{remark}

The next equation is the mass balance for the fluid and reads as
\begin{equation}
 \frac{\partial \varphi}{\partial t}= -\divv \qq  + S_f,
 \label{varphi}
\end{equation}
where $\qq=\rho_f \qq_v$ is the mass flux, $S_f$ is a source term and $\varphi$ the mass of the fluid in the medium, which is proportional with the volume. 
 Further, $\varphi$ can be expressed in terms of the fluid pressure $p$ and $\divv u$
\begin{equation}
\varphi = b(p) + \alpha \divv \uu.
 \label{varphi2}
\end{equation}
This approach extends the classical Biot model, as e.g. in \cite{biot1941general,Kim20112094,Mikelic} by allowing for a more general equation-of-state, here given by the density as a non-linear function $b(\cdot)$. Putting together the equations \eqref{varphi}-\eqref{varphi2} we obtain
$$\frac{\partial}{\partial t} \left( b(p) + \alpha \nabla \cdot \uu \right)+ \divv \qq  = S_f.   $$

Finally, the non-linear Biot's model considered in this paper reads as (the variables and coefficients are summarized in Table \ref{table:nomenclature})
\begin{align}
    -\nabla \cdot [2 \mu \varepsilon (\uu) + h (\nabla \cdot \uu) ] +\alpha \nabla \cdot (p I) & =  \vec{f} & \text{ in } \Omega \times ]0,T[, \label{DT_displacement}\\
 \qq & = -\frac{K}{\nu_f} \left( \nabla p -\rho_f\vec{g} \right) & \text{ in } \Omega \times ]0,T[, \label{DT_flow} \\
  \partial_t \left( b(p) + \alpha \nabla \cdot \uu \right)+  \nabla \cdot \qq & = S_f  & \text{ in } \Omega \times ]0,T[.       \label{DT_pressure} 
\end{align} 
To complete the model we consider homogeneous Dirichlet boundary conditions (BC) and initial conditions given by $\vec{u} = \uu_0$ and $p  = p_0$ at time $t = 0$. The functions  $ \uu_0,\ p_0$ are supposed to be given (and to be sufficiently regular). 
\begin{remark} We consider here homogeneous Dirichlet BC just for the sake of simplicity. The analysis in Section \ref{sec:convergence} can be extended to more general BC, as considered also in the numerical examples in Section \ref{sec:numericalresults}.
\end{remark}

\begin{remark} {\bf Linear Biot's model.} The linear Biot model, as in e.g. \cite{Mikelic,Mikelic2}, is a particular case of the non-linear model \eqref{DT_displacement}-\eqref{DT_pressure} and it can be immediately obtained by taking $b(p) := \frac{p}{M}$  and $h(\divv \uu) := \lambda \divv \uu$, where $M$ is a compressibility constant and $\lambda$ the Lame parameter. 
\end{remark}
 



\noindent {\bf Fully implicit discretization of the Biot model  \eqref{DT_displacement}-\eqref{DT_pressure}.}

 For the discretization of the considered non-linear Biot model we use conformal Galerkin finite elements for the displacement variable and mixed finite elements for the flow unknowns \cite{GayX2,wheeler}.  Precisely, we use linear elements for the displacement and lowest order Raviart-Thomas elements \cite{brezzi2012mixed} for the flow. Backward Euler is used for the temporal discretization.

Let  $\Omega = \cup_{K \in \mathcal{K}_h} K$ be a regular decomposition of $\Omega$ into $d$-simplices. We denote by $h$ the mesh size. The discrete spaces are given by
\begin{align}
& \hspace{1cm}{\mathbf Z}_h = \{\mathbf{z}_h \in {H^1(\Omega)}^d\,;\,  {{\mathbf{z}}_h}_{|K} \in {{\mathbb{P}_1^d}} \,, \,   {\forall} K \in  \mathcal{K}_h\}, \nonumber\\
& \hspace{1cm}Q_h = \{w_h \in L^2(\Omega) \,;\, {w_h}_{|K} \in {{\mathbb{P}}_0} \, ,\,  {\forall} K \in  \mathcal{K}_h \}, \nonumber\\
& \hspace{1cm} {\mathbf{V}}_h = \{{\vec{v}}_h \in {H(\mathrm{div}; \Omega)}\,;\,  {\vec{v}}_{h|K} (\vec{x})=\vec{a}+b \vec{x}, \, \vec{a} \in  \mathbb{R}^d,\, b \in \mathbb{R} ,  \, {\forall} K \in  \mathcal{K}_h\},\nonumber
\end{align}
where ${\mathbb{P}}_0, {\mathbb{P}}_1$ denote the spaces of constant functions and of linear polynomials, respectively.

 For $N \in \mathbb{N}$, we discretize the time interval uniformly and define the time step $\tau = \frac{T}{N}$ and $t_n = n\tau$. We use the index $n$ for the primary variable $p^{n}$,  $\qq^{n}$, $\uu^{n}$ at corresponding time step $t_n$.

We can now formulate a fully discrete variational formulation for the non-linear Biot model \eqref{DT_displacement}-\eqref{DT_pressure}.

{\bf Problem $P_h^n$}. Given $\left(p_h^{n-1},\qq_h^{n-1},\uu_h^{n-1}\right)$, find $\left(p_h^n,\qq_h^n,\uu_h^n\right) \in W_h  \times \mathbf{V}_h \times \mathbf{Z}_h$ such that
\begin{align}
2\mu   \langle  \varepsilon (\uu_h^n): \varepsilon (\vec{z}_h)\rangle +  \langle  h(\divv{ \uu_h^n}),\divv{\vec{z}_h}\rangle -\alpha \langle  p_h^n , \divv{\vec{z}_h}\rangle& =\langle \vec{f} ,\vec{z}_h\rangle ,\label{VF_displacement}\\
  \nu_f \langle K^{-1} \qq_h^n , \vec{v}_h \rangle - \langle   p_h^n ,\divv{\vec{v}_h} \rangle& = \langle \rho_f \vec{g},\vec{v}_h  \rangle , \label{VF_flow} \\
   \langle b(p_h^n),w_h\rangle  + \alpha \langle \divv{\uu_h^n},w_h\rangle + \tau \langle \nabla \cdot \qq_h^n ,w_h\rangle&=\nonumber\\
 \tau \langle f  ,w_h\rangle+   \langle b(p_h^{n-1}),w_h\rangle +&\alpha \langle \divv{\uu_h^{n-1}} w_h\rangle,
        \label{VF_pressure} 
\end{align}
 for all  $\vec{z}_h\in {\mathbf Z}_h $, $\vec{v}_h \in {\mathbf V}_h$ and $w_h \in W_h$.
\begin{remark} A continuous variational formulation can be analogously given, see \cite{Mikelic} for the linear case. We will show in Section \ref{sec:convergence} the existence and uniqueness of a solution for the fully discrete variational scheme above, by using the Banach fixed point theorem. Existence and uniqueness for the continuous case can be shown similarly. Error estimates can be also obtained, following the lines of \cite{wheeler} in combination with the techniques used in this paper to deal with the nonlinearities. Nevertheless, this is beyond the scope of this paper. 
\end{remark}

\noindent {\bf Non-linear solvers: a splitting L-scheme and a monolithic L-scheme.}
  	 	
The non-linear system \eqref{VF_displacement}-\eqref{VF_pressure}  can be solved monolithically by using the Newton method or a robust, linear convergent linearization scheme (see e.g.\cite{RaduPopKnabner2004,RaduNPK15,List2016}) or by using a splitting algorithm \cite{Jakub_2016,RIS_0,Mikelic2}  In this work we present a combination between the undrained split and fixed stress methods, adapted to the non-linear case and a monolithic, fixed point linearization scheme.

We begin by presenting the splitting L-scheme. At each time step (we use $n$ to denote the time index), we first solve the flow equations using the displacement from the last iteration and then, with the new computed pressure, we solve the displacement equation and iterate until convergence is reached. We use $i$ for indexing the iterations. We start the iterations with the solution at the last time step (or the initial values for the first time step), i.e. $p_h^{n,0} = p_h^{n-1}$, $\qq_h^{n,0} = \qq_h^{n-1}$ and $\uu_h^{n,0} = \uu_h^{n-1}$.  We further introduce two positive constants, $L_1$ and $L_2$ which are free to be chosen in order to optimize the scheme. \\

\noindent {\bf A robust splitting L-scheme: the extension of the fixed stress algorithm to non-linear Biot.}
  	
\textbf{Step 1:}  Given  $\uu_h^{n,i} \in {\mathbf Z}_h $, find $p_h^{n,i+1}  \in W_h$ and $\qq_h^{n,i+1}\in  {\mathbf V}_h$ such that there holds for all $\vec{v}_h \in {\mathbf V}_h$ and $w_h \in W_h$
\begin{align}
 \nu_f\langle K^{-1} \qq_h^{n,i+1} , \vec{v}_h \rangle - \langle   p_h^{n,i+1} ,\divv{\vec{v}_h} \rangle& = \langle \rho_f \vec{g},\vec{v}_h  \rangle , \label{S_flow}
   \\
   \langle b(p_h^{n,i}),w_h\rangle  +  L_1  \langle p_h^{n,i+1}-p_h^{n,i},w_h\rangle &\nonumber\\
   + \alpha \langle \divv{\uu_h^{n,i}},w_h\rangle 
  + \tau \langle \nabla \cdot \qq_h^{n,i+1} ,w_h\rangle&=
   \tau \langle S_F  ,w_h\rangle+   \langle b(p_h^{n-1}),w_h\rangle \nonumber\\&+\alpha \langle \divv{\uu_h^{n-1}} w_h\rangle.
        \label{S_pressure} 
\end{align}

\textbf{Step 2:} Given now $p_h^{n,i+1} \in W_h $, find $\uu_h^{n,i+1} \in {\mathbf Z}_h $ such that there holds  for all $\vec{z}_h\in {\mathbf Z}_h $

\begin{align}
2\mu   \langle  \varepsilon (\uu_h^{n,i+1}): \varepsilon (\vec{z}_h)\rangle +
 L_2 \langle \divv\uu_h^{n,i+1}-\divv{\uu_h^{n,i}},\divv{ \vec{z}_h} \rangle  \nonumber\\
 +\langle  h(\divv{ \uu_h^{n,i}}),\divv{\vec{z}_h}\rangle -\alpha\langle p_h^{n,i+1} , \divv{\vec{z}_h}\rangle& =\langle \vec{f} ,\vec{z}_h\rangle. \label{S_displacement}
\end{align}

\begin{remark}
Ideally the constants $L_1$ and $L_2$ should be chosen as small as possible (in order to increase the convergence rate, as it will be shown below), but large enough to ensure the convergence of the scheme. This will be discussed in detail in Section \ref{sec:convergence} (from a theoretical point of view) and in Section \ref{sec:numericalresults} (for practical computations).
\end{remark}

We introduce now a monolithic L-scheme, called the  $L$-scheme as alternative to the splitting method proposed above. The scheme is inspired by the works \cite{RaduPopKnabner2004,RaduNPK15,List2016}, where a similar idea is applied for the Richards equation.

We again start the iterations with the solutions at the last time step: $p_h^{n,0} = p_h^{n-1}$, $\qq_h^{n,0} = \qq_h^{n-1}$ and $\uu_h^{n,0} = \uu_h^{n-1}$ (recall that $n$ denotes the time step index and $i$ is the iteration step). Let $L_1$ and $L_2$ be two positive constants.\\ 
 
\noindent {\bf A monolithic L-scheme.}
 
  Given  $\uu_h^{n,i} \in {\mathbf Z}_h $, $p_h^{n,i}  \in W_h$ and  $\qq_h^{n,i}\in  {\mathbf V}_h$, find  $\uu_h^{n,i+1} \in {\mathbf Z}_h $,  $p_h^{n,i+1}  \in W_h$ and $\qq_h^{n,i+1}\in  {\mathbf V}_h$ such that there holds for all $\vec{z}_h\in {\mathbf Z}_h $, $\vec{v}_h \in {\mathbf V}_h$ and $w_h \in W_h$ 
\begin{align}
2\mu   \langle  \varepsilon (\uu_h^{n,i+1}): \varepsilon (\vec{z}_h)\rangle +
 L_2 \langle \divv\uu_h^{n,i+1}-\divv{\uu_h^{n,i}},\divv{ \vec{z}_h} \rangle  \nonumber\\
 +\langle  h(\divv{ \uu_h^{n,i}}),\divv{\vec{z}_h}\rangle -\alpha\langle p_h^{n,i+1} , \divv{\vec{z}_h}\rangle& =\langle \vec{f} ,\vec{z}_h\rangle ,\label{LL_displacement}
  \end{align}
  and
 \begin{eqnarray}    
  \nu_f\langle K^{-1} \qq_h^{n,i+1} , \vec{v}_h \rangle - \langle   p_h^{n,i+1} ,\divv{\vec{v}_h} \rangle& = & \langle \rho_f \vec{g},\vec{v}_h  \rangle , \label{LL_flow} \\
 \langle b(p_h^{n,i}),w_h\rangle  +  L_1  \langle p_h^{n,i+1}-p_h^{n,i},w_h\rangle  &+& \alpha \langle \divv{\uu_h^{n,i+1}},w_h\rangle \nonumber  \\
    +  \tau \langle \nabla \cdot \qq_h^{n,i+1} ,w_h\rangle &= & \tau \langle S_F  ,w_h\rangle +   \langle b(p_h^{n-1}),w_h\rangle +\alpha \langle \divv{\uu_h^{n-1}} w_h\rangle,  \nonumber  \\  \label{LL_pressure} 
\end{eqnarray}	

The convergence of the proposed schemes will be studied theoretically in Section \ref{sec:convergence} and numerically in Section \ref{sec:numericalresults}.

\section{Convergence analysis}
\label{sec:convergence}
In this section we will show the convergence of the splitting L-scheme \eqref{S_flow} - \eqref{S_displacement} and of the monolithic L-scheme \eqref{LL_displacement} - \eqref{LL_pressure} for the non-linear Biot problem \eqref{VF_displacement} - \eqref{VF_pressure}. For this we will combine the techniques developed in \cite{RaduPopKnabner2004,RaduNPK15} with the ones in \cite{Mikelic2}.
In the following we will use the algebraic identity
\begin{eqnarray}
 \langle x-y,x\rangle &=  &\frac{\norm{x}^2}{2}+\frac{\norm{x-y}^2}{2}-\frac{\norm{y}^2}{2}\label{algebraic_ident_2}
\end{eqnarray}
and Young's inequality 
$$ \abs{ab}\leq \frac{a^2}{2\delta} + \frac{\delta b^2}{2}, \ \forall \delta >0.$$ 

We will also use the next lemma in the theorems below. The proof can be found e.g. in \cite{Thomas}.
\begin{lemma}
\label{thomas}
Given a $w_h \in W_h$ there exists $\vec{v}_h \in {\mathbf V}_h$ satisfying
$$ \divv {\vec{v}_h} = w_h \ \ \text{and} \ \ \norm{\vec{v}_h} \leq C_{\Omega,d}\norm{w_h},$$
with $C_{\Omega,d} > 0 $ not depending on $w_h$ or mesh size.
 \end{lemma}


Throughout this section we assume that the following assumptions hold true.
\begin{enumerate}
\item[(A1) ]  $ b(\cdot):\mathbb{R}\rightarrow \mathbb{R}$ is $C^1$ (i.e. derivable, having a continuous derivative), strictly increasing 
and Lipschitz continuous, i.e there exist $b_m>0$ and $L_b$  such that 
 $  b_m \leq b'(\cdot) \leq L_b < +\infty$.

\item[(A2) ]  $ h(\cdot):\mathbb{R}\rightarrow \mathbb{R}$  is $C^1$, strictly increasing 
and Lipschitz  continuous, i.e. there  exist $h_m>0$ and $L_h$  such that
 $  h_m \leq h'(\cdot) \leq L_h < +\infty$.
 
 \item[(A3) ]$K:\mathbb{R}^d\rightarrow \mathbb{R} $  is assumed to be  constant in time and bounded, i.e. there exist  $k_{m}>0$ and $k_{M}$, such that  
$ k_{m}  \leq K(\vec{x})\leq  k_{M},$ $\forall \vec{x}\in \Omega.$
\end{enumerate}

\begin{remark}
The assumptions (A1)-(A2) are obviously satisfied in the linear case, where $b^\prime =\dfrac{1}{M} $ and $h^\prime = \lambda$. 
\end{remark}

\begin{remark}
The Lipschitz continuity and monotonicity of $b(\cdot), h(\cdot) $ and  $b_m, h_m  > 0$ are essential for the proof of the convergence for the splitting L-scheme. In the case of the monolithic L-scheme, one can relax the latter assumption: $b_m, h_m   \ge  0$ is enough to ensure the convergence of the $L$-scheme.
\end{remark}

\noindent{\bf Convergence of the splitting L-scheme.}

Let us denote by $ e_p^i =p^{n,i}-p^n$, $ \vec{e}_{\qq}^i =\qq^{n,i}-\qq^n$ and $ \vec{e}_{\uu}^i =\uu^{n,i}-\uu^n$ the errors at iteration $i$, where ($p^n$, $\qq^n$, $\uu^n$) is the solution of the non-linear problem  \eqref{VF_displacement}-\eqref{VF_pressure} and ($p^{n,i}$, $\qq^{n,i}$, $\uu^{n,i}$) is the solution of  \eqref{S_flow}- \eqref{S_displacement}. We assume here that ($p^n$, $\qq^n$, $\uu^n$) exists, this being rigorously proved later in Theorem \ref{theorem_L_scheme}.

The next result is showing the convergence of the splitting algorithm.
\begin{theorem}
Assuming that (A1)-(A3) hold true and that $L_1 \geq L_b$ and $L_2 \geq L_h + \frac{\alpha^2}{ b_m}$, the splitting algorithm \eqref{S_flow}-\eqref{S_displacement} is linearly convergent. There holds
 \begin{equation}
\begin{aligned}
        \left( L_1 - \frac{b_m}{2}+ \frac{\tau k_m^2}{  \nu_f C_{\Omega,d }^2 k_M}\right) \norm{e_p^{i+1}}^2 + \frac{\tau   \nu_f }{ k_M} \norm{\vec{e}^{i+1}_{\qq}}^2
        + L_2 \norm{ \nabla \cdot \vec{e}^{i+1}_{\uu}}^2 \\
      \leq 
      (L_1-b_m) \norm{e_p^{i}}^2 + (L_2-h_m) \norm{ \nabla \cdot \vec{e}^{i}_{\uu}}^2.
\label{S_MainResult}
 \end{aligned}
  \end{equation}
  \end{theorem}

\begin{remark}
The inequality above ensure the convergences  $e_p^i  \rightarrow 0$,  $\vec{e}^{i}_{\qq} \rightarrow 0$ and $ \nabla \cdot \vec{e}^{i}_{\uu} \rightarrow 0$. The convergence $ \nabla \cdot \vec{e}^{i}_{\uu} \rightarrow 0$ follows from \eqref{MassStructural_Proof_6} below.
\end{remark}

\begin{proof}

We start by subtracting \eqref{VF_displacement} - \eqref{VF_pressure} from \eqref{S_displacement} - \eqref{S_flow}, respectively to obtain for all $w_h \in W_h, \vec{v}_h \in Z_h, v_h \in V_h$
\begin{equation}
\begin{aligned}
         \langle b(p_h^{n,i})- b(p_h^{n}),w_h\rangle  + L_1  \langle e_p^{i+1}-e_p^{i},w_h\rangle + \tau \langle \nabla \cdot \vec{e}^{i+1}_{\qq} ,w_h\rangle &\\
         =-\alpha \langle \nabla \cdot \vec{e}^{i}_{\uu} ,w_h\rangle,& \label{MassConservacions_Proof}
       \end{aligned}
\end{equation}

\begin{equation}
 \begin{aligned}
    \nu_f \langle K^{-1}  \vec{e}^{i+1}_{\qq} ,\vec{v}_h\rangle-  \langle     e_p^{i+1},\divv{\vec{v}_h}\rangle& =0 , \label{Darcy_Proof}
  \end{aligned}
\end{equation}

\begin{equation}
\begin{aligned}
       2\mu \langle  \varepsilon(\vec{e}^{i+1}_{\uu}),\varepsilon(\vec{z}_h)\rangle + 
        \langle    h (\nabla \cdot \uu_h^{n,i})   - h (\nabla \cdot \uu_h^n)  ,\nabla \cdot\vec{z}_h\rangle &\\ 
     +L_2 \langle \nabla \cdot (\vec{e}^{i+1}_{\uu}-\vec{e}^{i}_{\uu}),\nabla \cdot \vec{z}_h \rangle 
      =\alpha \langle (e_p^{i+1}), \nabla \cdot \vec{z}_h\rangle.&\label{Structural_Proof}
 \end{aligned}
\end{equation}

The estimates are obtained in a stepwise manner. Accordingly, we handle the flow equations \eqref{MassConservacions_Proof}-\eqref{Darcy_Proof} and displacement equation \eqref{Structural_Proof} in Step 1 and Step 2, respectively. Then, the obtained estimates will be combined in Step 3 to show the result \eqref{S_MainResult}.\\

\noindent{\it Step 1: Flow equations}\\
We first choose  $w_h =e_p^{i+1} \in W_h $ in (\ref{MassConservacions_Proof}) and  $\vec{v}_h=\tau  \vec{e}^{i+1}_{\qq} \in V_h$ in (\ref{Darcy_Proof}),  then add the results  to obtain
\begin{equation}
\begin{aligned}
         \langle b(p_h^{n,i})- b(p_h^{n}),e_p^{i+1}\rangle  + L_1  \langle e_p^{i+1}-e_p^{i},e_p^{i+1}\rangle + \tau   \nu_f \langle K^{-1}  \vec{e}^{i+1}_{\qq} , \vec{e}^{i+1}_{\qq}\rangle &\\
         =-\alpha \langle \nabla \cdot \vec{e}^{i}_{\uu} ,e_p^{i+1}\rangle & . \label{MassConservacions_Proof2}
       \end{aligned}
\end{equation}

By some algebraic manipulations, using \eqref{algebraic_ident_2} and Cauchy-Schwarz and  Young inequalities we get from \eqref{MassConservacions_Proof2}
\begin{equation}
\begin{array}{l}
         \frac{L_1}{2}\norm{e_p^{i+1}}^2+\frac{L_1}{2}\norm{e_p^{i+1}-e_p^i}^2+\langle b(p_h^{n,i})- b(p_h^n),e_p^{i}\rangle+ \tau   \nu_f \langle K^{-1}  \vec{e}^{i+1}_{\qq} , \vec{e}^{i+1}_{\qq}\rangle \\[2ex]
       \hspace*{1cm} =\frac{L_1}{2}\norm{e_p^{i}}^2 +\langle b(p_h^{n,i})- b(p_h^n),e_p^{i}-e_p^{i+1}\rangle-\alpha \langle \nabla \cdot \vec{e}^{i}_{\uu} ,e_p^{i+1}\rangle \\[2ex]
          \hspace*{1cm}  \leq     \frac{L_1}{2}\norm{e_p^{i}}^2+ \frac{\delta_1 \norm{b(p^{n,i}) -b(p^n)}^2}{2} + \frac{1}{2\delta_1}\norm{e_p^i -e_p^{i+1}}^2-\alpha\langle \nabla \cdot \vec{e}^{i}_{\uu} ,e_p^{i+1}\rangle,\nonumber
       \end{array}
\end{equation}
 for any $\delta_1>0$. Using (A3), we obtain from the above equation
\begin{equation}
\begin{array}{l}
         \frac{L_1}{2}\norm{e_p^{i+1}}^2+\left(\frac{L_1}{2}-\frac{1}{2\delta_1}\right) \norm{e_p^{i+1}-e_p^i}^2+\langle b(p_h^{n,i})- b(p_h^n),e_p^{i}\rangle\\
      \hspace*{0.5cm}  +  \dfrac{\tau \nu_f} {k_M^{-1} }\norm{\vec{e}^{i+1}_{\qq}}^2  \leq \frac{L_1}{2}\norm{e_p^{i}}^2 + \frac{\delta_1}{2} \norm{b(p_h^{n,i})
         -b(p_h^n)}^2  -\alpha\langle \nabla \cdot \vec{e}^{i}_{\uu} ,e_p^{i+1}\rangle.
       \end{array}
        \label{MassConservacions_Proof_3}
\end{equation}
Furthermore, using now (A1), i.e. the monotonicity and the Lipschitz continuity of  $b(\cdot)$, we get from \eqref{MassConservacions_Proof_3}
\begin{equation}
\begin{array}{l}
         \frac{L_1}{2}\norm{e_p^{i+1}}^2 +  \left(\frac{L_1}{2}-\frac{1}{2\delta_1}\right) \norm{e_p^{i+1}-e_p^i}^2+\frac{b_m}{2} \norm{e_p^{i}}^2 +  \dfrac{\tau \nu_f} {k_M }\norm{\vec{e}^{i+1}_{\qq}}^2\\ 
        \hspace*{1cm} +\left(\frac{1}{2L_b} -  \frac{\delta_1}{2}\right) \norm{b(p_h^{n,i}) -b(p_h^n)}^2   \leq \frac{L_1}{2}\norm{e_p^{i}}^2  -\alpha\langle \nabla \cdot \vec{e}^{i}_{\uu} ,e_p^{i+1}\rangle.
       \end{array}
        \label{MassConservacions_Proof4}
\end{equation}

\noindent{\it Step 2: Displacement Equation}\\
Testing \eqref{Structural_Proof} with $\vec{z}_h=\vec{e}^{i+1}_{\uu} \in {\mathbf V}_h $ and using \eqref{algebraic_ident_2} we obtain
\begin{equation} \label{eq_error_displ}
\begin{array}{l}
        2\mu \langle   \varepsilon(\vec{e}^{i+1}_{\uu}):\varepsilon(\vec{e}^{i+1}_{\uu})\rangle +          \langle   h(\nabla \cdot \uu_h^{n,i}) - h (\nabla \cdot \uu_h^n)  ,\nabla \cdot\vec{e}^{i+1}_{\uu}\rangle +\frac{L_2}{2} \norm{ \nabla \cdot \vec{e}^{i+1}_{\uu}}^2\\ [2ex]
     \hspace*{1cm}+ \frac{L_2}{2}\norm{\nabla \cdot (\vec{e}^{i+1}_{\uu}-\vec{e}^{i}_{\uu})}^2   =\frac{L_2}{2} \norm{ \nabla \cdot \vec{e}^{i}_{\uu}}^2+\alpha \langle e_p^{i+1}, \nabla \cdot\vec{e}^{i+1}_{\uu}\rangle. 
 \end{array}
\end{equation}
Proceeding as in the {\it Step 1} above, by some algebraic manipulations, using Cauchy-Schwarz and Young inequalities and assumption (A2) we obtain from \eqref{eq_error_displ}
\begin{equation}
\begin{aligned}
       2\mu \langle   \varepsilon(\vec{e}^{i+1}_{\uu}):\varepsilon(\vec{e}^{i+1}_{\uu})\rangle +\left(\frac{1}{2L_h}-\frac{\delta_2}{2}\right)\norm{h (\nabla \cdot \uu_h^{n,i}) - h (\nabla \cdot \uu_h^{n})}^2
        &\\
        +\frac{L_2}{2} \norm{ \nabla \cdot \vec{e}^{i+1}_{\uu}}^2  
       + \left(\frac{L_2}{2}-\frac{1}{2\delta_2}\right)\norm{\nabla \cdot (\vec{e}^{i+1}_{\uu}-\vec{e}^{i}_{\uu})}^2 &\\
      \leq 
      \frac{L_2-h_m}{2} \norm{ \nabla \cdot \vec{e}^{i}_{\uu}}^2   
+\alpha \langle e_p^{i+1}, \nabla \cdot\vec{e}^{i+1}_{\uu}\rangle.&  
\label{Structural_Proof8}
 \end{aligned}
\end{equation}
for any $\delta_2>0$.\\

\noindent{\it Step 3: Combining flow and displacement}\\
Adding \eqref{MassConservacions_Proof4} and  \eqref{Structural_Proof8}  we obtain
\vspace*{-1ex}
\begin{equation}
\begin{array}{l}
         \frac{L_1}{2} \norm{e_p^{i+1}}^2 + \left(\frac{L_1}{2}-\frac{1}{2\delta_1}\right) \norm{e_p^{i+1}-e_p^i}^2 + \dfrac{\tau   \nu_f }{ k_M} \norm{\vec{e}^{i+1}_{\qq}}^2   \\ 
     \hspace*{0.5cm}    +\left(\frac{1}{2L_b}-\frac{\delta_1}{2}\right) \norm{b(p_h^{n,i}) -b(p_h^n)}^2    + 2\mu \langle   \varepsilon(\vec{e}^{i+1}_{\uu}):\varepsilon(\vec{e}^{i+1}_{\uu})\rangle +\frac{L_2}{2} \norm{ \nabla \cdot \vec{e}^{i+1}_{\uu}}^2 \\ 
      \hspace*{0.5cm}     +\left(\frac{1}{2L_h}-\frac{\delta_2}{2}\right)\norm{h (\nabla \cdot \uu_h^{n,i}) - h (\nabla \cdot \uu_h^{n})}^2  + \left(\frac{L_2}{2}-\frac{1}{2\delta_2}\right)\norm{\nabla \cdot (\vec{e}^{i+1}_{\uu}-\vec{e}^{i}_{\uu})}^2 \\
  \hspace*{1cm}      \leq \frac{L_1-b_m}{2}\norm{e_p^{i}}^2+ \frac{L_2-h_m}{2} \norm{ \nabla \cdot \vec{e}^{i}_{\uu}}^2 
+\underbrace{\alpha \langle \nabla \cdot(\vec{e}^{i+1}_{\uu}-\vec{e}^{i}_{\uu}),e_p^{i+1} \rangle}_{T}   .
\label{MassStructural_Proof_45}
 \end{array}
 \end{equation}
  We denoted the last term on the right hand side by $T$, which can be estimated separately by using  Cauchy Schwarz  and Young inequalities. We get
\begin{equation}
\begin{aligned}
T  
\leq \frac{\alpha^2}{2\delta_3}\norm{\nabla \cdot(\vec{e}^{i+1}_{\uu}-\vec{e}^{i}_{\uu})}^2+\frac{\delta_3}{2}\norm{e_p^{i+1}}^2\nonumber
 \end{aligned}
 \end{equation}
for any $\delta_3>0$. Then,  \eqref{MassStructural_Proof_45} takes the form
\begin{equation}
\begin{array}{l}
         (\frac{L_1}{2} - \frac{\delta_3}{2})\norm{e_p^{i+1}}^2 + \left(\frac{L_1}{2}-\frac{1}{2\delta_1}\right) \norm{e_p^{i+1}-e_p^i}^2 + \dfrac{\tau   \nu_f }{ k_M} \norm{\vec{e}^{i+1}_{\qq}}^2   \\ 
     \hspace*{0.5cm}    +\left(\frac{1}{2L_b}-\frac{\delta_1}{2}\right) \norm{b(p_h^{n,i}) -b(p_h^n)}^2    + 2\mu \langle   \varepsilon(\vec{e}^{i+1}_{\uu}):\varepsilon(\vec{e}^{i+1}_{\uu})\rangle +\frac{L_2}{2} \norm{ \nabla \cdot \vec{e}^{i+1}_{\uu}}^2 \\ 
      \hspace*{0.5cm}     +\left(\frac{1}{2L_h}-\frac{\delta_2}{2}\right)\norm{h (\nabla \cdot \uu_h^{n,i}) - h (\nabla \cdot \uu_h^{n})}^2  + \left(\frac{L_2}{2}-\frac{1}{2\delta_2} -\frac{\alpha^2}{2\delta_3}\right)\norm{\nabla \cdot (\vec{e}^{i+1}_{\uu}-\vec{e}^{i}_{\uu})}^2 \\
  \hspace*{1cm}      \leq \frac{L_1-b_m}{2}\norm{e_p^{i}}^2+ \frac{L_2-h_m}{2} \norm{ \nabla \cdot \vec{e}^{i}_{\uu}}^2.
\label{MassStructural_Proof_5}
 \end{array}
 \end{equation}

Due to Lemma \ref{thomas}, there exists a $\vec{v}_h \in {\mathbf V}_h$ such that  $\divv \vec{v}_h =e_p^{i+1}$ and $\norm{\vec{v}_h} \leq C_{\Omega,d}\norm{e_p^{i+1}}$. Testing \eqref{Darcy_Proof} with this  $\vec{v}_h$, and using Cauchy-Schwarz's inequality we obtain
\begin{equation}
 \begin{aligned}
\norm{e_p^{i+1}} \leq    
    C_{\Omega,d} \frac{\nu_f}{k_m} \norm{  \vec{e}^{i+1}_{\qq}}. \label{thomasproof4}
  \end{aligned}
\end{equation}
Using now \eqref{thomasproof4} in \eqref{MassStructural_Proof_5} further gives
\begin{equation}
\begin{array}{l}
         \left( \frac{L_1}{2}+  \frac{\tau k_m^2}{ 2  \nu_f C_{\Omega,d }^2 k_M} -\frac{\delta_3}{2}\right) \norm{e_p^{i+1}}^2 + 
         \left(\frac{L_1}{2}-\frac{1}{2\delta_1}\right) \norm{e_p^{i+1}-e_p^i}^2
          + \frac{\tau   \nu_f}{2k_M} \norm{\vec{e}^{i+1}_{\qq}}^2    \\ 
     \hspace*{0.5cm}    +\left(\frac{1}{2L_b}-\frac{\delta_1}{2}\right) \norm{b(p^{n,i}) -b(p^n)}^2    + 2\mu \langle   \varepsilon(\vec{e}^{i+1}_{\uu}):\varepsilon(\vec{e}^{i+1}_{\uu})\rangle +\frac{L_2}{2} \norm{ \nabla \cdot \vec{e}^{i+1}_{\uu}}^2 \\ 
      \hspace*{0.5cm}     +\left(\frac{1}{2L_h}-\frac{\delta_2}{2}\right)\norm{h (\nabla \cdot \uu^{n,i}) - h (\nabla \cdot \uu^{n})}^2  + \left(\frac{L_2}{2}-\frac{1}{2\delta_2} -\frac{\alpha^2}{2\delta_3}\right)\norm{\nabla \cdot (\vec{e}^{i+1}_{\uu}-\vec{e}^{i}_{\uu})}^2 \\
  \hspace*{1cm}      \leq \frac{L_1-b_m}{2}\norm{e_p^{i}}^2+ \frac{L_2-h_m}{2} \norm{ \nabla \cdot \vec{e}^{i}_{\uu}}^2.
\label{MassStructural_Proof_6}
 \end{array}
 \end{equation}
Finally, choosing $\delta_1 = \frac{1}{L_b}$, $\delta_2 = \frac{1}{L_h}$, $\delta_3 = b_m$ and assuming that there holds  $L_1 \geq L_b$ and $L_2 \geq L_h + \frac{\alpha^2}{b_m}$ we obtain from \eqref{MassStructural_Proof_6}
\begin{displaymath}
\begin{array}{l}
  \ds      \left( \frac{L_1}{2}-\frac{b_m}{2}+  \frac{\tau k_m^2}{ 2  \nu_f C_{\Omega,d }^2 k_M}\right)  \norm{e_p^{i+1}}^2 + 
        \frac{\tau   \nu_f }{2k_M} \norm{\vec{e}^{i+1}_{\qq}}^2 +\frac{L_2}{2} \norm{ \nabla \cdot \vec{e}^{i+1}_{\uu}}^2 \\[2ex]
 \ds  \hspace*{2cm}     + 2\mu \langle   \varepsilon(\vec{e}^{i+1}_{\uu}):\varepsilon(\vec{e}^{i+1}_{\uu})\rangle \leq 
      \frac{L_1-b_m}{2}\norm{e_p^{i}}^2+
      \frac{L_2-h_m}{2} \norm{ \nabla \cdot \vec{e}^{i}_{\uu}}^2. 
 \end{array}
 \end{displaymath}
 
 The above result gives immediately \eqref{S_MainResult}.
  \end{proof}
\begin{remark}
In the linear case, i.e. $b^\prime =\dfrac{1}{M} $ and $h^\prime = \lambda$, the undrained split scheme \cite{RIS_0,Mikelic2} is  obtained by taking $L_1 = \dfrac{1}{M}$ and $ L_2 = \lambda + M \alpha^2$. The convergence result is the same then as the one obtained in  \cite{Mikelic2}, but now in energy norms. Nevertheless, the optimal convergence would be obtained for $ L_2 = \lambda + \dfrac{M \alpha^2}{2}$, which can be shown e.g. by using the techniques in \cite{Jakub_2016} (for the linear case).
\end{remark}

\begin{remark}
The convergence rate is actually better, due to $\langle \varepsilon(\vec{e}^{i+1}_{\uu}):\varepsilon(\vec{e}^{i+1}_{\uu})\rangle \geq \frac{1}{d} \norm{ \nabla \cdot \vec{e}^{i+1}_{\uu}}^2$ which furnishes the term $(L_2+2\mu/d ) \norm{ \nabla \cdot \vec{e}^{i+1}_{\uu}}^2$ on the left hand side of the inequality \eqref{S_MainResult}.
\end{remark}

%

\noindent{\bf Convergence of the monolithic L-scheme and existence and uniqueness of the non-linear variational formulation Problem $P_h^n$ \eqref{VF_displacement}--\eqref{VF_pressure}.}

We prove now also the convergence of the monolithic L-scheme \eqref{LL_displacement} - \eqref{LL_pressure}. The idea is to show that the scheme is a contraction and apply the Banach fixed point theorem. In particular, we obtain by this also the existence and uniqueness of the original, non-linear problem \eqref{VF_displacement}-\eqref{VF_pressure}.

We define now $e_p^i =p^{n,i}-p^{n, i-1}$, $ \vec{e}_{\qq}^i =\qq^{n,i}-\qq^{n, i-1}$ and $ \vec{e}_{\uu}^i =\uu^{n,i}-\uu^{n,i-1}$ the differences between the solutions at iteration $i$ and $i-1$ of problem  \eqref{LL_displacement} - \eqref{LL_pressure}, respectively. Please remark the different definition compared to the proof of the convergence of the splitting algorithm. In that case the existence of a solution of the non-linear problem \eqref{VF_displacement}-\eqref{VF_pressure}  was assumed, in this case it will be proved.

\begin{theorem} \label{theorem_L_scheme} Assuming that (A1)-(A3) hold true and that $L_1 \ge  \dfrac{L_b}{2}$ and $L_2 \ge L_h$, the fixed point scheme \eqref{LL_displacement} - \eqref{LL_pressure} is a contraction satisfying
\begin{equation} \label{main_result_L}
\begin{array}{l}
 (L_1  +   \dfrac{\tau k_m^2}{  \nu_f k_M C_{\Omega,d}^2})   \norm{e_p^{i+1}}^2    + \dfrac{\tau   \nu_f }{ k_M} \norm{\vec{e}^{i+1}_{\qq}}^2 +  4 \mu   \langle \varepsilon (\vec{e}_{\uu}^{i+1}):  \langle \varepsilon (\vec{e}_{\uu}^{i+1}) \rangle \\
  \hspace*{0.5cm} + L_2 \norm{\divv\vec{e}_{\uu}^{i+1}}^2 \leq  L_1 \norm{e_p^{i}}^2 +  (L_2 - h_m) \norm{\divv\vec{e}_{\uu}^{i}}^2.
  \end{array}
\end{equation}

The limit is then the unique solution of \eqref{VF_displacement}-\eqref{VF_pressure}.
\end{theorem}

\begin{proof} We begin by writing the equations for $e_p^i, \vec{e}_{\qq}^i, \vec{e}_{\uu}^i$. By subtracting equations \eqref{LL_displacement} - \eqref{LL_pressure} at $i$ from the ones at $i+1$ we get for all $\vec{z}_h\in {\mathbf Z}_h $, $\vec{v}_h \in {\mathbf V}_h$ and $w_h \in W_h$ 
\begin{equation} \label{proof_L_1}
\begin{array}{l}
2\mu   \langle  \varepsilon (\vec{e}_{\uu}^{i+1}): \varepsilon (\vec{z}_h)\rangle +
 L_2 \langle \divv\vec{e}_{\uu}^{i+1}-\divv{\vec{e}_{\uu}^{i}},\divv{ \vec{z}_h} \rangle \\
\hspace*{2cm} +\langle  h(\divv{ \uu_h^{n,i}})  - h(\divv{ \uu_h^{n,i-1}}),\divv{\vec{z}_h}\rangle = \alpha\langle e_p^{i+1} , \divv{\vec{z}_h}\rangle,
  \end{array}
  \end{equation}
  and
 \begin{eqnarray}  
  \nu_f\langle K^{-1} \vec{e}_{\qq}^{i+1} , \vec{v}_h \rangle - \langle   e_p^{n,i+1} ,\divv{\vec{v}_h} \rangle& = & 0, \label{proof_L_2} \\
 \langle b(p_h^{n,i})   -  b(p_h^{n,i-1}), w_h\rangle  +  L_1  \langle e_p^{i+1}-e_p^{i},w_h\rangle  &+ &\alpha \langle \divv{\vec{e}_{\uu}^{i+1}}, w_h\rangle \nonumber  \\
    +  \tau \langle \nabla \cdot \vec{e}_{\qq}^{i+1} ,w_h\rangle &= & 0.    \label{proof_L_3} 
\end{eqnarray}	
Testing now \eqref{proof_L_2}   with  $\vec{v}_h = \tau \vec{e}_{\qq}^{i+1} \in  {\mathbf V}_h$  and \eqref{proof_L_3} with $w_h =  e_p^{i+1} \in W_h$,  adding the results and using the identity \eqref{algebraic_ident_2}  together with some algebraic manipulations we obtain
\begin{equation} \label{proof_L_4}
\begin{array}{l}
 \langle b(p_h^{n,i})   -  b(p_h^{n,i-1}), e_p^{i} \rangle  +  \dfrac{L_1}{2}  \norm{e_p^{i+1}}^2  +  \dfrac{L_1}{2}  \norm{e_p^{i+1}-e_p^{i}}^2 +  \tau   \nu_f   \langle K^{-1} \vec{e}^{i+1}_{\qq}, \vec{e}^{i+1}_{\qq}\rangle  \\
  \hspace*{1cm} =   \dfrac{L_1}{2}  \norm{e_p^{i}}^2 + \langle b(p_h^{n,i})   -  b(p_h^{n,i-1}), e_p^{i}  - e_p^{i+1} \rangle  - \alpha  \langle \divv{\vec{e}_{\uu}^{i+1}}, e_p^{i+1}\rangle.
  \end{array}
  \end{equation}
We proceed by testing \eqref{proof_L_1} with $ \vec{z}_h  = \vec{e}_{\uu}^{i+1} \in {\mathbf Z}_h $,  and using again \eqref{algebraic_ident_2}  to get
\begin{equation} \label{proof_L_5}
\begin{array}{l}
2\mu   \langle \varepsilon (\vec{e}_{\uu}^{i+1}):   \varepsilon (\vec{e}_{\uu}^{i+1}) \rangle + \dfrac{L_2}{2} \norm{\divv\vec{e}_{\uu}^{i+1}}^2  + \dfrac{L_2}{2} \norm{\divv\vec{e}_{\uu}^{i+1} - \divv{\vec{e}_{\uu}^{i}}}^2    \\
\hspace*{0.5cm} +\langle  h(\divv{ \uu_h^{n,i}})  - h(\divv{ \uu_h^{n,i-1}}), \divv \vec{e}_{\uu}^{i} \rangle = \dfrac{L_2}{2} \norm{\divv\vec{e}_{\uu}^{i}}^2 \\
 \hspace*{1.5cm} +  \langle  h(\divv{ \uu_h^{n,i}})  - h(\divv{ \uu_h^{n,i-1}}),  \divv(\vec{e}_{\uu}^{i} -  \vec{e}_{\uu}^{i+1}) \rangle+ \alpha\langle e_p^{i+1} ,  \divv\vec{e}_{\uu}^{i+1}\rangle.
  \end{array}
  \end{equation}
We add now \eqref{proof_L_4} and \eqref{proof_L_5}, use (A1) - (A3), the Cauchy-Schwarz and Young inequalities and Lemma \ref{thomas} in a similar manner as in the proof of the convergence for the splitting algorithm to finally obtain (after a multiplication by 2)
\begin{equation} \label{proof_L_6}
\begin{array}{l}
 (L_1  +   \dfrac{\tau k_m^2}{  \nu_f k_M C_{\Omega,d}^2})   \norm{e_p^{i+1}}^2    + \dfrac{\tau   \nu_f }{ k_M} \norm{\vec{e}^{i+1}_{\qq}}^2 + 4 \mu   \langle \varepsilon (\vec{e}_{\uu}^{i+1}):  \langle \varepsilon (\vec{e}_{\uu}^{i+1}) \rangle \\
  \hspace*{0.5cm} + L_2 \norm{\divv\vec{e}_{\uu}^{i+1}}^2 + (L_1 - \dfrac{L_b}{2})  \norm{e_p^{i+1}-e_p^{i}}^2 + (L_2 - L_h)  \norm{\divv\vec{e}_{\uu}^{i+1} - \divv{\vec{e}_{\uu}^{i}}}^2 \\
   \hspace*{2cm} \leq  L_1 \norm{e_p^{i}}^2 +  (L_2 - h_m) \norm{\divv\vec{e}_{\uu}^{i}}^2.
  \end{array}
  \end{equation}
The above result gives immediately \eqref{main_result_L}, implying that the considered fixed point scheme is a contraction. The rest follows by applying Banach fixed point theorem.

\end{proof}

\begin{remark}  Theorem \ref{theorem_L_scheme} holds also for $b(\cdot)$ increasing, i.e. $b_m \ge 0$, not necessarily strictly increasing (see assumption (A1)). It implies that the monolithic L-scheme converges also for an incompressible fluid $b = 0$, being therefore more robust than the splitting L-scheme. We point out that the monolithic L-schemes converges also for $h_m = 0$ if $ \mu > 0$.
\end{remark}

 \section{Numerical results}
\label{sec:numericalresults}

In this section, we present numerical experiments with the purpose of illustrating the performance of the iterative schemes proposed.  We propose two main test problems: an academic problem with a manufactured analytical solution, and a non-linear extension of Mandel's problem. All numerical experiments were implemented using the open-source finite element library Deal II \cite{DealII}.\\
\\
\\
\noindent {\bf Test problem 1: an academic example with a manufactured solution}\\

We solve the non-linear Biot problem in the unit-square  $\Omega=(0,1)^2$ and until final time $T = 1$, with a manufactured right hand side ($S_f$ and $\vec{f}$)  such that the problem admits the following analytical  solution
\begin{align}
p(x,y,t)&=tx(1-x)y(1-y),\nonumber\\
\qq(x,y,t)&=-K\nabla p, \nonumber        \\
u_1(x,y,t)=u_2(x,y,t)&=tx(1-x)y(1-y), \nonumber
\end{align} 
which has homogeneous boundary values for $p$ and $\uu$. We consider $K=\nu_f = M =\alpha=\lambda = \mu =1.0$.  For this case, all initial conditions are $0$. For all cases, we use as convergence criteria for the schemes 
 $ \norm{e_p^{i}} +\norm{\vec{e}^{i}_{\qq}} +\norm{\vec{e}_{\uu}^{i}} \leq 10^{-8}$.


In order to study the performance of the considered schemes, with a special focus on the splitting algorithm, we propose three coefficient functions for  $b(\cdot) $ and two for $h(\cdot)$, and define five test cases as given in Table 2. 
 The Lipschitz constants $L_b$, $L_h$ depend on the pressure and the divergence of displacements, respectively.  Unfortunately, for realistic problems one does not have the exact values of $L_b$ and $L_h$. Hence, it is necessary to determine how sensitive is the convergence of the proposed numerical schemes with respect to the tuning parameters $L_1$ and $L_2$. For each case, we investigated a range of values for $L_1$ and $L_2$ to assess the sensitivity of the performance of both monolithic and splitting L-schemes with respect to these parameters.

\begin{table}[h]
 \begin{center}
\caption{The coefficient functions $b(\cdot), h(\cdot)$ for test problem 1.} 
 \label{Nonlinearfuntions}
 \begin{tabular}{lc|c}
\hline
Case &$b(p)$& $h(\divv{\uu})$\\[0.2cm]
\hline
1  &$e^p$&$ (\divv{\uu})^{3} $\\[0.2cm]
2  &$p^{3}$& $      (\divv{\uu})^{3} $\\[0.2cm]
3  &$ \sqrt[3]{p} $&$    (\divv{\uu})^{3} $\\[0.2cm]
4  &$ p^3 $&$\sqrt[3]{(\divv{\uu})^{5}} $\\[0.2cm]
5  &$ \sqrt[3]{p}$&$\sqrt[3]{(\divv{\uu})^{5}}$ \\
\hline
\end{tabular}
 \end{center}
\end{table}

Figures \ref{FL1L2_iterationsG} - \ref{FL1L2_iterationsG5} illustrate the numbers of iterations for the  splitting and monolithic L-schemes for different values of $L_1$ and $L_2$. A relative similar behaviour with respect to the tuning parameters $L_1, L_2$, the value of $K$ and of coefficient $\alpha$ for the two proposed schemes is observed. The schemes are sensitive  to the  choice of the coefficient functions $b(\cdot)$ and $h(\cdot)$. We remark that the region for faster convergence of the fourth and fifth cases (Figures \ref{FL1L2_iterationsG3} and \ref{FL1L2_iterationsG5}) is more narrow than of the firsts three cases.  In all cases, the proposed iterative scheme was more sensitive with respect to the parameter $L_2$ as to $L_1$, Figures \ref{FL1L2_iterationsG} - \ref{FL1L2_iterationsG5}. The fastest convergence for the both schemes was obtained when $L_1 \sim L_b$ and $L_2 \sim L_h$. 

The convergence of the proposed scheme is clearly depending on the value of the permeability, as one can see in Figures \ref{KFL1L2_iterations4} and \ref{KFL1L2_iterations5}. In accordance with the theoretical results in Section \ref{sec:convergence}, a higher permeability implies a faster convergence. Moreover, we have tested the numerical schemes with different mesh sizes and different time step. The results in Figure \ref{dh_bexp} shows that the schemes are converging faster when the time step decrease and the convergence is not depending of the mesh diameter. This was obtained  by setting $L_1=L_b$ and $L_2=L_h$. Nevertheless, by running the same case 1 but decreasing the linearization parameters  $L_1$ and $L_2$ in two order of magnitude the schemes are converging faster when the time step increases and that the convergence is not depending of the mesh diameter, confirming again the theory (See Figure \ref{dh_cbrtF2}).
Finally, we test the scheme for different values of the Biot coupling constant $\alpha$, see Figures  \ref{AFL1L2_iterations5} and \ref{AFL1L2_iterations6}. It seems that the convergence is relatively independent of the values of $\alpha$.

In Figures \ref{timeL_scheme}-\ref{GMRESITERATION} we further compared the two proposed schemes  with respect to  CPU time and number of both non-linear and GMRES iterations.	A natural advantage of the splitting L-scheme is that it decouples the system of equations in two, one corresponding to flow and one for mechanics. Thus the resulting linear systems are composed of two positive definite problems, rather than the saddle-point structure arising from the monolithic linearization. A comparison of the CPU time for test problem 1, case 1 is shown in Figure \ref{timeL_scheme}. The schemes are performing similar, when no preconditioning is applied. The same is observed for the total number of iterations in Figure \ref{Linearization_procedure}. In Figure \ref{GMRESITERATION}  we report the number of GMRES iterations needed to resolve the linear system associated with the flow problem for the splitting L-scheme, compared to number of iterations needed for the monolithic L-scheme. For the cases considered here, the number of iterations needed to solve the mechanics problem in the splitting L-scheme was neglectable. When no preconditioning is applied, we see that the number of iterations needed increases dramatically with grid refinement, as expected. The two schemes are performing similarly, with the splitting being a bit better.

Next, a block preconditioner based on the proposed splitting L-scheme was applied to the monolithic L-scheme in order to reduce the number of GMRES iterations \cite{Tchelepi_2016_Preconditionet}. As a result of preconditioning, GMRES method needed less than 10 iterations to converge for all mesh sizes tested (see Figure \ref{GMRESITERATION}). A substantial reduction in the CPU time is observed, see Figure \ref{timeL_scheme}. We applied the same block preconditioner also to the splitting L-scheme, separately to flow problem and mechanics. Again, the CPU time is strongly reduced (by two orders of magnitude for the largest mesh), see Figure \ref{timeL_scheme}. The splitting L-scheme is now faster then the monolithic L-scheme, as one can see in Figure \ref{timeL_scheme}. We remark also the much smaller number of non-linear iterations for the splitting L-scheme comparing to the monolithic L-scheme in Figure \ref{Linearization_procedure} and the converse situation regarding GMRES iterations in Figure \ref{GMRESITERATION}.  Although the monolithic L-scheme preconditioned needs just 10 GMRES iterations to converge, it requires more computational cost due to an inner linear solver that the block preconditioner has inside.

\begin{figure}[h!]
\centering
 \begin{subfigure}{0.5\textwidth}
 \centering
\includegraphics[scale=0.25,trim={0 0 0 0},clip]{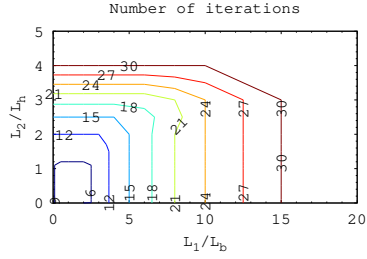} 
        \caption{Splitting}
       \label{FL1L2_SimpleNoLineal1}
    \end{subfigure}%
       ~ 
    \begin{subfigure}{0.5\textwidth}
    \centering
\includegraphics[scale=0.25]{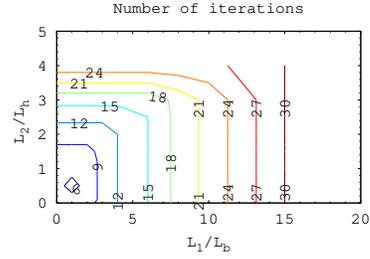}      
   \caption{Monolithic}
      \label{FL1L2_SimpleNoLineal2}
    \end{subfigure}    
    \caption{Performance of the iterative schemes for different values of $L_1$ and $L_2$ for  test problem 1, case 1: $b(p)=e^p;\ $ $h(\divv \uu)=(\divv \uu) ^3$.}
    	    \label{FL1L2_iterationsG}
\end{figure}	 	
 		
\begin{figure}[h!]
\centering
 \begin{subfigure}{0.5\textwidth}
 \centering
\includegraphics[scale=0.25,trim={0 0 0 0},clip]{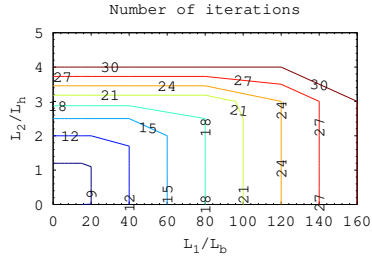} 
        \caption{Splitting}
       \label{FL1L2_SimpleNoLineal12}
    \end{subfigure}%
       ~ 
    \begin{subfigure}{0.5\textwidth}
    \centering
\includegraphics[scale=0.25]{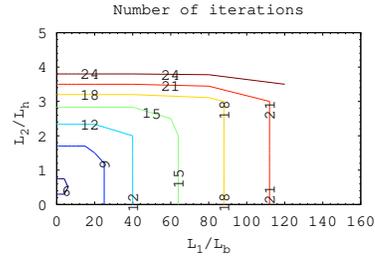}      
   \caption{Monolithic}
      \label{FL1L2_SimpleNoLineal22}
    \end{subfigure}    
    \caption{Performance of splitting algorithm for different values of $L_1$ and $L_2$ for  test problem 1, case 2: $b(p)=p^3;\ $ $h(\divv \uu)=(\divv \uu) ^3$.}
    	    \label{FL1L2_iterationsG2}
\end{figure}

\begin{figure}[h!]
\centering
 \begin{subfigure}{0.5\textwidth}
 \centering
\includegraphics[scale=0.25,trim={0 0 0 0},clip]{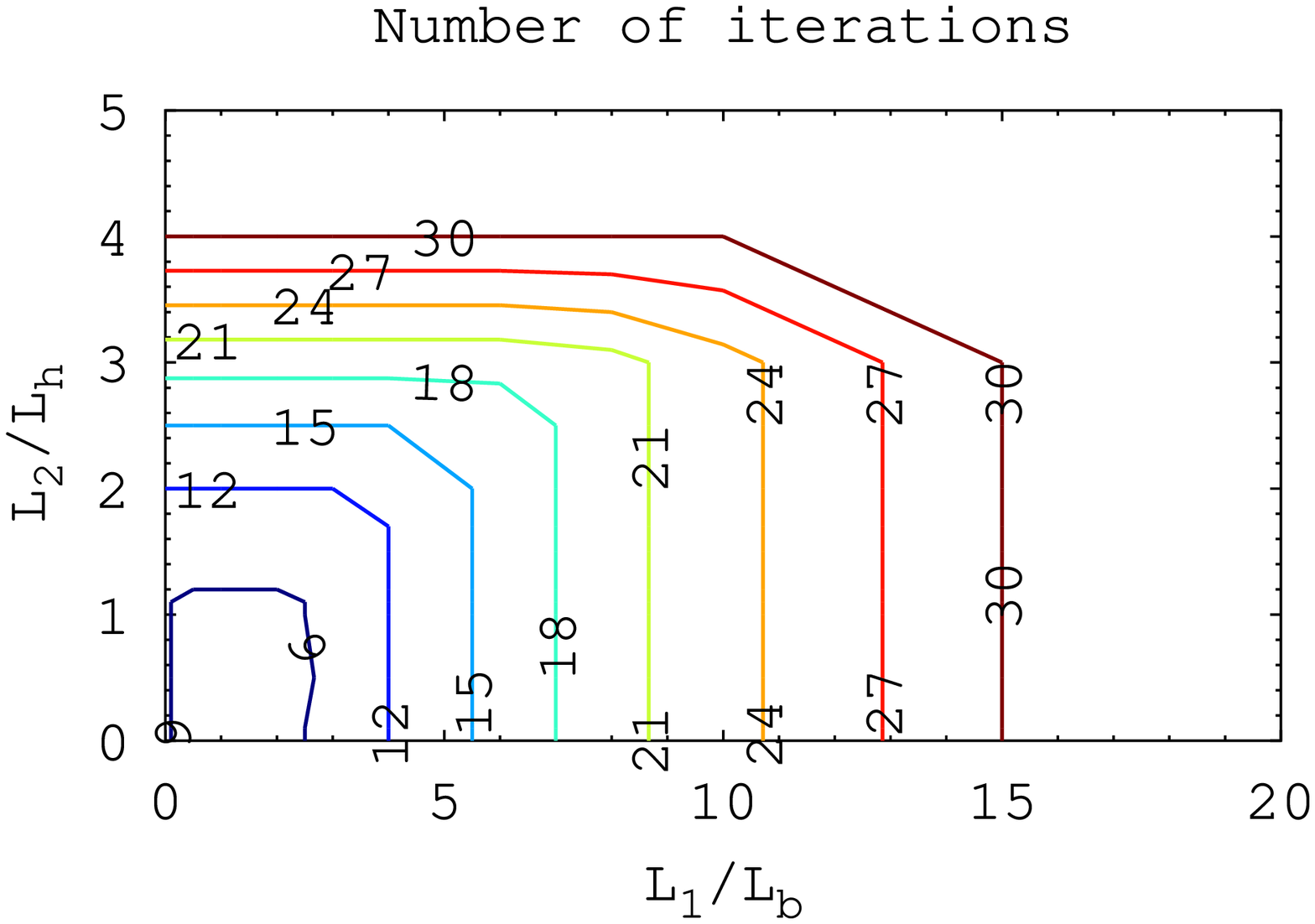} 
        \caption{Splitting}
       \label{FL1L2_SimpleNoLineal14}
    \end{subfigure}%
       ~ 
    \begin{subfigure}{0.5\textwidth}
    \centering
\includegraphics[scale=0.25]{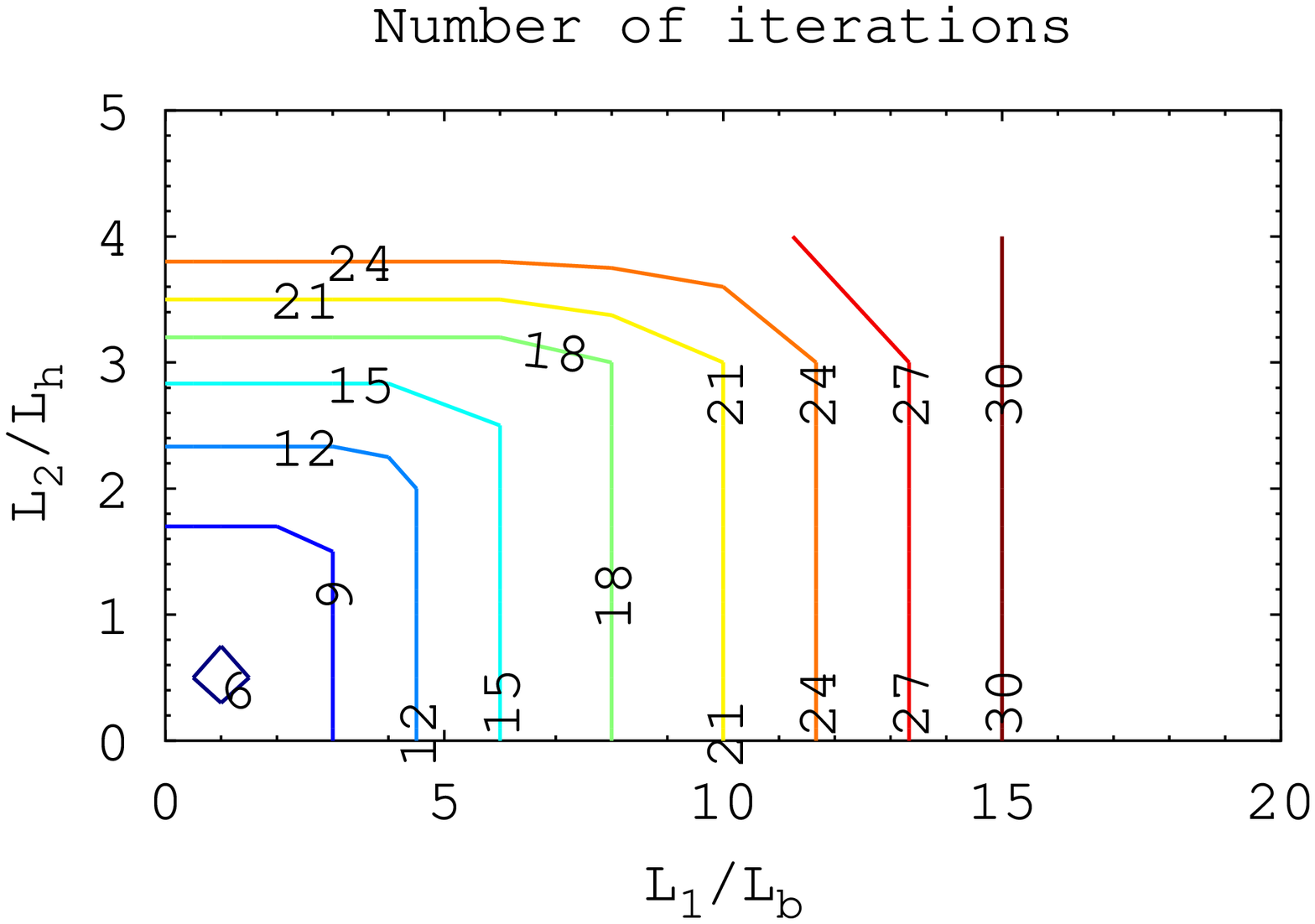}      
   \caption{Monolithic}
      \label{FL1L2_SimpleNoLineal24}
    \end{subfigure}    
    \caption{Performance of the iterative schemes for different values of $L_1$ and $L_2$ for  test problem 1, case 3: $b(p)=\sqrt[3]{p};\ $ $h(\divv \uu)=(\divv \uu) ^3$.}
    	    \label{FL1L2_iterationsG4}
\end{figure}

\begin{figure}[h!]
\centering
 \begin{subfigure}{0.5\textwidth}
 \centering
\includegraphics[scale=0.25,trim={0 0 0 0},clip]{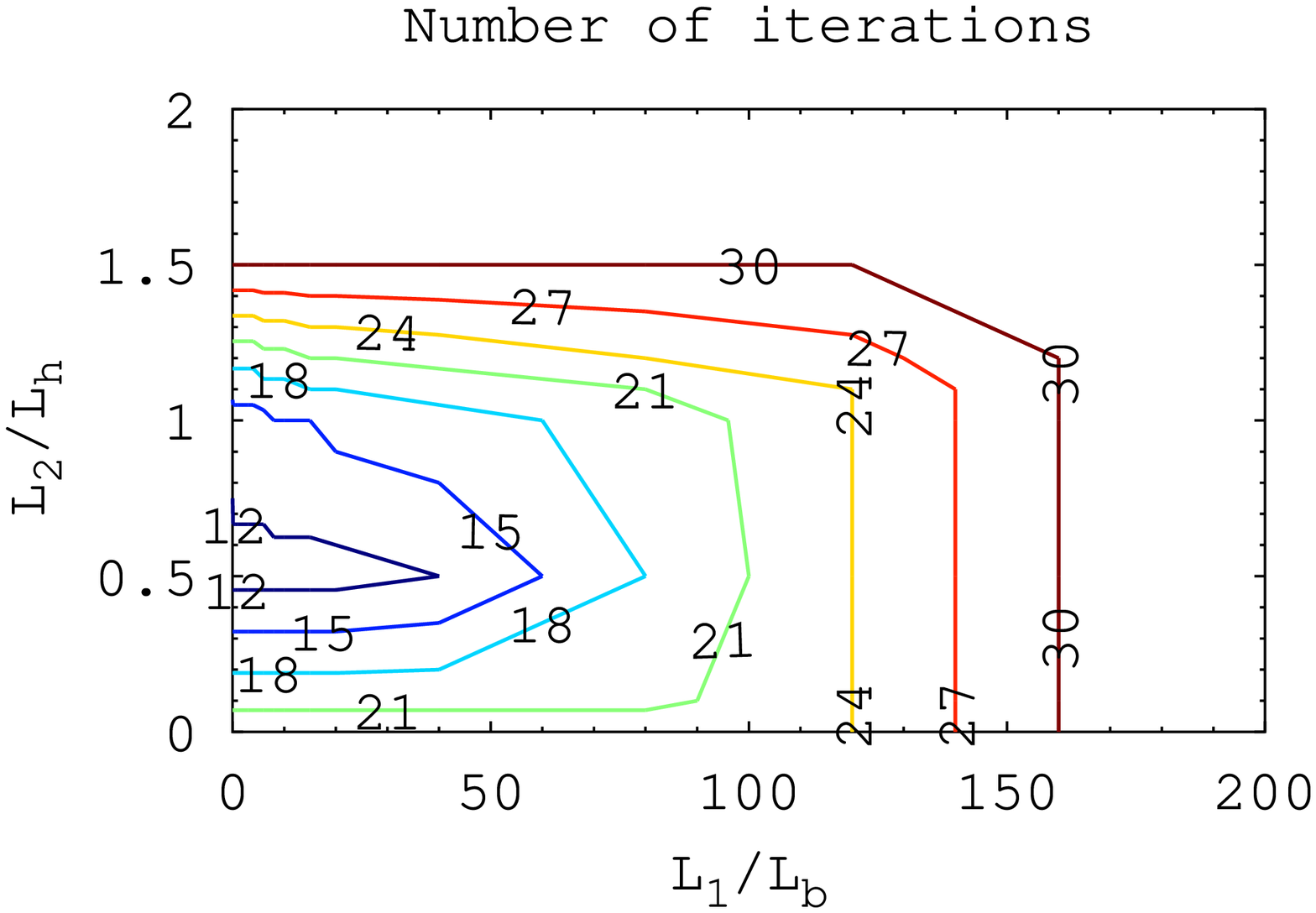} 
        \caption{Splitting}
       \label{FL1L2_SimpleNoLineal13}
    \end{subfigure}%
       ~ 
    \begin{subfigure}{0.5\textwidth}
    \centering
\includegraphics[scale=0.25]{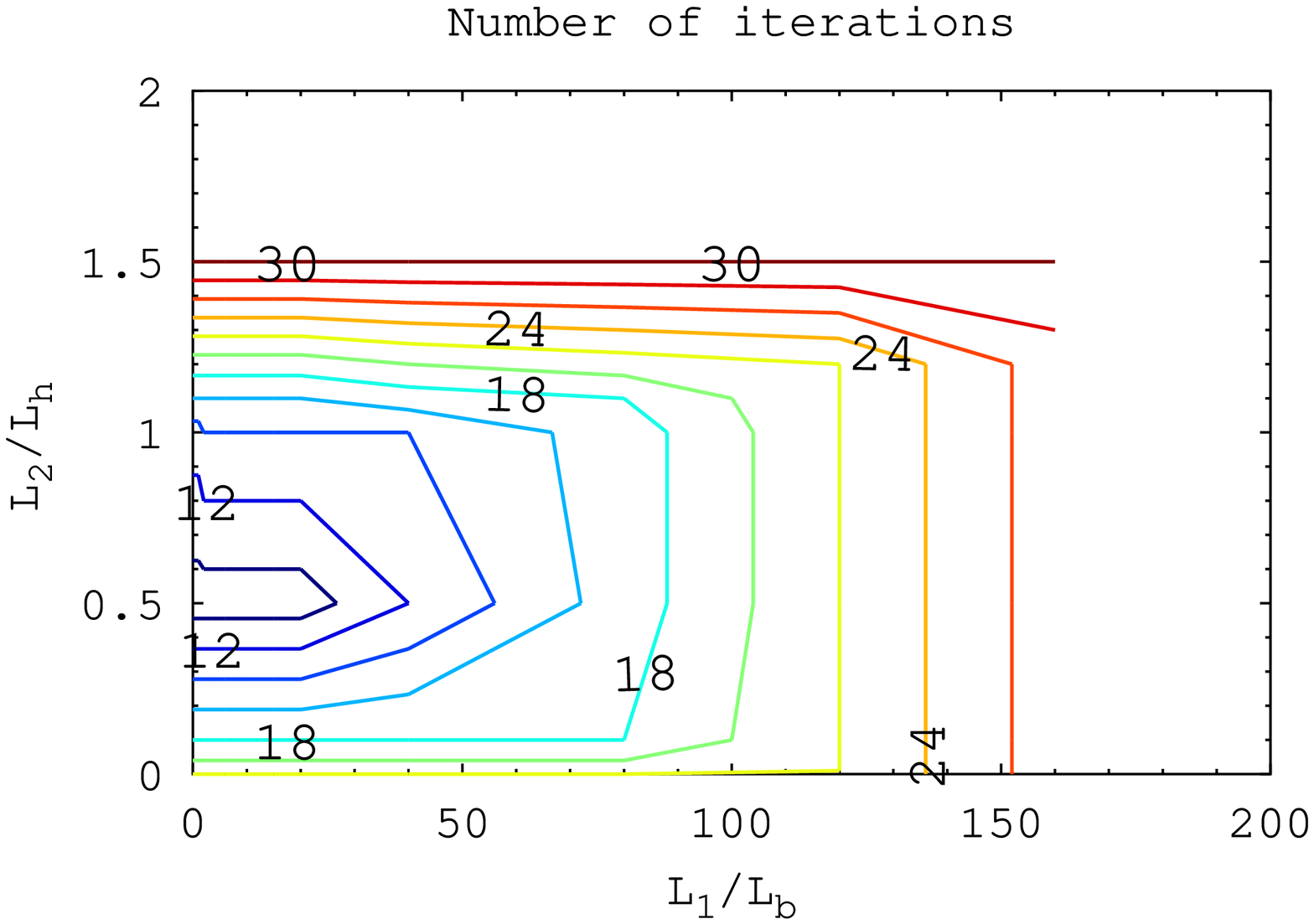}      
   \caption{Monolithic}
      \label{FL1L2_SimpleNoLineal23}
    \end{subfigure}    
    \caption{Performance of the iterative schemes for different values of $L_1$ and $L_2$ for  test problem 1, case 2: $b(p)=p^3;\ $ $h(\divv \uu)= \sqrt[3]{(\divv \uu)^5}$.}
    	    \label{FL1L2_iterationsG3}
\end{figure}

\begin{figure}[h!]
\centering
 \begin{subfigure}{0.5\textwidth}
 \centering
\includegraphics[scale=0.25,trim={0 0 0 0},clip]{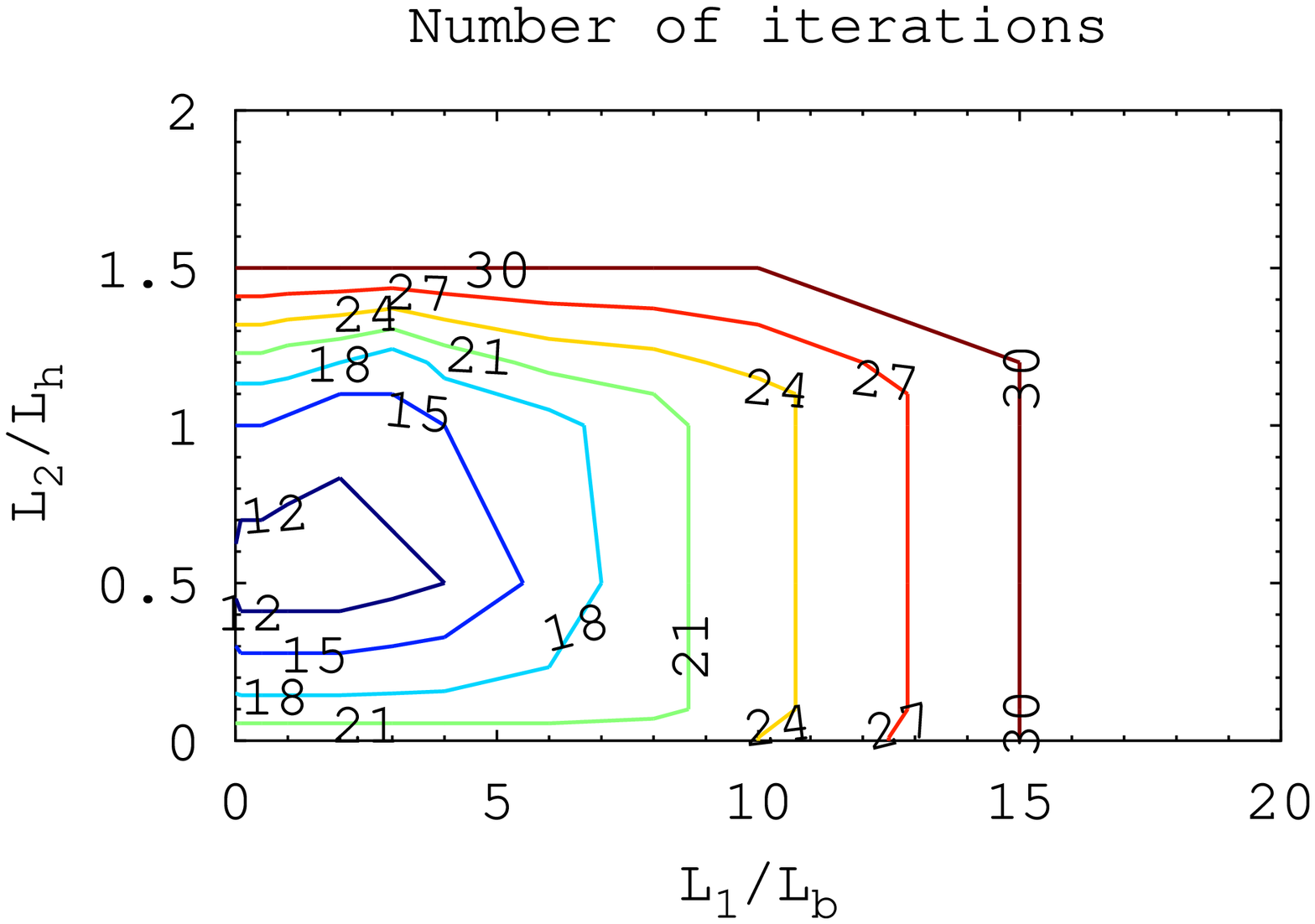} 
        \caption{Splitting}
       \label{FL1L2_SimpleNoLineal15}
    \end{subfigure}%
       ~ 
    \begin{subfigure}{0.5\textwidth}
    \centering
\includegraphics[scale=0.25]{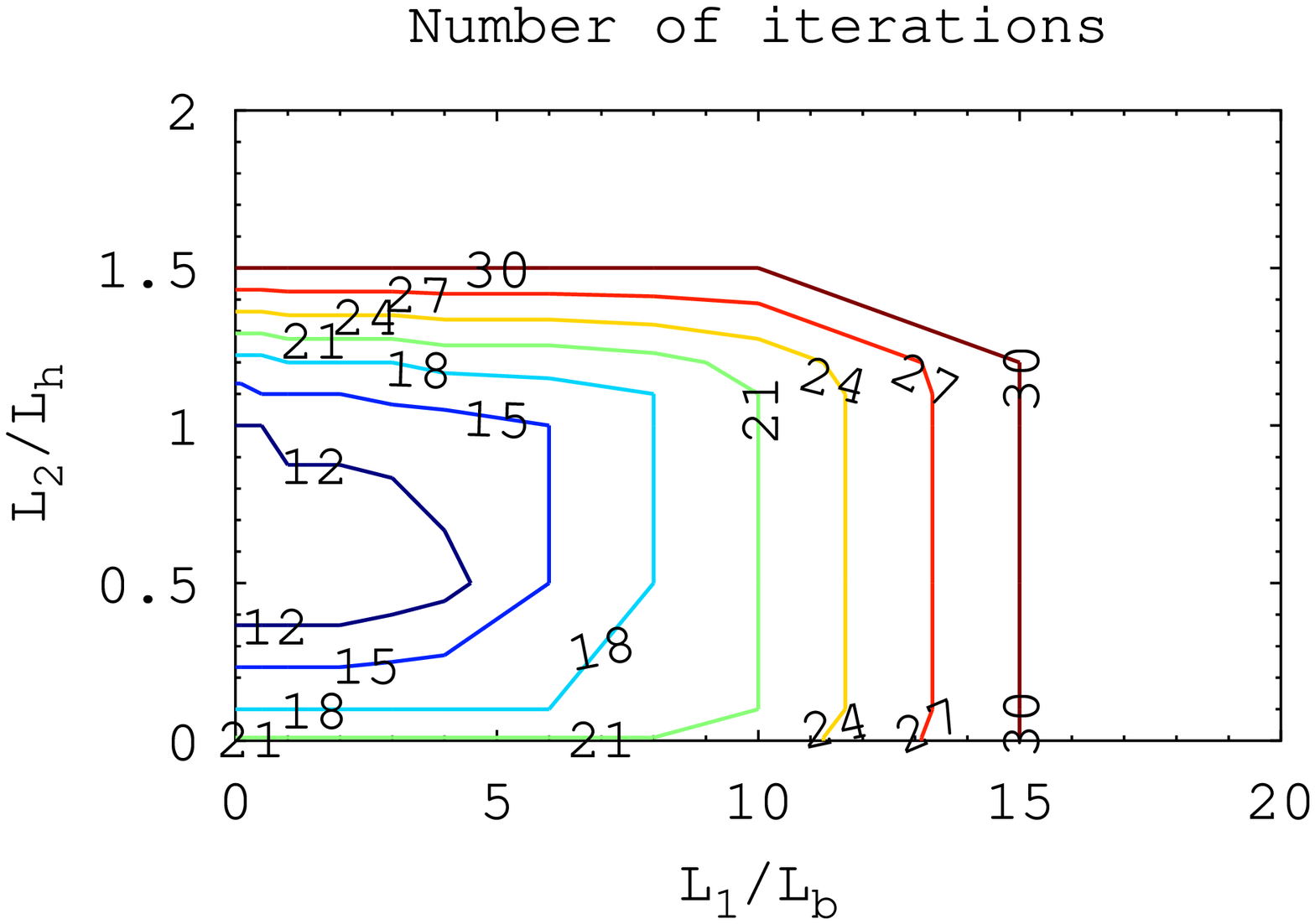}      
   \caption{Monolithic}
      \label{FL1L2_SimpleNoLineal25}
    \end{subfigure}    
    \caption{Performance of the iterative schemes for different values of $L_1$ and $L_2$ for test problem 1, case 3: $b(p)=\sqrt[3]{p};\ $ $h(\divv \uu)= \sqrt[3]{(\divv \uu)^5}$.}
    	    \label{FL1L2_iterationsG5}
\end{figure}	
 
\begin{figure}[h]
\centering
 \begin{subfigure}{0.5\textwidth}
 \centering
\includegraphics[scale=0.25,trim={0 0 0 0},clip]{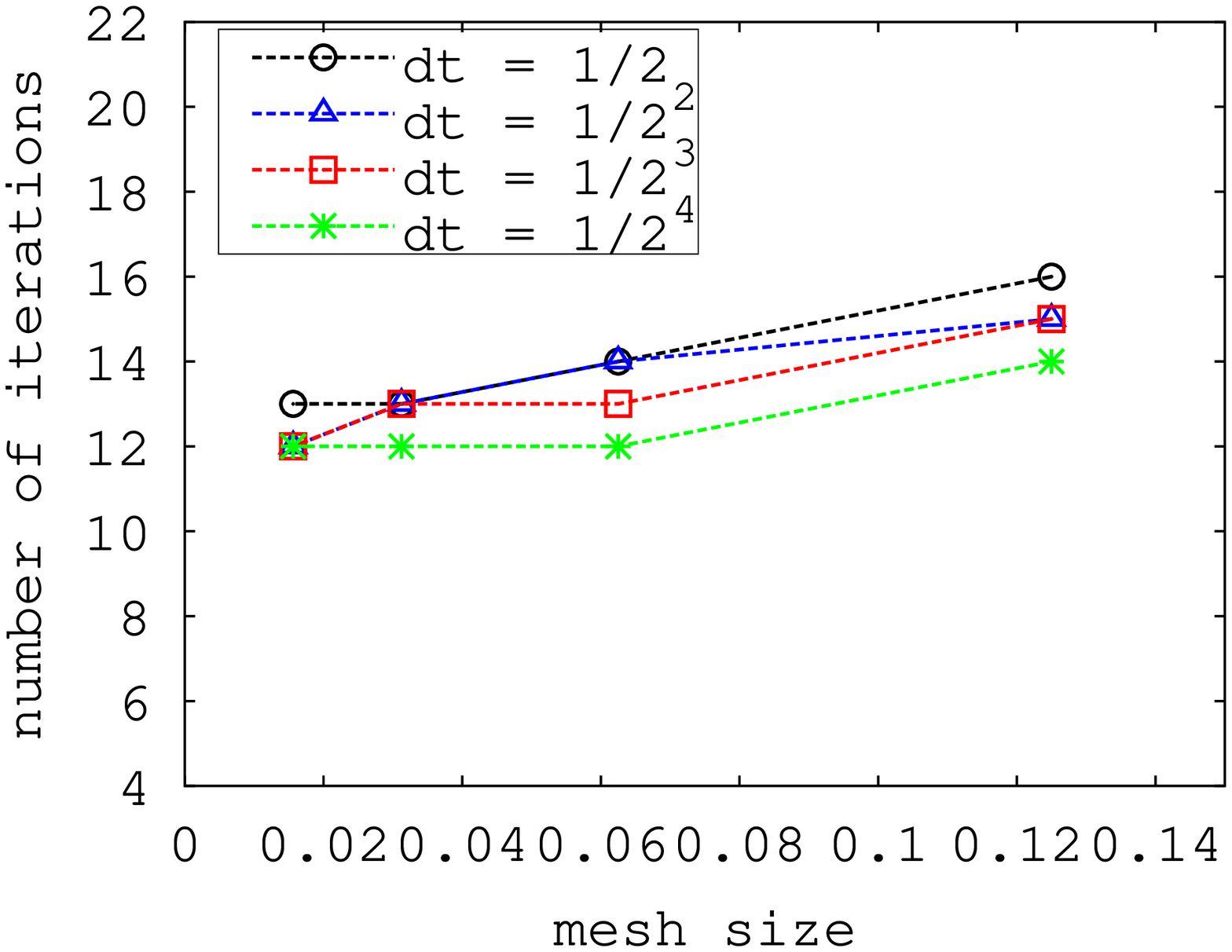} 
        \caption{Splitting}
       \label{dh_bexpS}
    \end{subfigure}%
       ~ 
    \begin{subfigure}{0.5\textwidth}
    \centering
\includegraphics[scale=0.25]{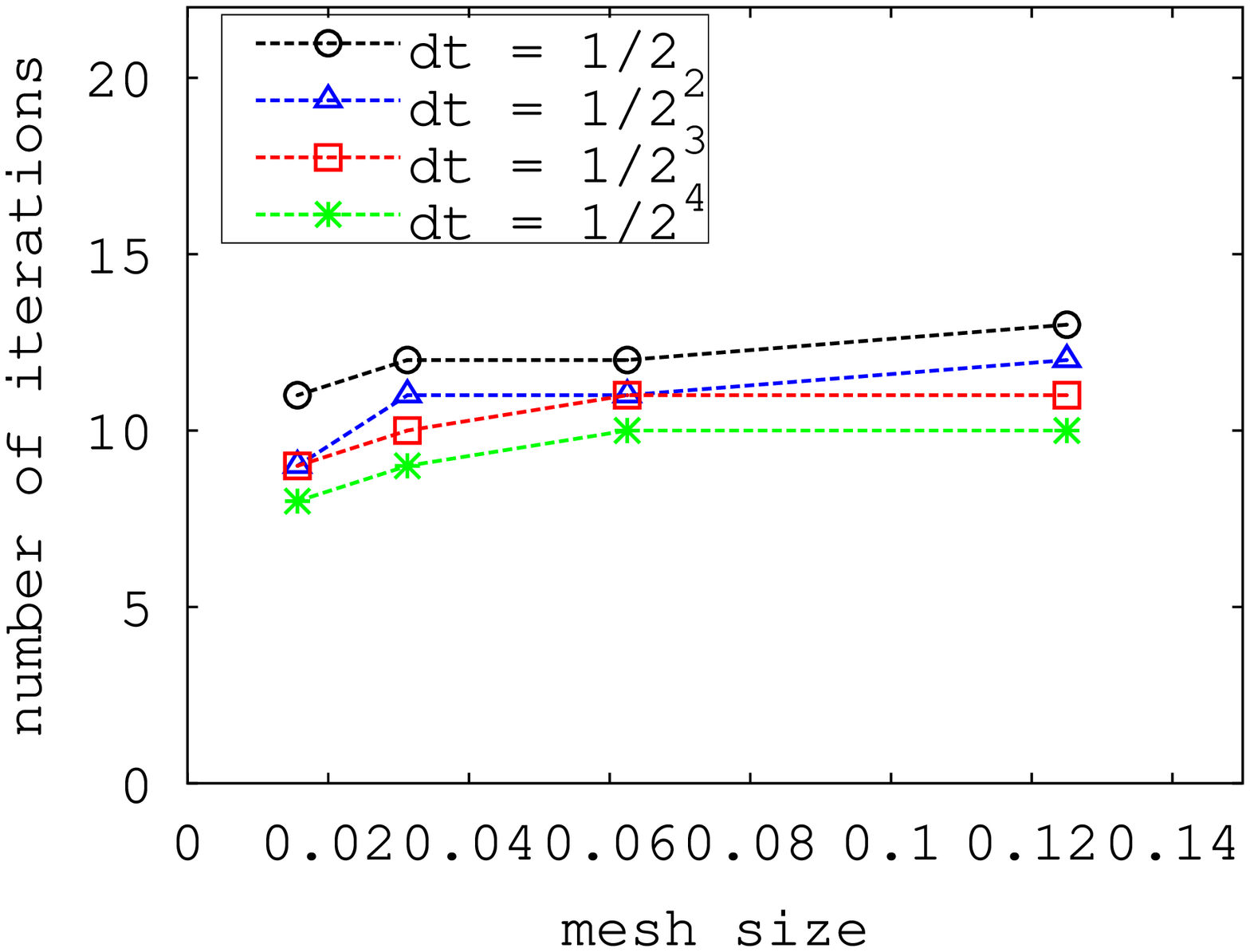}     
   \caption{Monolithic}
      \label{dh_bexpM}
    \end{subfigure}    
    \caption{Performance of the iterative schemes for different mesh sizes for test problem 1, case 1: $L_1=L_b;\ $  $L_2=L_h$.}
    	    \label{dh_bexp}
\end{figure}	 	
 
 \begin{figure}[h]
\centering
 \begin{subfigure}{0.5\textwidth}
 \centering
\includegraphics[scale=0.25,trim={0 0 0 0},clip]{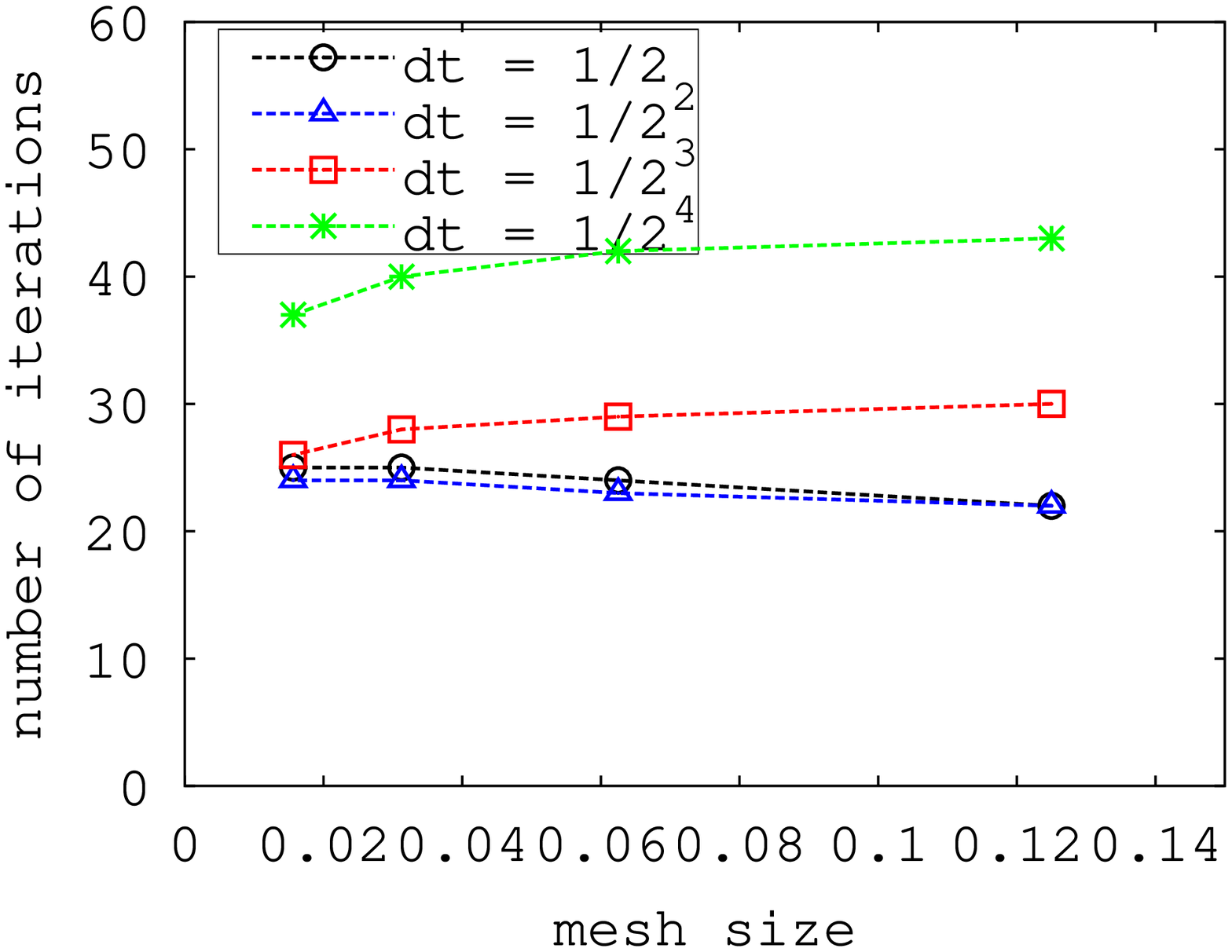} 
        \caption{Splitting}
       \label{dh_cbrtF2S}
    \end{subfigure}%
       ~ 
    \begin{subfigure}{0.5\textwidth}
    \centering
\includegraphics[scale=0.25]{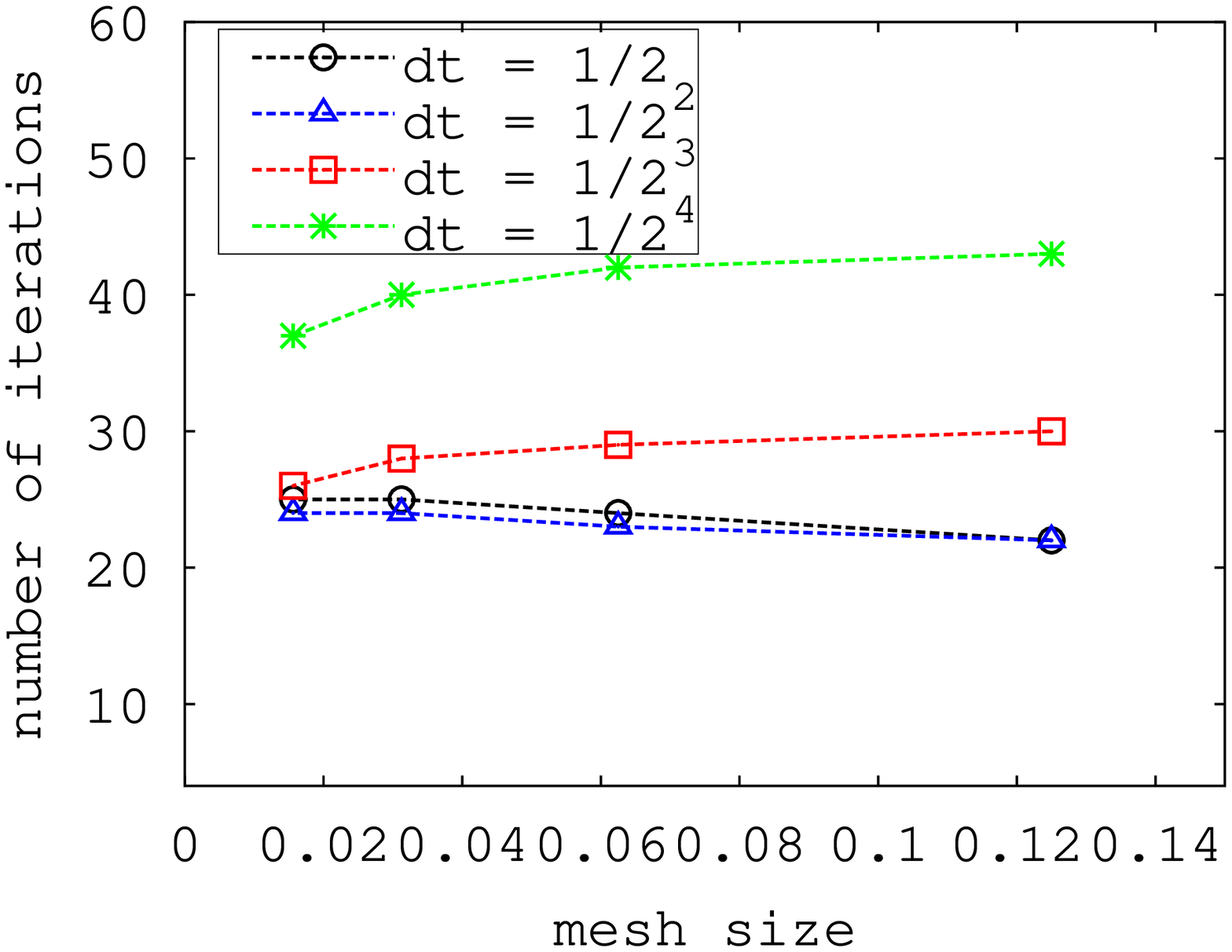}     
   \caption{Monolithic}
      \label{dh_cbrtF2M}
    \end{subfigure}    
    \caption{Performance of the iterative schemes for different mesh sizes for test problem 1, case 1: $L_1=10^{-3}L_b;\ $  $L_2=10^{-3}L_h$.}
    	    \label{dh_cbrtF2}
\end{figure}	 	
 
 \begin{figure}[h]
\centering
 \begin{subfigure}{0.5\textwidth}
 \centering
\includegraphics[scale=0.25,trim={0 0 0 0},clip]{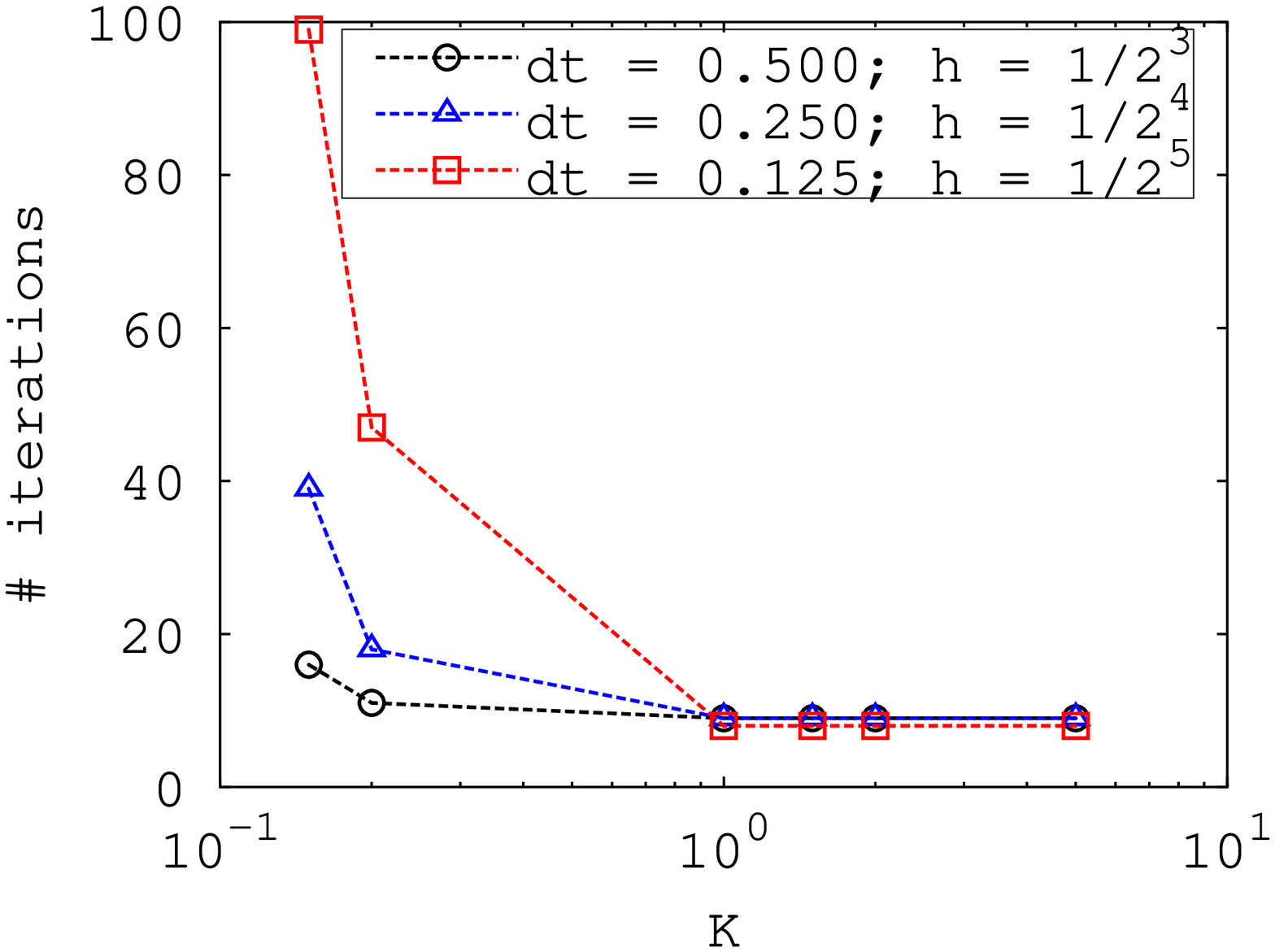} 
        \caption{Splitting}
       \label{FL1L2_SimpleNoLineal14}
    \end{subfigure}%
       ~ 
    \begin{subfigure}{0.5\textwidth}
    \centering
\includegraphics[scale=0.25]{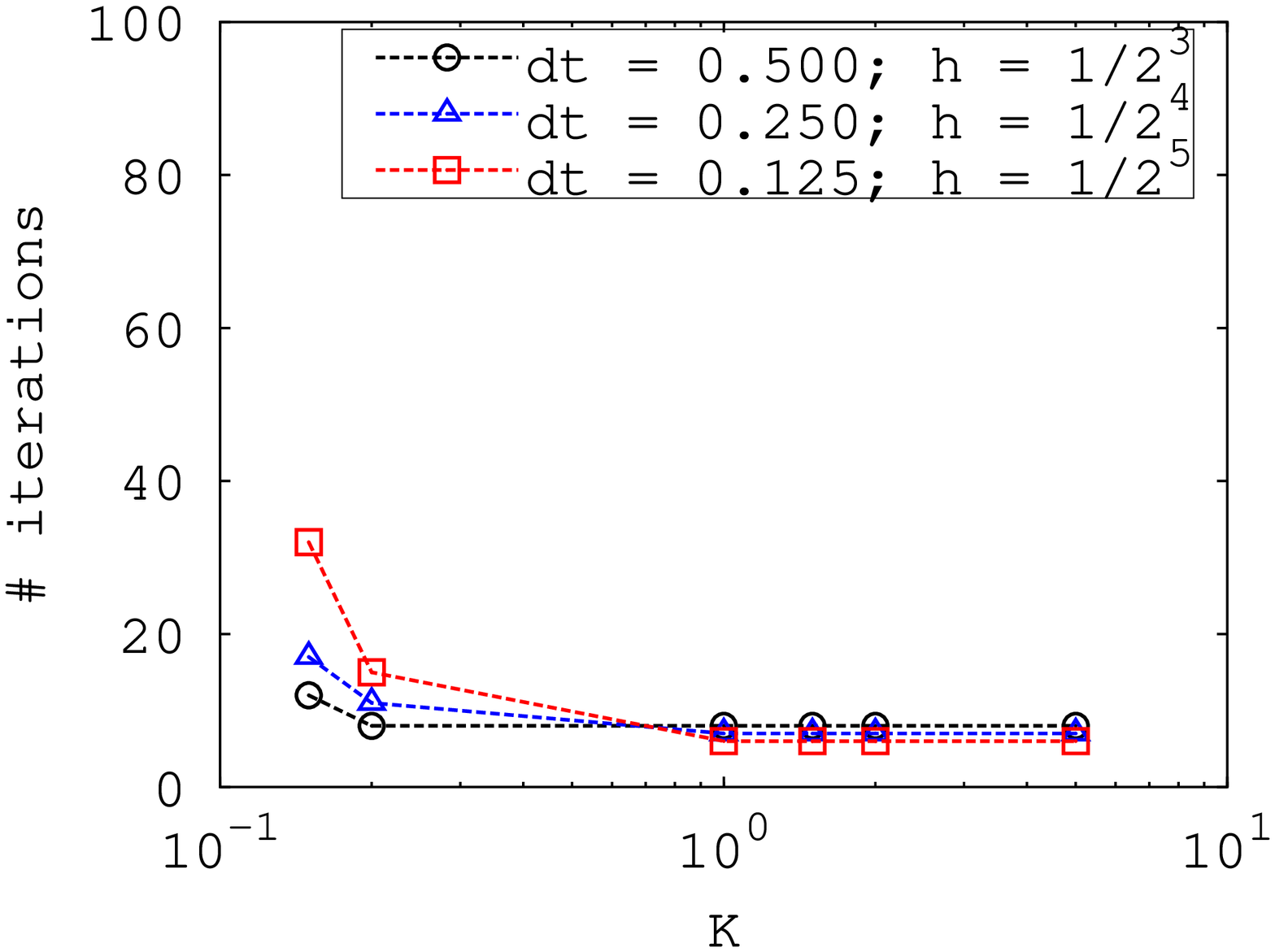}      
   \caption{Monolithic}
      \label{FL1L2_SimpleNoLineal24}
    \end{subfigure}    
    \caption{Number of iterations for different mesh sizes and different values of $\Delta t$, $K$ for test problem 1, case 3: $b(p)=p^3;\ $ $h(\divv \uu)=(\divv \uu) ^3$.}
    	    \label{KFL1L2_iterations4}
\end{figure}	
 

  \begin{figure}[h]
\centering
 \begin{subfigure}{0.5\textwidth}
 \centering
\includegraphics[scale=0.25,trim={0 0 0 0},clip]{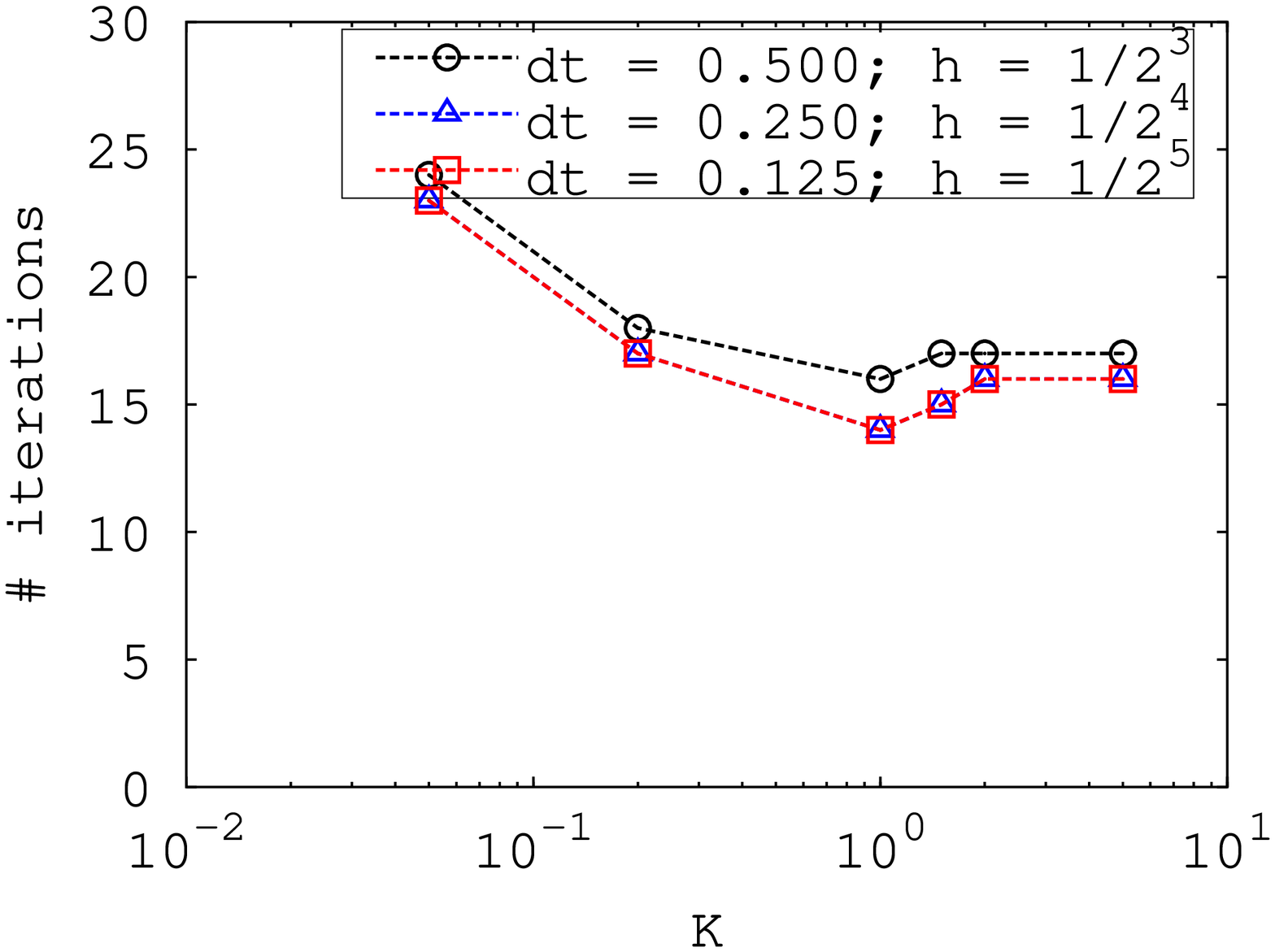} 
        \caption{Splitting}
       \label{KFL1L2_SimpleNoLineal14}
    \end{subfigure}%
       ~ 
    \begin{subfigure}{0.5\textwidth}
    \centering
\includegraphics[scale=0.25]{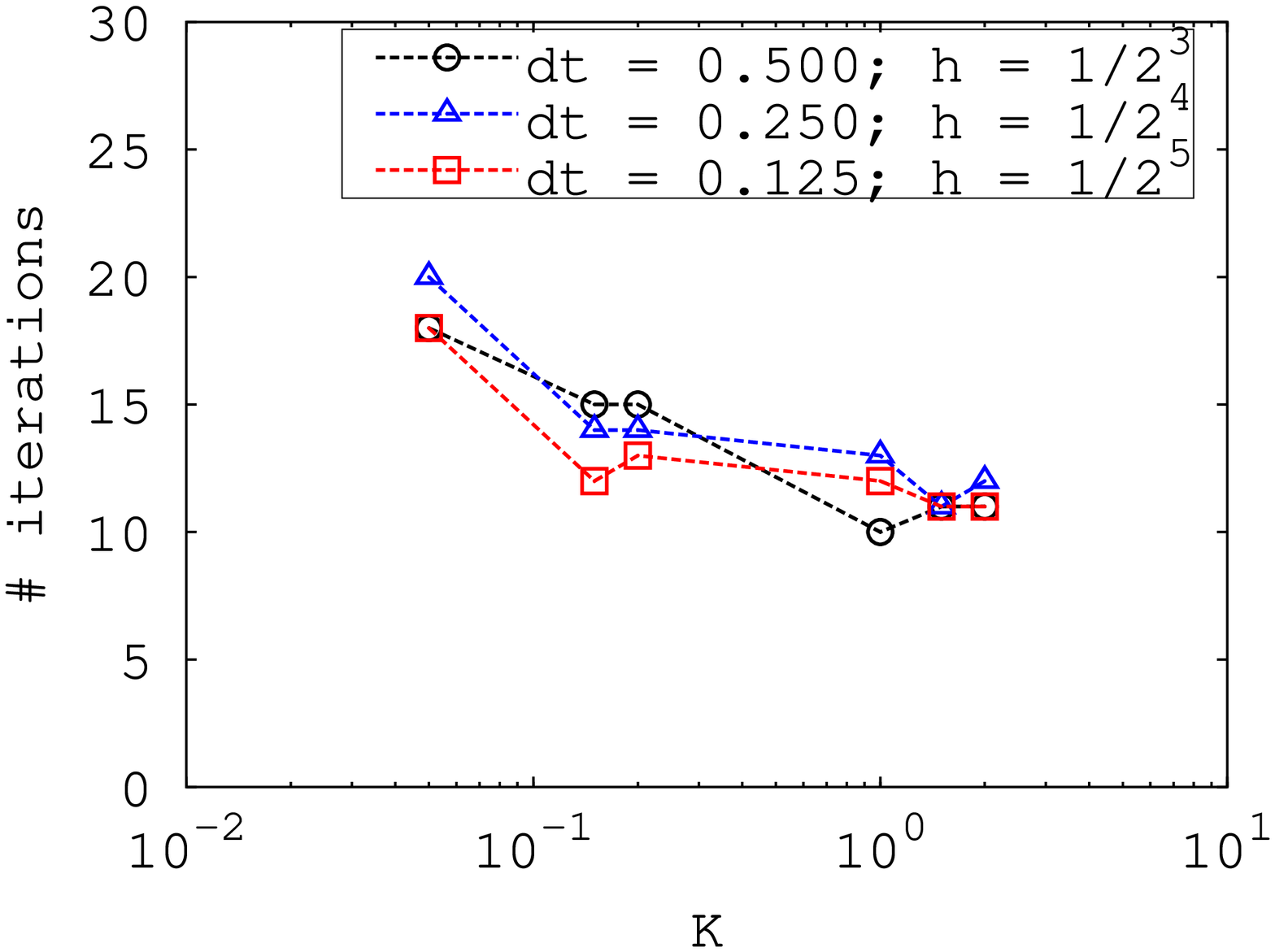}      
   \caption{Monolithic}
      \label{KFL1L2_SimpleNoLineal24}
    \end{subfigure}    
    \caption{Number of iterations for different mesh sizes and different values of $\Delta t$, $K$  for test problem 1, case 5: $b(p)=\sqrt[3]{p};\ $ $h(\divv \uu)=\sqrt[3]{(\divv \uu)^5}$.}
    	    \label{KFL1L2_iterations5}
\end{figure}

\begin{figure}[h]
\centering
 \begin{subfigure}{0.5\textwidth}
 \centering
\includegraphics[scale=0.25,trim={0 0 0 0},clip]{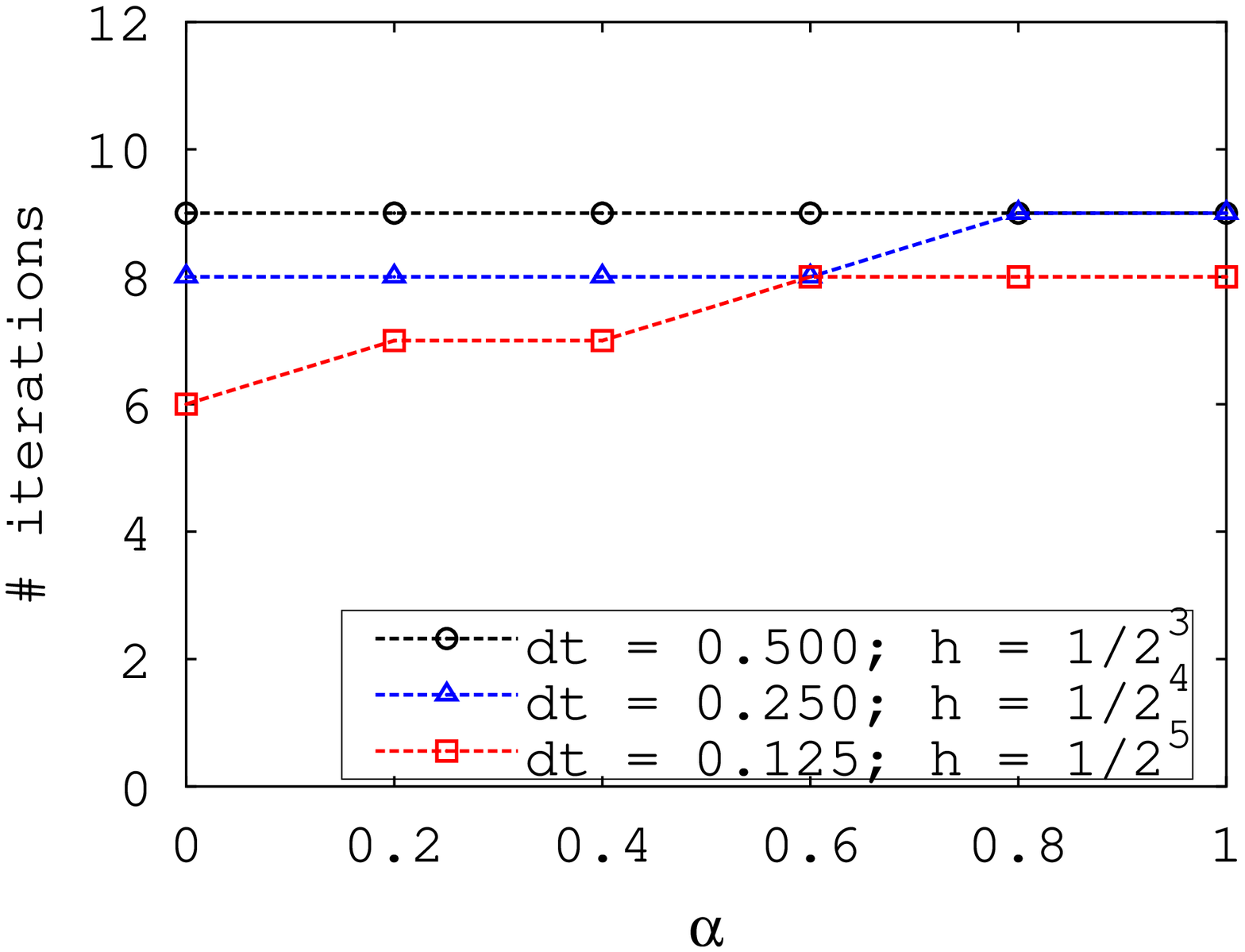} 
        \caption{Splitting}
       \label{FL1L2_SimpleNoLineal1}
    \end{subfigure}%
       ~ 
    \begin{subfigure}{0.5\textwidth}
    \centering
\includegraphics[scale=0.25]{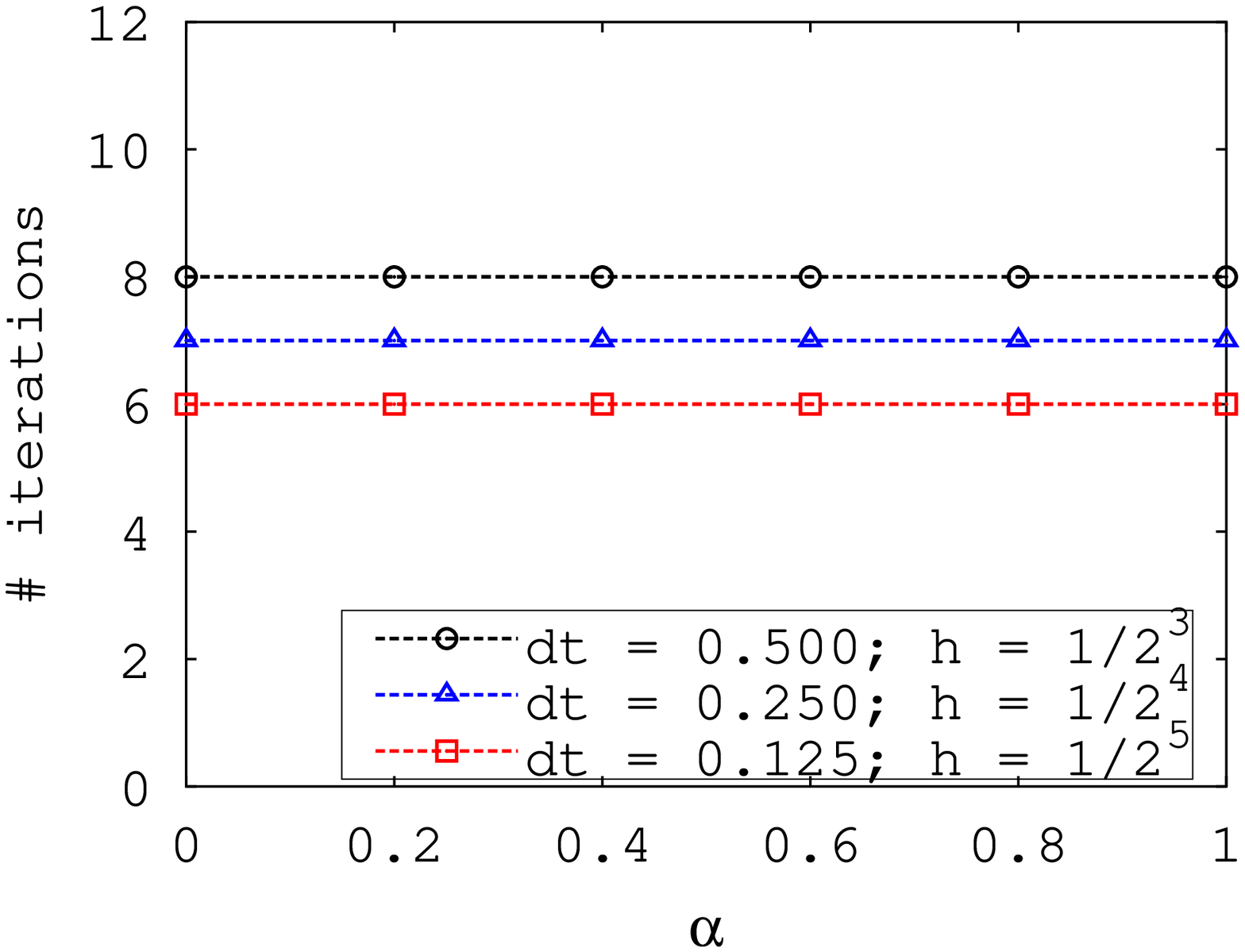}      
   \caption{Monolithic}
      \label{FL1L2_SimpleNoLineal2}
    \end{subfigure}    
    \caption{Number of iterations for different mesh sizes and different values of $\Delta t$, $\alpha$  for test problem 1, case 3: $b(p)=p^3;\ $ $h(\divv \uu)=(\divv \uu)^3$.}
    	    \label{AFL1L2_iterations5}
\end{figure}

\begin{figure}[t]
\centering
 \begin{subfigure}{0.5\textwidth}
 \centering
\includegraphics[scale=0.25,trim={0 0 0 0},clip]{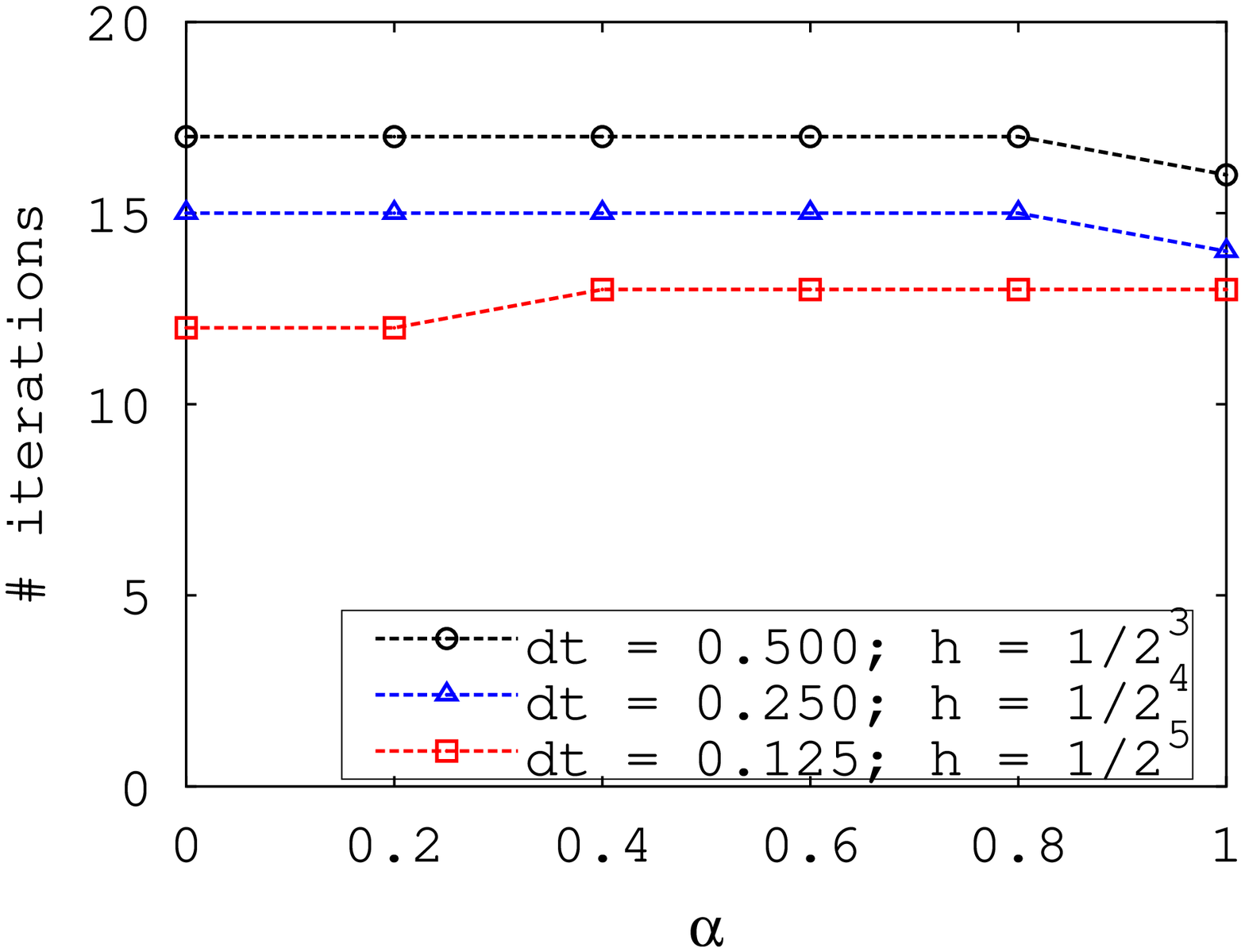} 
        \caption{Splitting}
       \label{AFL1L2_iterations6}
    \end{subfigure}%
       ~ 
    \begin{subfigure}{0.5\textwidth}
    \centering
\includegraphics[scale=0.25]{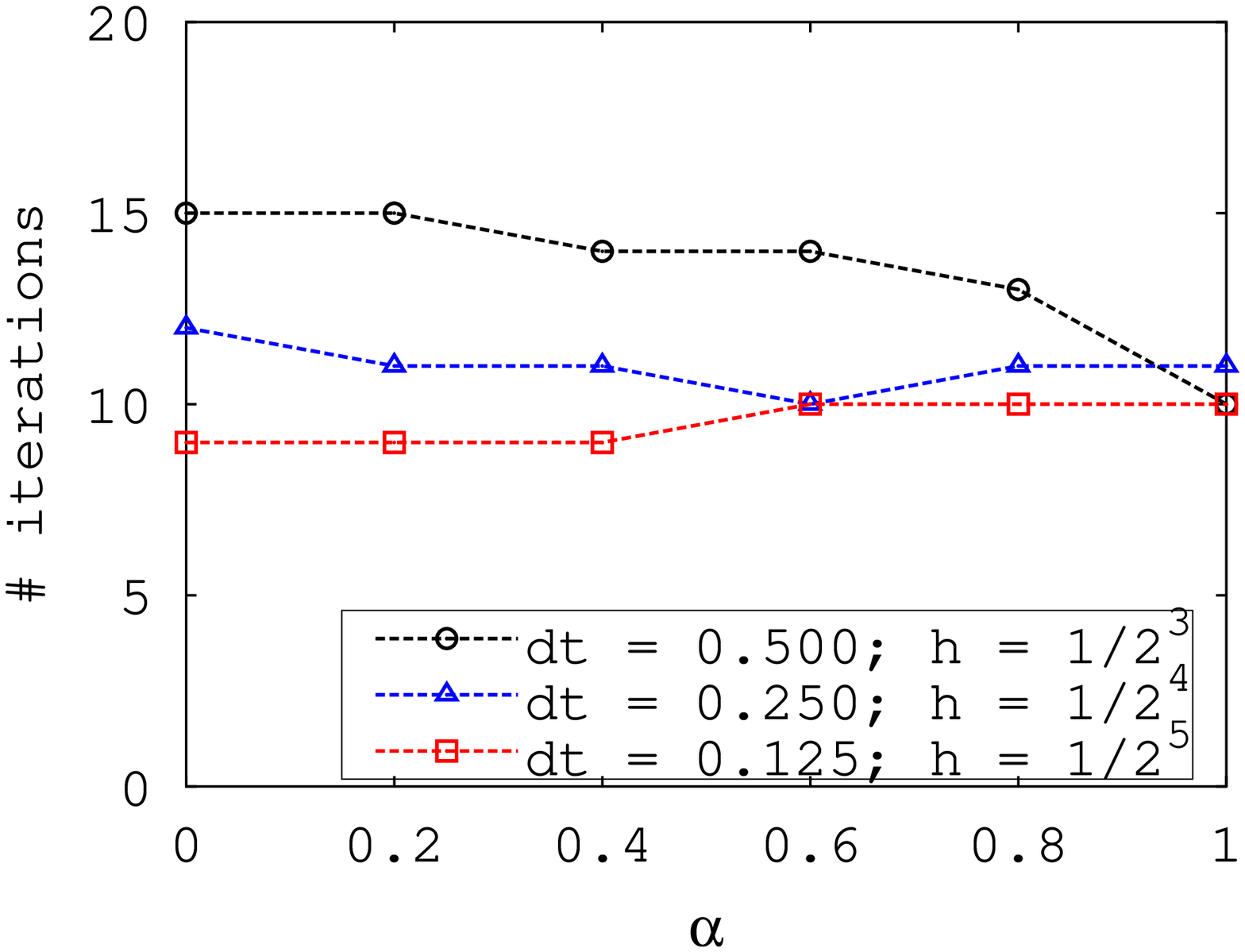}      
   \caption{Monolithic}
      \label{FL1L2_SimpleNoLineal2}
    \end{subfigure}    
    \caption{Number of iterations for different mesh sizes and different values of $\Delta t$, $\alpha$  for test problem 1, case 3: $b(p)=\sqrt[3]{p};\ $ $h(\divv \uu)=\sqrt[3]{(\divv \uu)^5}$.}
    	    \label{AFL1L2_iterations6}
\end{figure}


\begin{figure}[b]
    \centering
    \begin{subfigure}[h]{0.5\textwidth}
        \centering
        \includegraphics[scale=0.5]{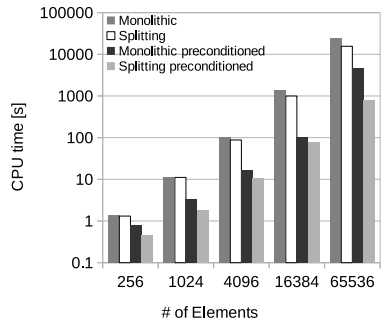}
        \caption{CPU time}
                \label{timeL_scheme}	
    \end{subfigure}%
    ~ 
    \begin{subfigure}[h]{0.5\textwidth}
        \centering
     \includegraphics[scale=0.5]{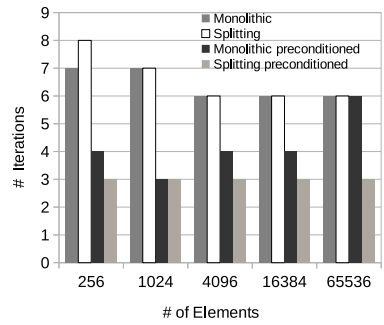}
        \caption{Linearisation procedure}
                \label{Linearization_procedure}
    \end{subfigure}
    ~
       \begin{subfigure}[h]{0.5\textwidth}
        \centering
     \includegraphics[scale=0.5]{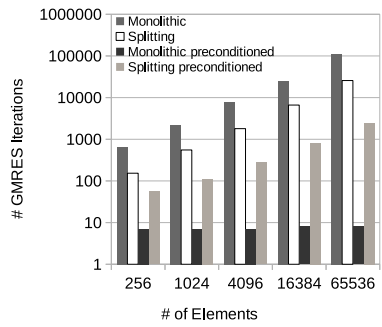}
        \caption{Number GMRES iterations}
        \label{GMRESITERATION}
    \end{subfigure}
    \caption{Performance comparison between the splitting L-scheme and the monolithic L-scheme solver for the case 1}
\label{L_scheme}	
\end{figure}

		
\clearpage

\noindent {\bf Test problem 2: a non-linear extension of Mandel's problem}\\

Mandel's problem is a relevant  2D benchmark problem with a known analytical solution  \cite{Abousleiman,Mandel}. The problem is very often used in the community, see e.g. \cite{RIS_0,Mikelic,wheeler,Carmen_2016} for verifying  the implementation and the performance of the schemes.

Mandel's problem consists in a poroelastic slab of extent $2a$ in the $x$ direction, $2b$ in the $y$ direction, and infinitely long in the z-direction, and is sandwiched between two rigid impermeable plates (see Figure \ref{xMandelProblem}). At time $t=0$, a uniform vertical load of magnitude $2F$ is applied and equal, but upward force is applied to the bottom plate. This load is supposed to remain constant. The domain is free to drain and stress-free at $x=\pm a$. Gravity is neglected.

For the numerical solution, the symmetry of the problem allows us to use a quarter of the physical domain as a computational domain (see Figure \ref{tMandel2Problem2}). Moreover, the rigid plate condition is enforced by adding constrained equations so that vertical displacement $U_y(b,t)$ on the top are equal to a known constant value.
	
\begin{figure}[h!]
    \centering
    \begin{subfigure}{0.5\textwidth}
        \centering
        		\includegraphics[scale=0.6]{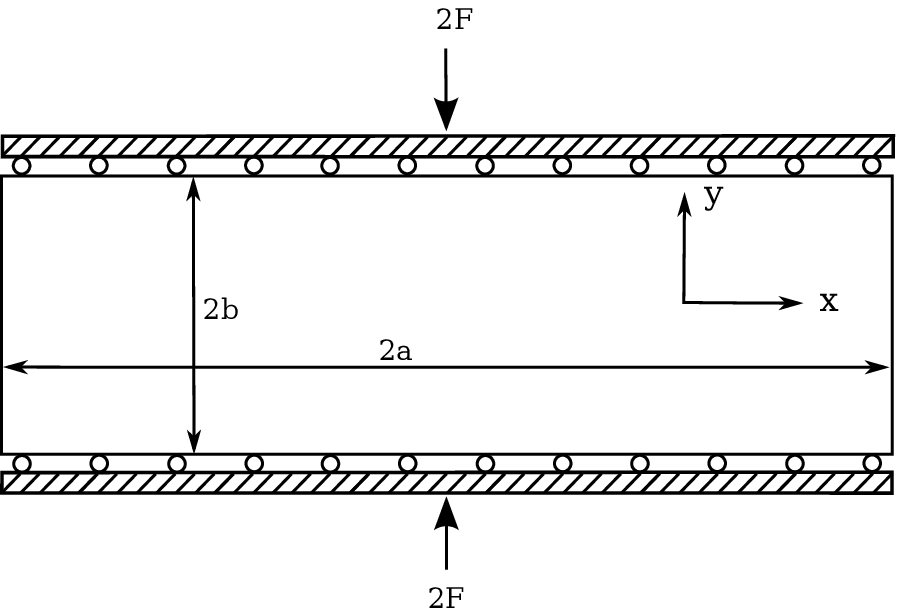} 
        \caption{Mandel's problem domain.}
                    \label{xMandelProblem}	
    \end{subfigure}%
    ~ 
    \begin{subfigure}{0.5\textwidth}
        \centering
     		\includegraphics[scale=0.6]{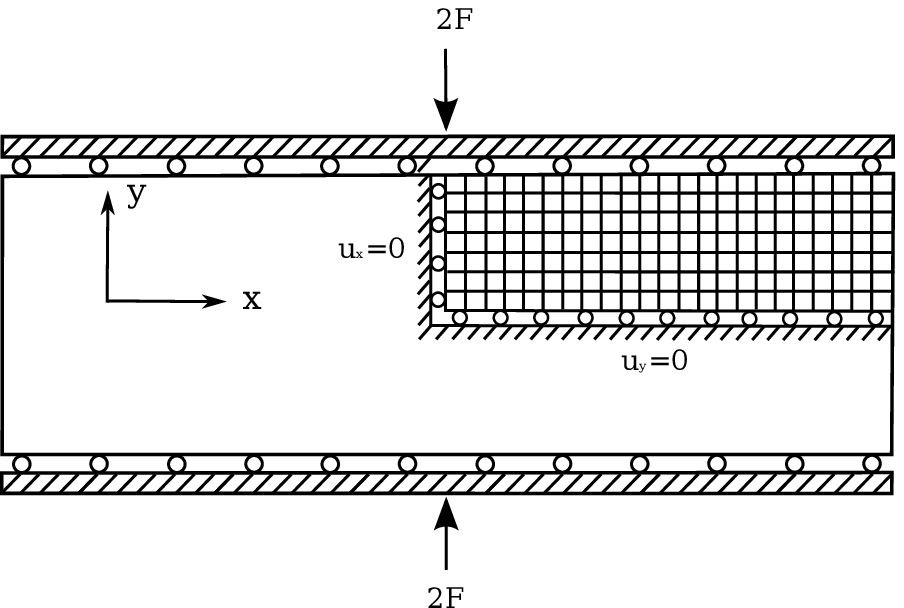} 
        \caption{Mandel's problem quarter domain}
                \label{tMandel2Problem2}
    \end{subfigure}
    \caption{Mandel's problem}
                \label{MandelProblem}	
\end{figure}	

The application of a load (2F) causes an instantaneous and uniform pressure increase throughout the domain \cite{GayX}; this is predicted theoretically 
\cite{Abousleiman} and it can be used as an initial condition
	\begin{align}
    p(x,y,0)&=\frac{F B(1+v_u)}{3a}, \nonumber\\
    \qq(x,y,0)&=\vec{0},\nonumber\\
 \uu(x,y,0)&=  \begin{pmatrix} \frac{F v_u x}{2\mu},  &  \frac{-F b(1-v_u)y}{2\mu a} \end{pmatrix}^t.\nonumber
        \end{align} 

The input parameters for Mandel's problem are listed in Table \ref{Parameter}, and the boundary conditions are specified in Table \ref{BoundaryCondition}. 
	 \begin{table}[h!]
\centering
\caption{Boundary conditions for Mandel's problem}
  \begin{tabular}{ l  l  l }
    \hline
    Boundary & Flow & Mechanics \\ 
    \specialrule{.1em}{.05em}{.05em} 
    $x=0$ & $ \qq \cdot \vec{n}=0$ & $\uu \cdot  \vec{n}=0$ \\
    $y=0$ & $ \qq \cdot \vec{n}=0$ & $\uu \cdot  \vec{n}=0$ \\
    $x=a$ & $p=0$                  & $\vec{\sigma} \cdot  \vec{n}=0$ \\
    $y=b$ & $ \qq \cdot \vec{n}=0$ & $\vec{\sigma}_{12}=0$; $\uu \cdot \vec{n} = U_y(b,t)$  \\
\hline
  \end{tabular}
  \label{BoundaryCondition}
	\end{table}  

\begin{table}[h!]
\centering
\caption{Input parameter for Mandel's problem}
  \begin{tabular}{ l  l  l }
    \hline
    Symbol & Quantity & Value \\ 
    \specialrule{.1em}{.05em}{.05em} 
    a & Dimension in $x$ & 100 m \\
    b & Dimension in $y$ & 10 m \\
    K & Permeability  & 100 D \\
    $\mu_f$ & Dynamic viscosity  & 10 cp \\
    $\alpha$ &  Biot's constant  & 1.0 \\
    M & Biot's modulus & $1.65\times 10^{10}$ Pa \\
    $\mu$& Lame coefficients & $2.4750\times 10^{9}$  \\
    $\lambda$& Lame coefficients & $ 1.6500\times 10^{9}$  \\
    $\Delta x$& Grid spacing  in $x$ & 2.5 m \\
    $\Delta y$& Grid spacing  in $y$ & 0.25 m\\ 
    $\Delta t$& Time step & 1 s\\ 
    $t_T$& Total simulation time  & 500 s\\ 
    \hline
  \end{tabular}
  \label{Parameter}
	\end{table}

 In Figure \ref{solution}, the solution of the system for the variables pressure and displacement is depicted. This implementation demonstrates the Mandel-Cryer effect, first showing a pressure raise during the first 20 seconds and then, a sudden dissipation throughout the domain. 
	%
\begin{figure}[t]
    \centering
    \begin{subfigure}{0.5\textwidth}
        \centering
        \includegraphics[scale=.3]{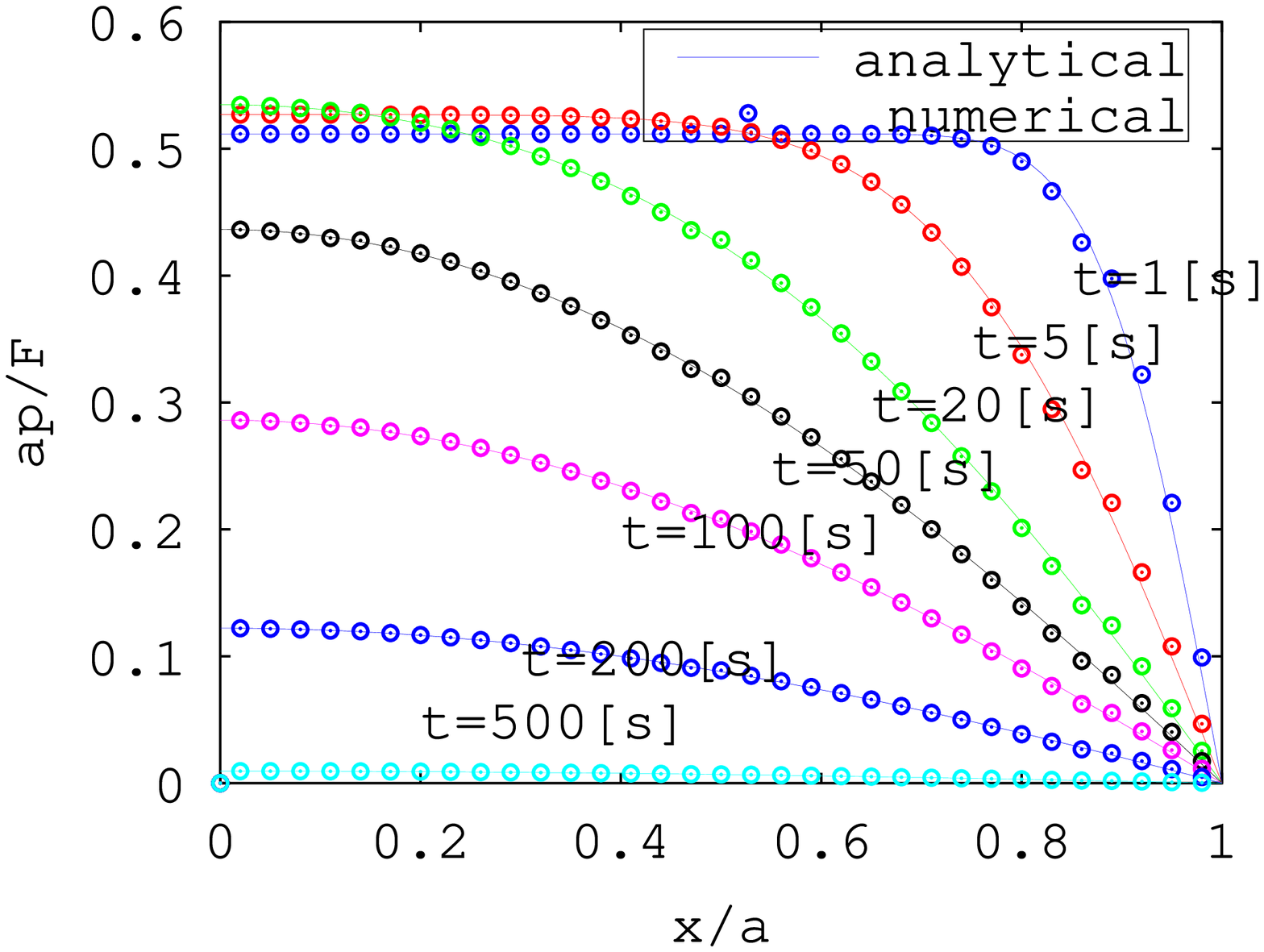} 
        \caption{Pressure solution}
                \label{solutiona}
    \end{subfigure}%
    ~ 
    \begin{subfigure}{0.5\textwidth}
        \centering
      \includegraphics[scale=0.3]{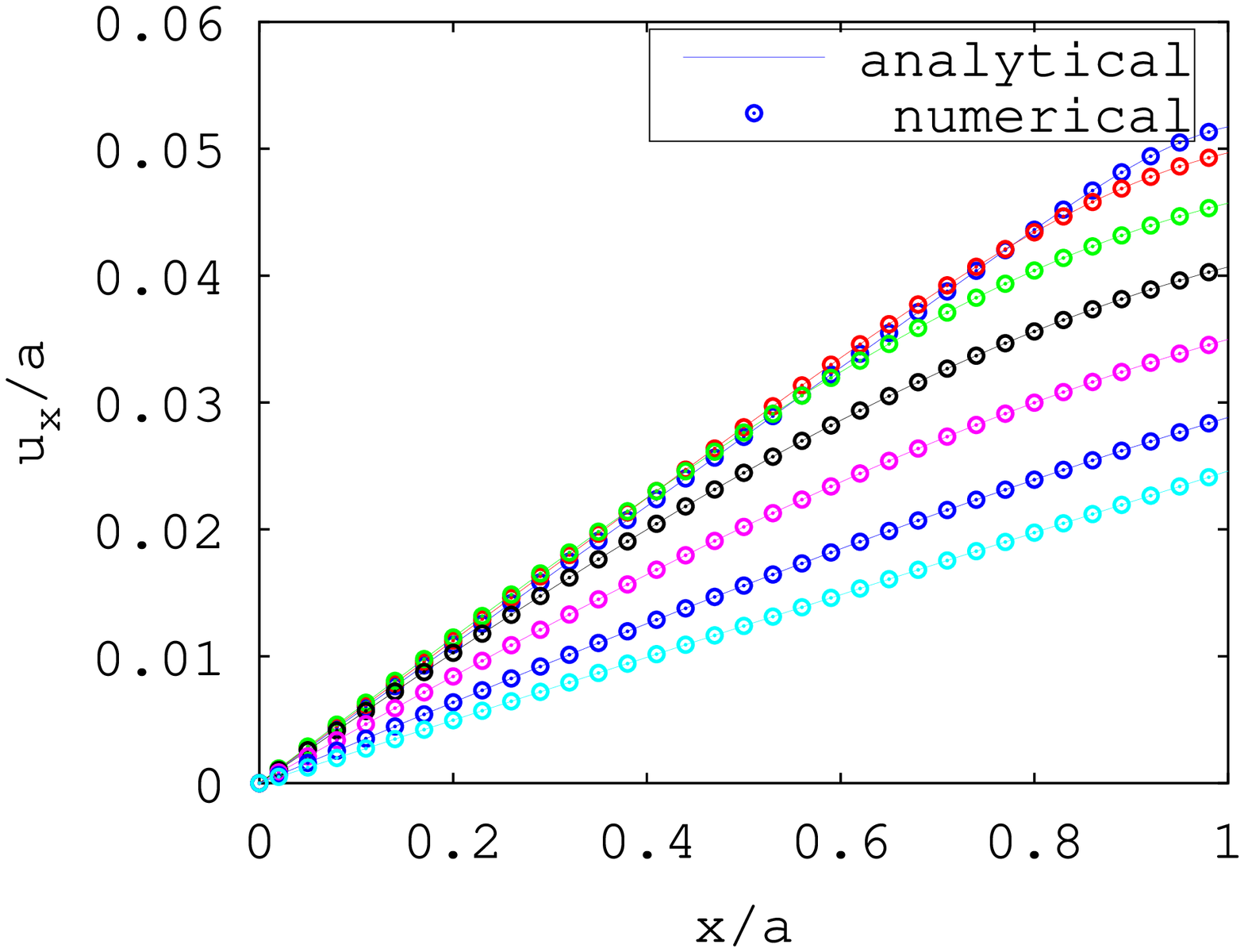}
        \caption{Displacement solution}
        \label{solutionb}
    \end{subfigure}
    \caption{Pressure and displacement match results for Mandel's problem}
            \label{solution}
\end{figure}

We consider now a non-linear extension of Mandel's problem. We use the same parameters, boundary and initial conditions as in the linear case above  (see Table \ref{Parameter}). We propose different coefficient functions $b(\cdot)$  and $h(\cdot)$ to study the performance of the proposed schemes. 

 \begin{table}[b]
 \begin{center}
 \label{Nonlinearfuntions2}
\caption{Cases for test problem 2}
 \begin{tabular}{lc c}
\hline
Case & $b(p)$& $h(\divv{\uu})$\\[0.2cm]
\hline
1  &$\frac{p + p^{3}}{M}$& $ \lambda  \divv{\uu} +  \lambda (\divv{\uu})^{3} $ \\[0.2cm]
2 &$ \frac{p + \sqrt[3]{p}}{M} $& $ \lambda  \divv{\uu} +  \lambda \sqrt[3]{(\divv  \uu)^5} $ \\[0.2cm]
 3 &$ \frac{e^p}{M} $&$ \lambda  \divv{\uu} +  \lambda \sqrt[3]{(\divv  \uu)^5} $ \\
\hline
\end{tabular}
 \end{center}
\end{table}

The influence of nonlinearities on the convergence is smaller for the test problem 1, see Figures \ref{MFL1L2_iterations}-\ref{MFL1L2_iterations3}. Moreover, both schemes are even more sensitive on the choice of the tuning parameter $L_2$.  We point out that the schemes are converging also for $L_2 = 0$, but the convergence is slow (around 200 iterations). The schemes are performing similarly.

Figures \ref{dh_3} - \ref{AFL1L2_iterations10} illustrates the influence of the mesh size, time step, value of permeability $K$ and Biot's coupling coefficient on the convergence of the schemes for test problem 2. We observe, again according to the theory, that a higher time step or higher permeability imply a faster convergence. The schemes are converging faster when the time step increases and that the convergence is not depending of the mesh diameter, confirming again the theory.
  We remark a slight increase in the number of iterations for an increasing Biot coefficient $\alpha$, see Figures \ref{AFL1L2_iterations3} and \ref{AFL1L2_iterations10}.

\begin{figure}[h!]
\centering
 \begin{subfigure}{0.5\textwidth}
 \centering
\includegraphics[scale=0.3,trim={0 0 0 0},clip]{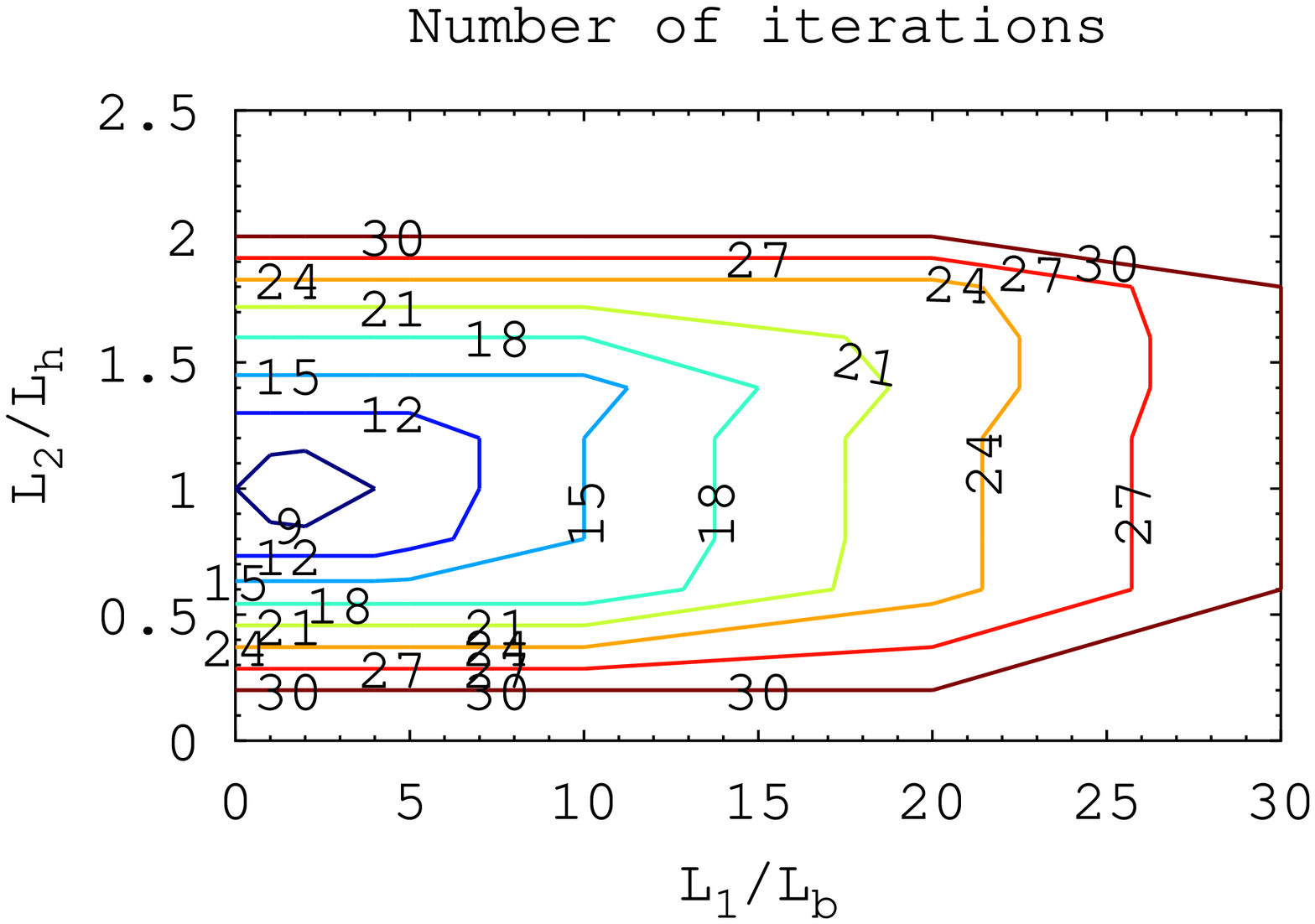} 
        \caption{Splitting}
       \label{MFL1L2_SimpleNoLineal1}
    \end{subfigure}%
       ~ 
    \begin{subfigure}{0.5\textwidth}
    \centering
\includegraphics[scale=0.3]{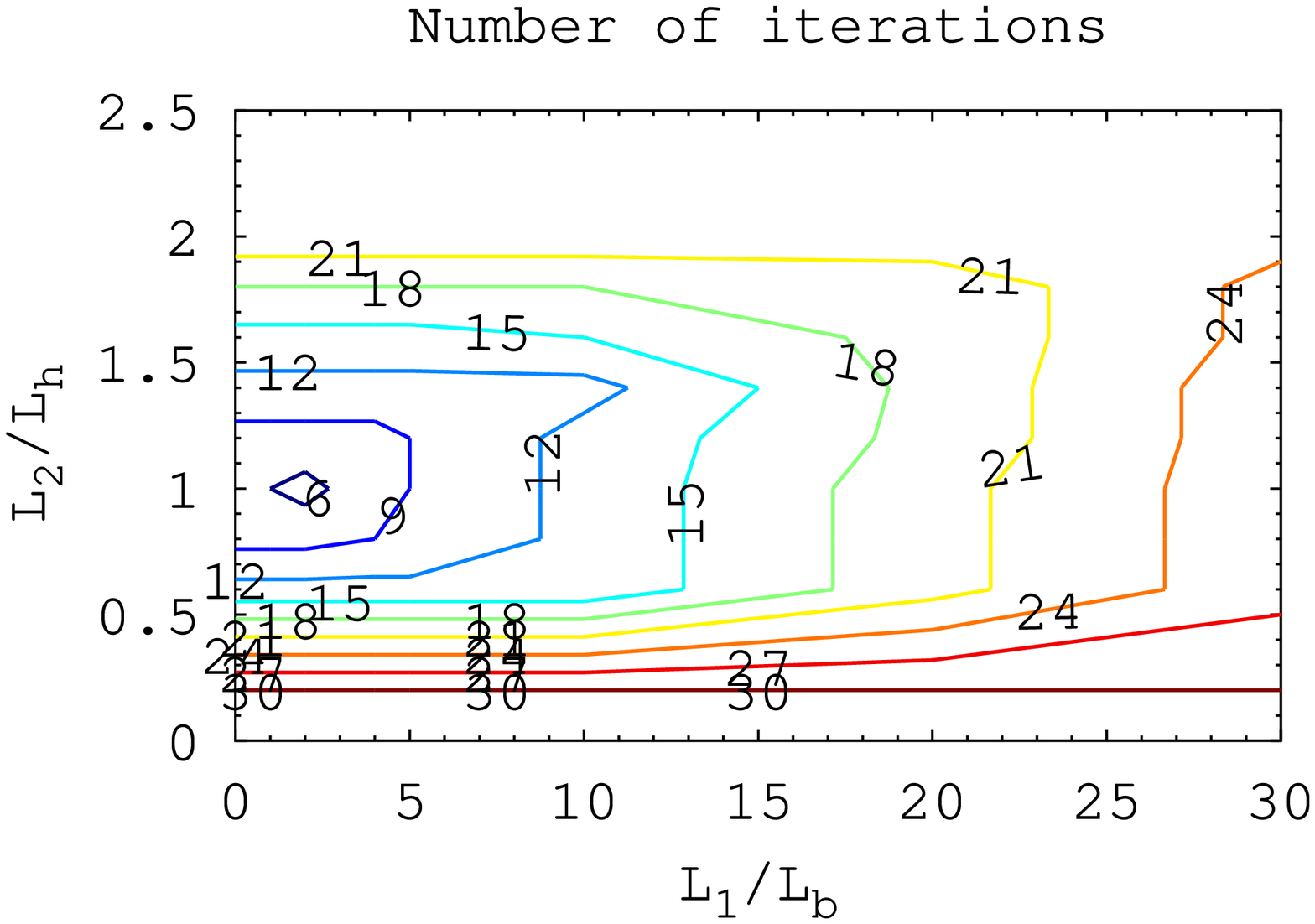}      
   \caption{Monolithic}
      \label{MFL1L2_SimpleNoLineal2}
    \end{subfigure}    
    \caption{Performance of the iterative schemes for different values of $L_1$ and $L_2$ for  test problem 2, case 1:  $b(p)=\frac{p+p^3}{M};\ $ $h(\divv \uu)=\lambda\divv \uu+\lambda(\divv \uu) ^3$.}
    	    \label{MFL1L2_iterations}
\end{figure}	 

\begin{figure}[h!]
\centering
 \begin{subfigure}{0.5\textwidth}
 \centering
\includegraphics[scale=0.3,trim={0 0 0 0},clip]{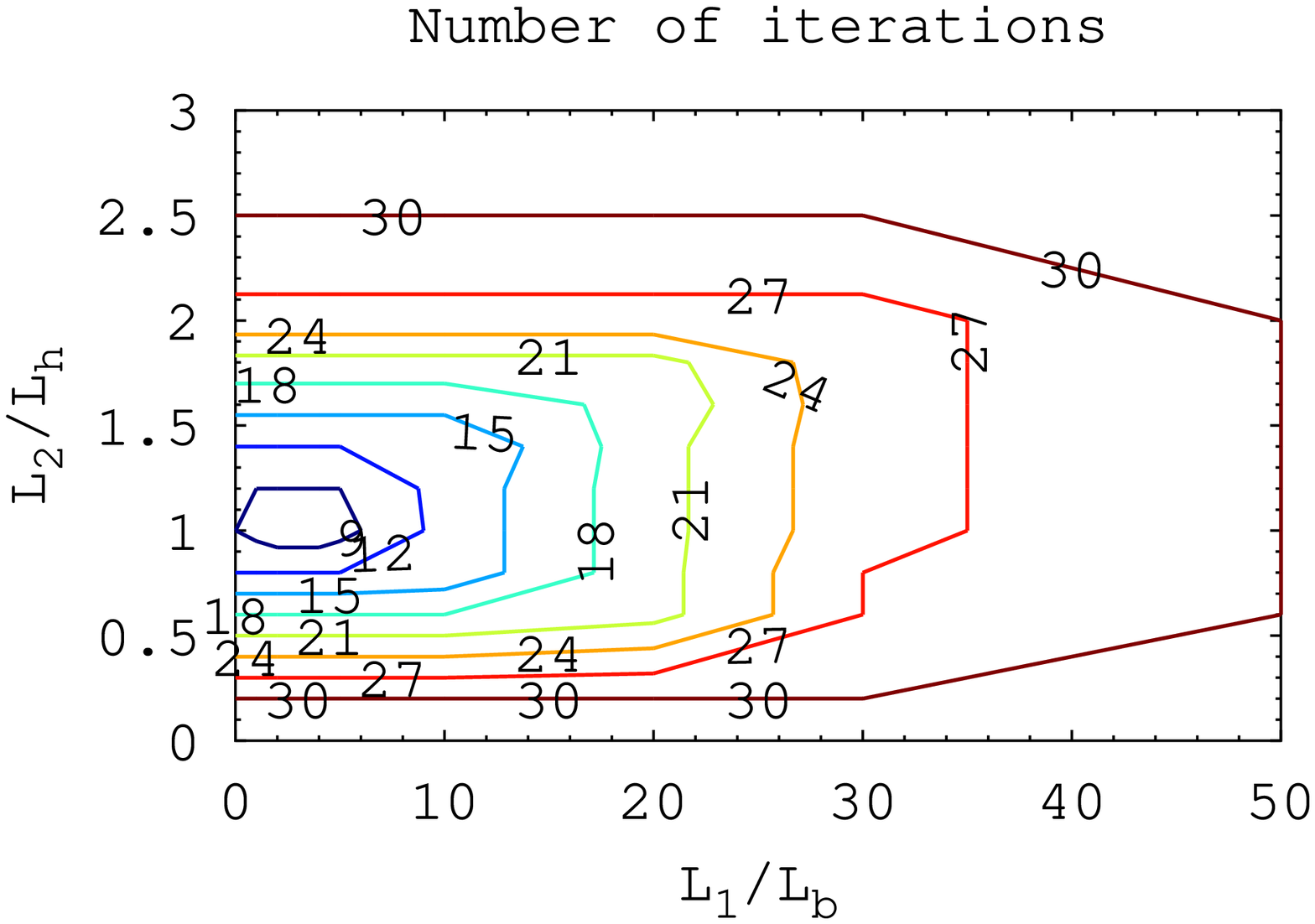} 
        \caption{Splitting}
       \label{MFL1L2_SimpleNoLineal1}
    \end{subfigure}%
       ~ 
    \begin{subfigure}{0.5\textwidth}
    \centering
\includegraphics[scale=0.3]{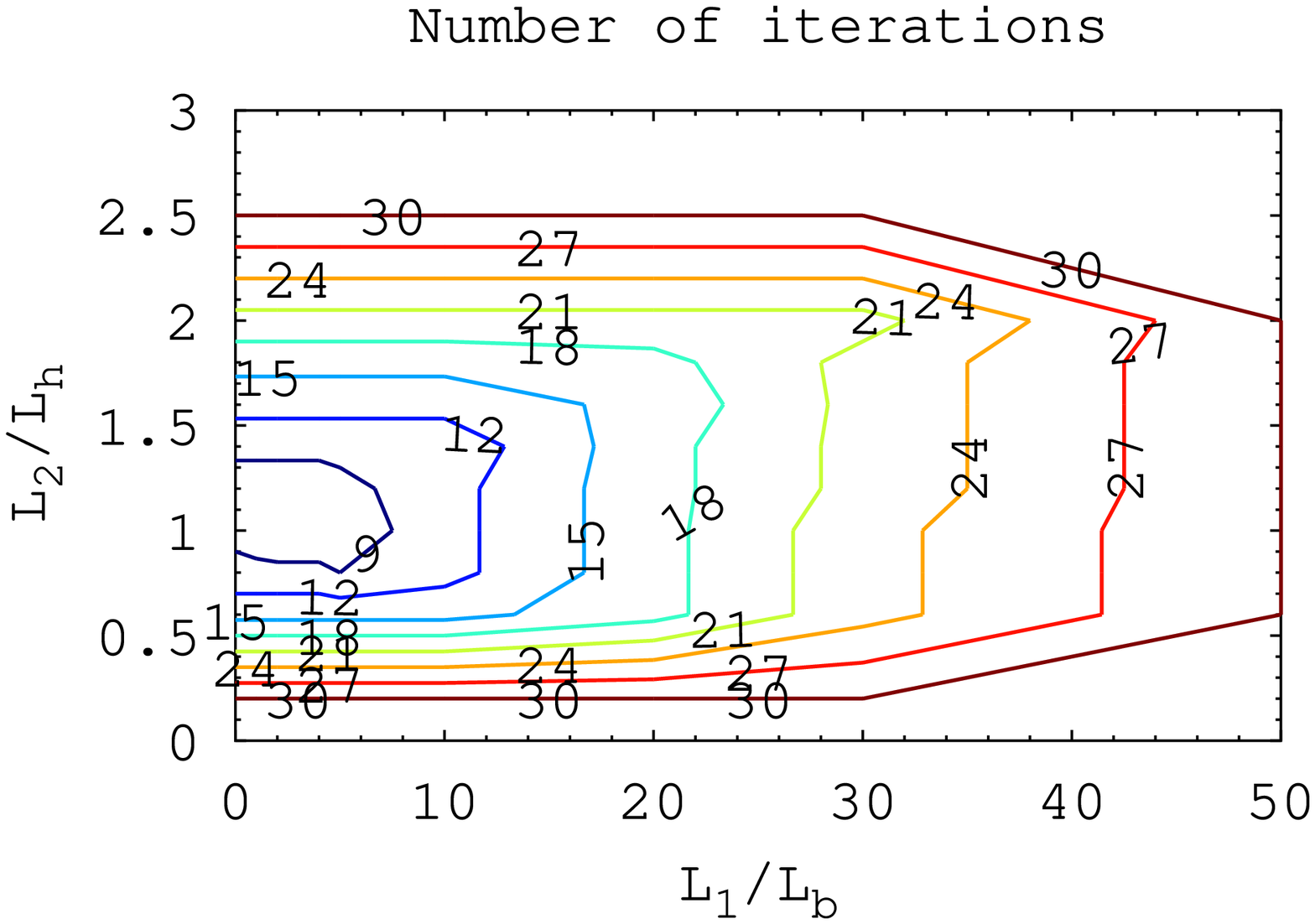}      
   \caption{Monolithic}
      \label{MFL1L2_SimpleNoLineal2}
    \end{subfigure}    
    \caption{Performance of the iterative schemes for different values of $L_1$ and $L_2$ for  test problem 2, case 2:  $b(p)=\frac{p+\sqrt[3]{p}}{M};\ $ $h(\divv \uu)=\lambda \divv \uu+\lambda \sqrt[3]{(\divv \uu)^5}$.}
    	    \label{MFL1L2_iterations2}
\end{figure}

\begin{figure}[h!]
\centering
 \begin{subfigure}{0.5\textwidth}
 \centering
\includegraphics[scale=0.3,trim={0 0 10 0},clip]{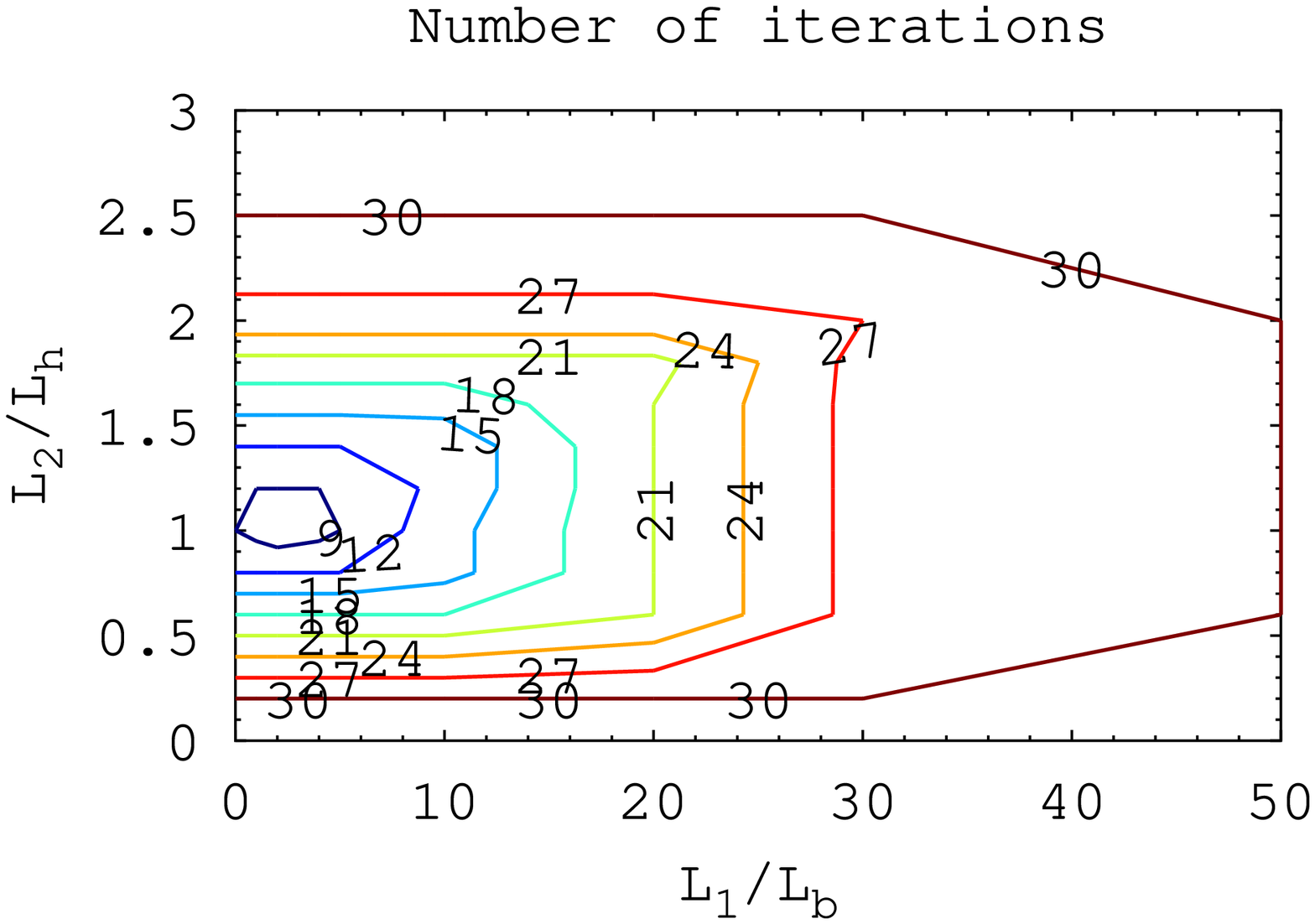} 
        \caption{Splitting}
       \label{MFL1L2_SimpleNoLineal1}
    \end{subfigure}%
       ~ 
    \begin{subfigure}{0.5\textwidth}
    \centering
\includegraphics[scale=0.3]{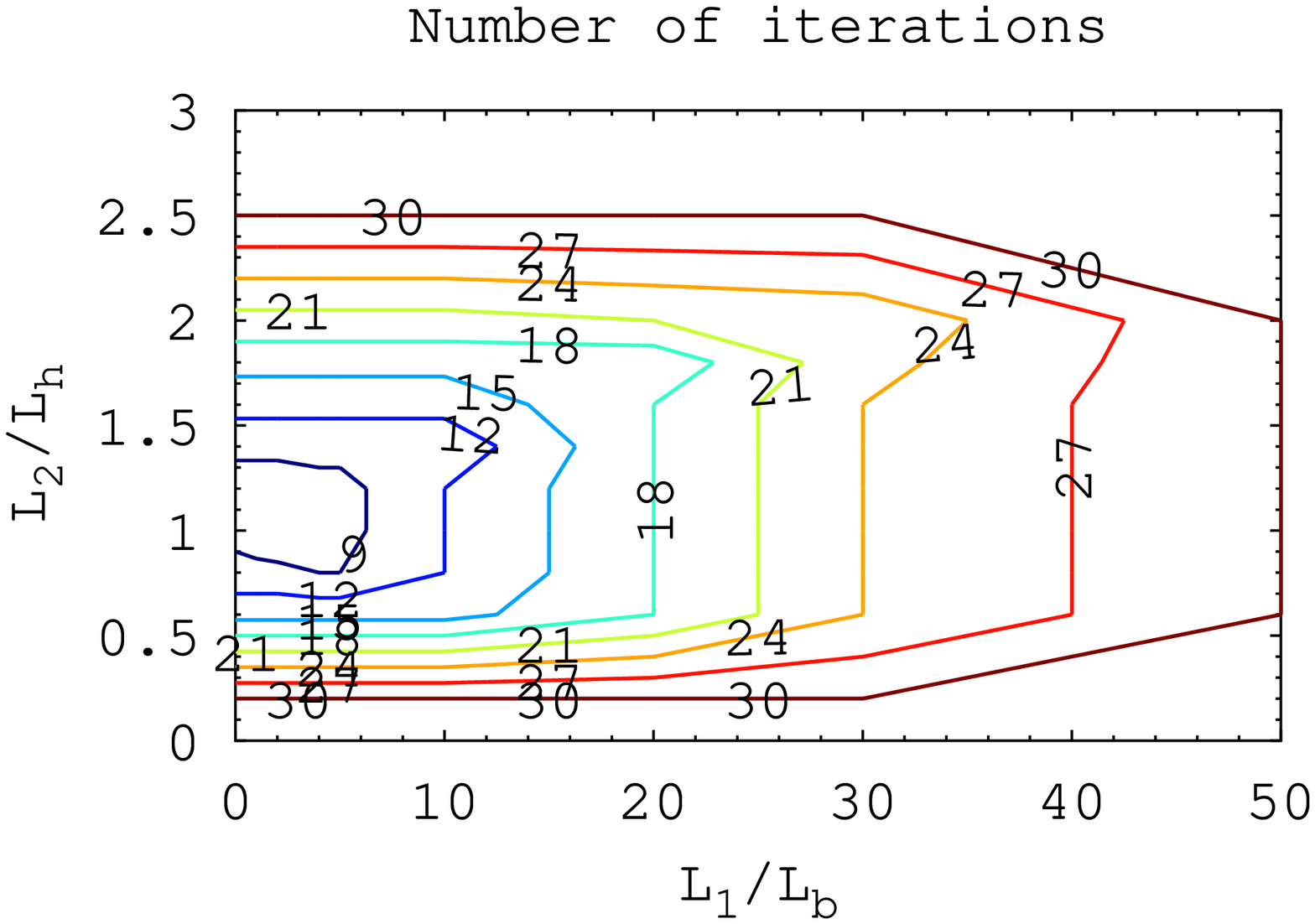}      
   \caption{Monolithic}
      \label{MFL1L2_SimpleNoLineal2}
    \end{subfigure}    
    \caption{Performance of the iterative schemes for different values of $L_1$ and $L_2$ for  test problem 2, case 3: $b(p)=\frac{e^p}{M};\ $ $h(\divv \uu)=\lambda \divv \uu+\lambda \sqrt[3]{(\divv \uu)^5}$.}
    	    \label{MFL1L2_iterations3}
\end{figure}	 


\begin{figure}[h]
\centering
 \begin{subfigure}{0.5\textwidth}
 \centering
\includegraphics[scale=0.25,trim={0 0 0 0},clip]{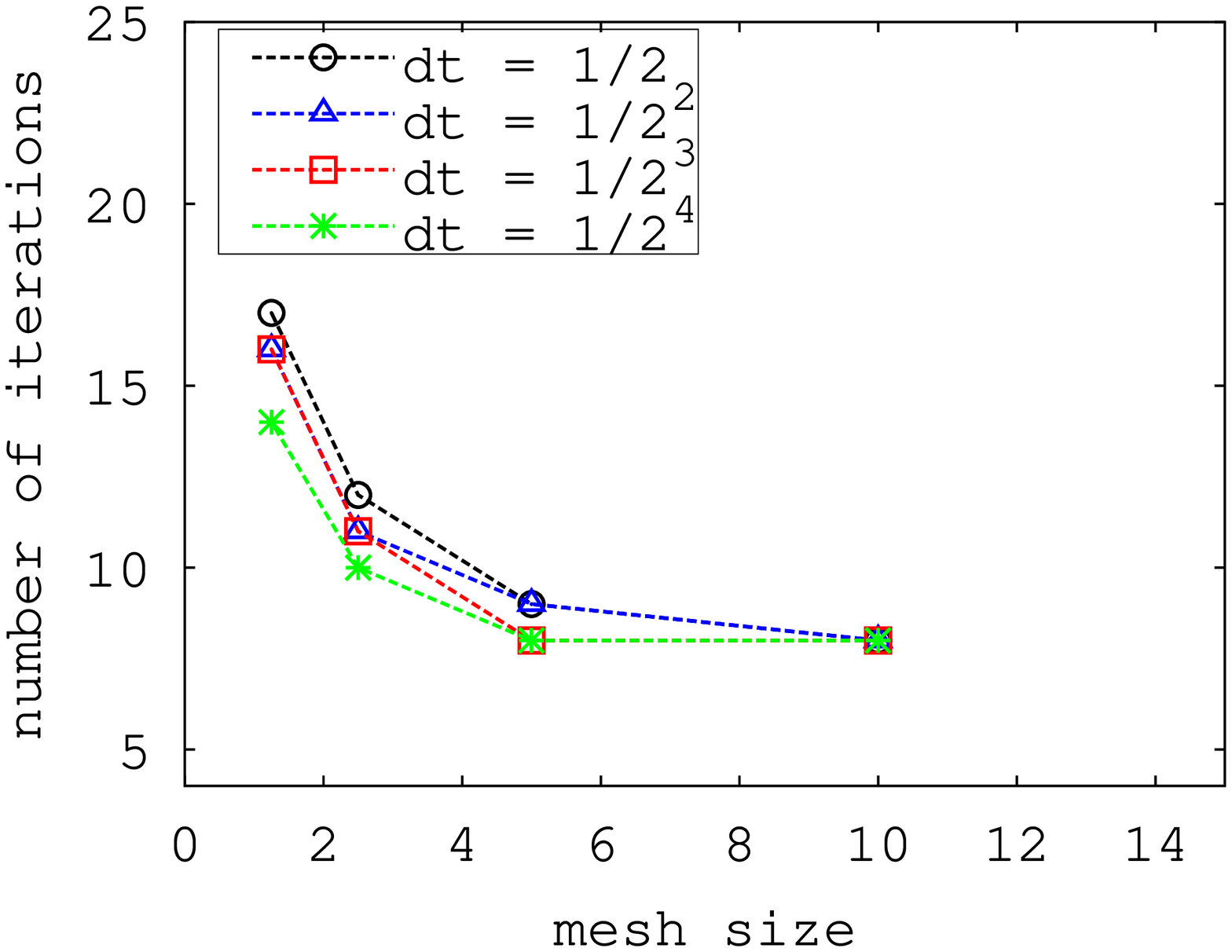} 
        \caption{Splitting}
       \label{dh_3S}
    \end{subfigure}%
       ~ 
    \begin{subfigure}{0.5\textwidth}
    \centering
\includegraphics[scale=0.25]{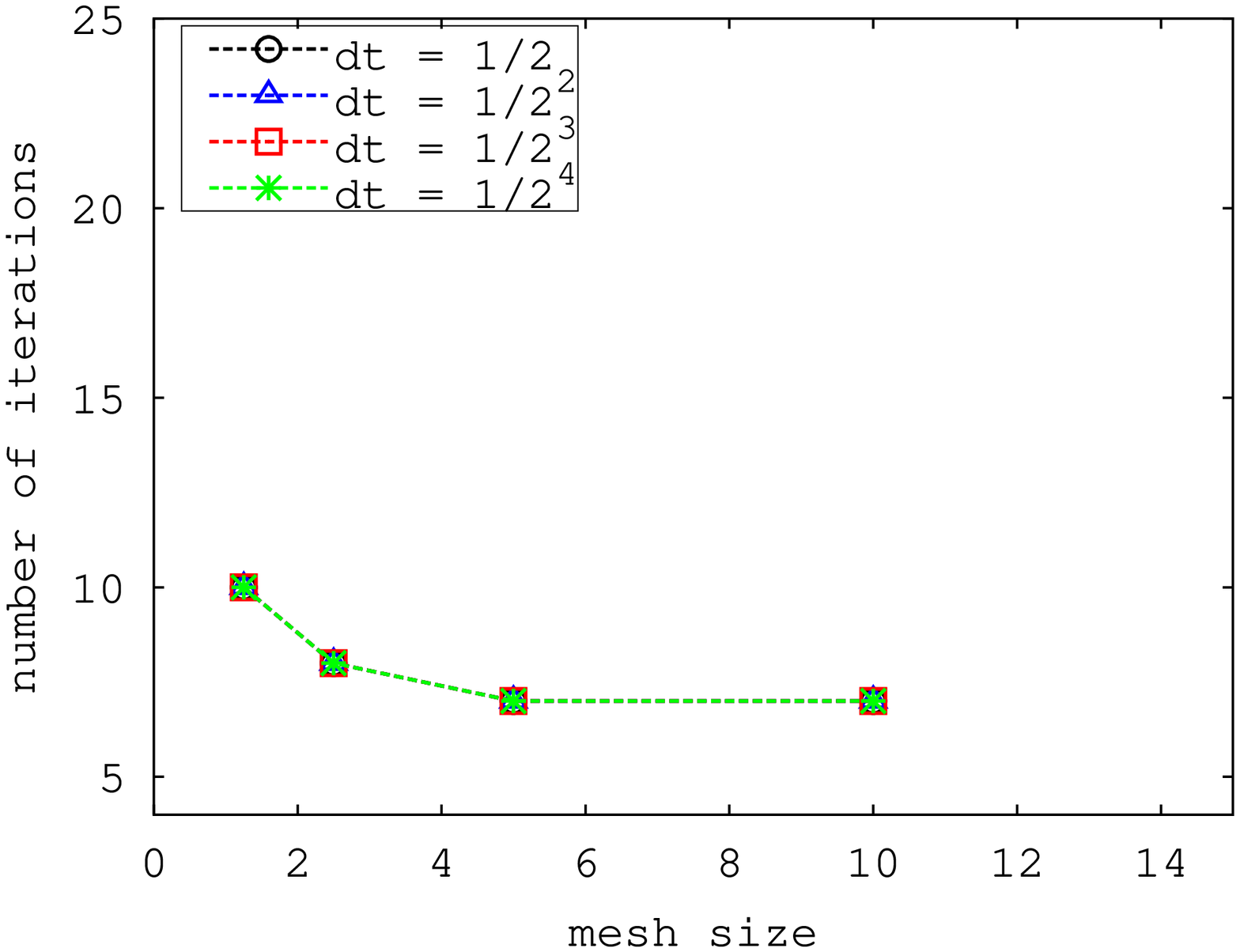}     
   \caption{Monolithic}
      \label{dh_3M}
    \end{subfigure}    
    \caption{Performance of the iterative schemes for different mesh sizes for test problem 2, case 1: $b(p)={\frac{p + p^{3}}{M}};\ $ $h(\divv \uu)=\lambda  \divv{\uu} +  \lambda (\divv{\uu})^{3}$.}
    	    \label{dh_3}
\end{figure}

\begin{figure}[h]
\centering
 \begin{subfigure}{0.5\textwidth}
 \centering
\includegraphics[scale=0.25,trim={0 0 0 0},clip]{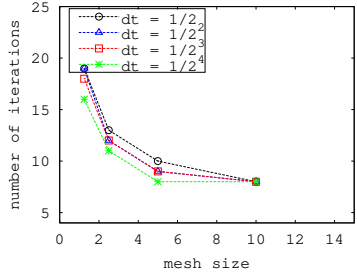} 
        \caption{Splitting}
       \label{dh_cbrt3S}
    \end{subfigure}%
       ~ 
    \begin{subfigure}{0.5\textwidth}
    \centering
\includegraphics[scale=0.25]{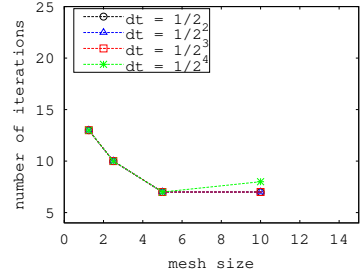}     
   \caption{Monolithic}
      \label{dh_cbrt3M}
    \end{subfigure}    
    \caption{Performance of the iterative schemes for different mesh sizes for test problem 2, case 2: $b(p)=\frac{p + \sqrt[3]{p}}{M};\ $  $h(\divv \uu)= \lambda  \divv{\uu} +  \lambda \sqrt[3]{(\divv  \uu)^5} $.}
    	    \label{dh_cbrt3}
\end{figure}

			
\begin{figure}[h]
\centering
 \begin{subfigure}{0.5\textwidth}
 \centering
\includegraphics[scale=0.25,trim={0 0 0 0},clip]{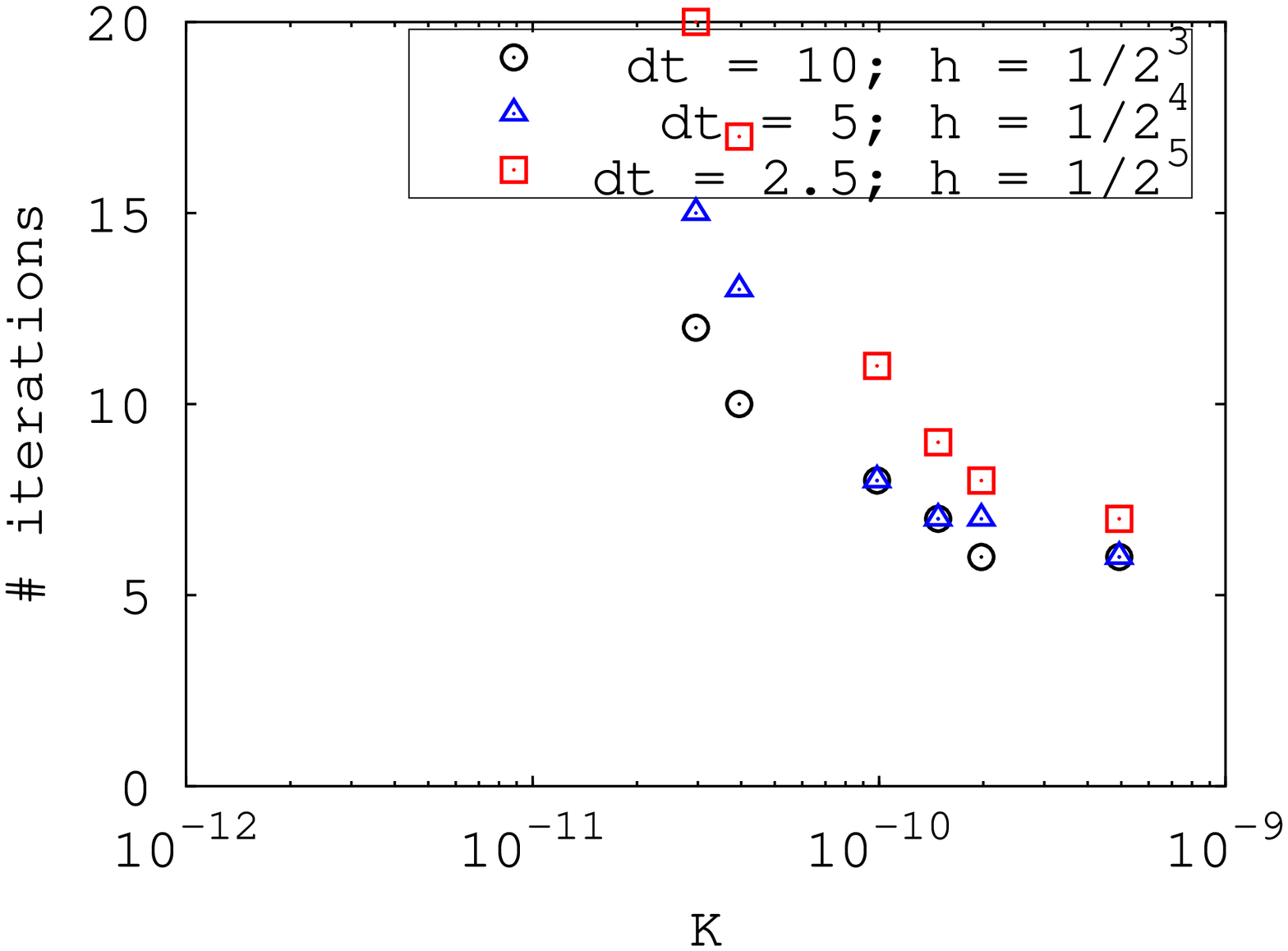} 
        \caption{Splitting}
       \label{KflorinS1}
    \end{subfigure}%
       ~ 
    \begin{subfigure}{0.5\textwidth}
    \centering
\includegraphics[scale=0.25]{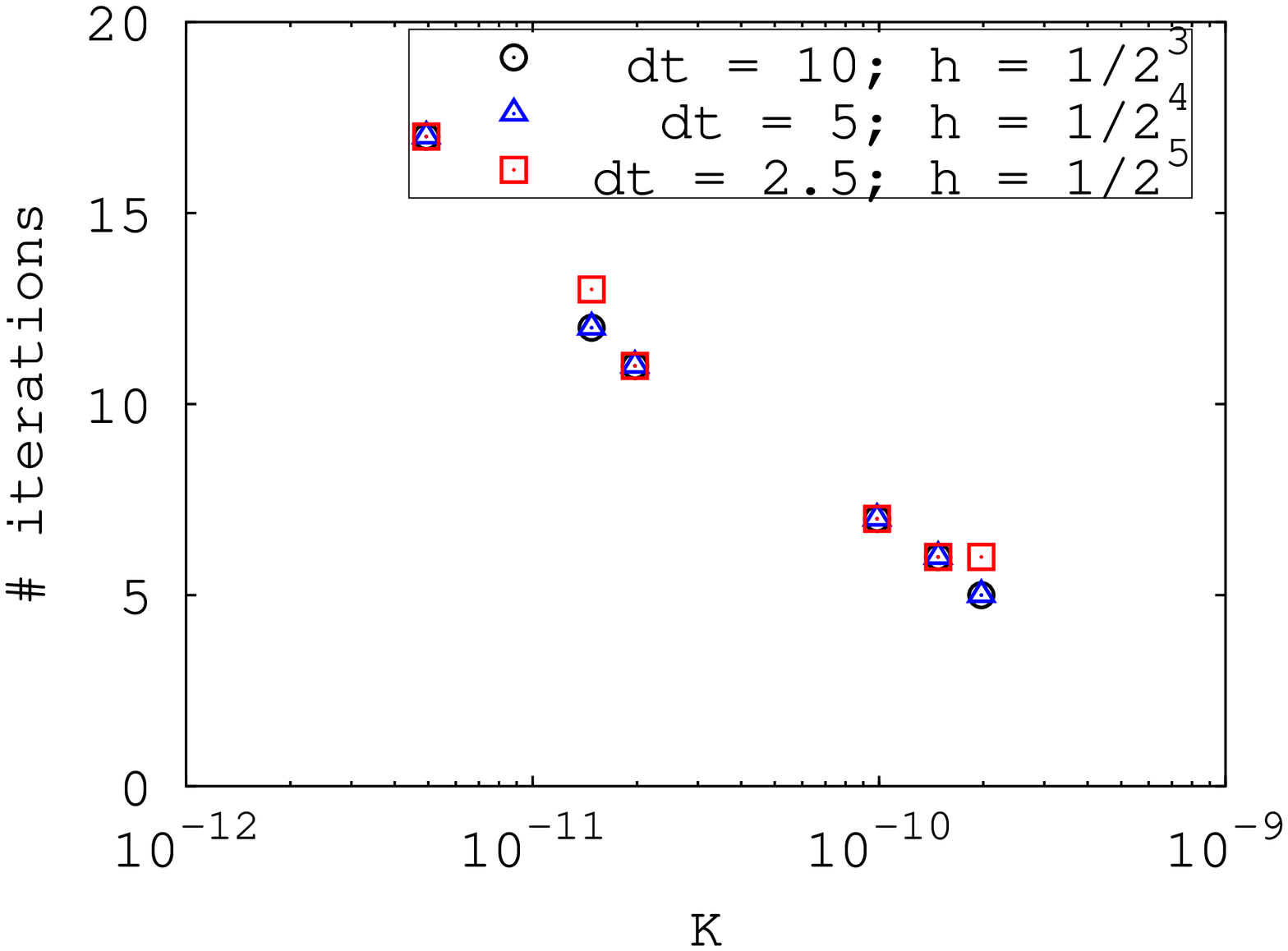}      
   \caption{Monolithic}
      \label{KflorinM1}
    \end{subfigure}    
    \caption{Number of iterations for different mesh sizes and different values of $\Delta t$, $K$  for test problem 2, case 2: $b(p)=\frac{p+\sqrt[3]{p}}{M};\ $ $h(\divv \uu)=\lambda \divv \uu+\lambda\sqrt[3]{(\divv \uu)^5})$.}
    	    \label{Kflorin1}
\end{figure}

\begin{figure}[h!]
\centering
 \begin{subfigure}{0.5\textwidth}
 \centering
\includegraphics[scale=0.25,trim={0 0 0 0},clip]{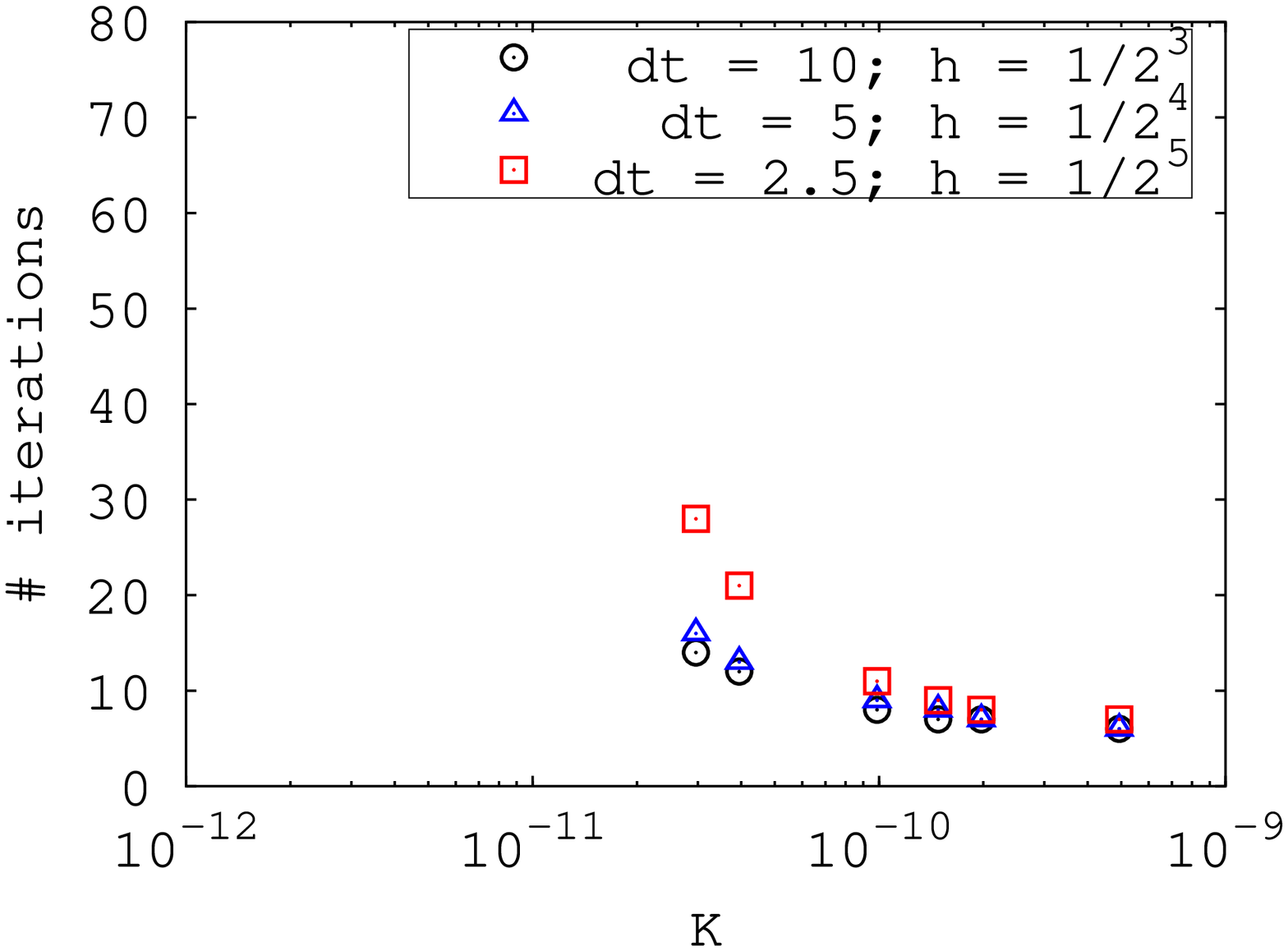} 
        \caption{Splitting}
       \label{KflorinS}
    \end{subfigure}%
       ~ 
    \begin{subfigure}{0.5\textwidth}
    \centering
\includegraphics[scale=0.25]{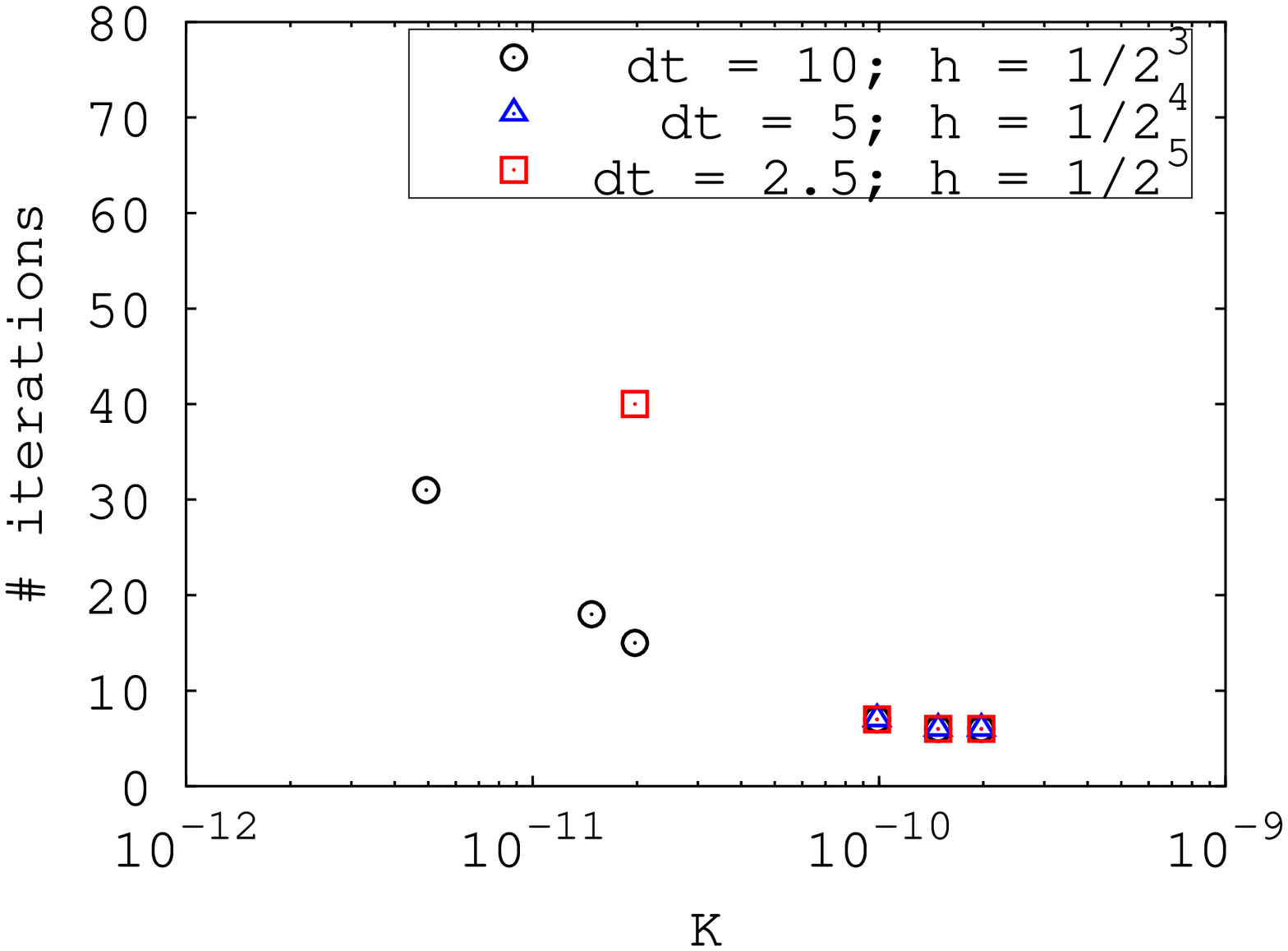}      
   \caption{Monolithic}
      \label{KflorinM}
    \end{subfigure}    
    \caption{Number of iterations for different mesh sizes and different values of $\Delta t$, $K$  for test problem 2, case 3: $b(p)=\frac{e^p}{M};\ $ $h(\divv \uu)=\lambda \divv \uu+\lambda \sqrt[3]{(\divv \uu)^5}$.}
    	    \label{Kflorin}
\end{figure}	 


\begin{figure}[h!]
\centering
 \begin{subfigure}{0.5\textwidth}
 \centering
\includegraphics[scale=0.25,trim={0 0 0 0},clip]{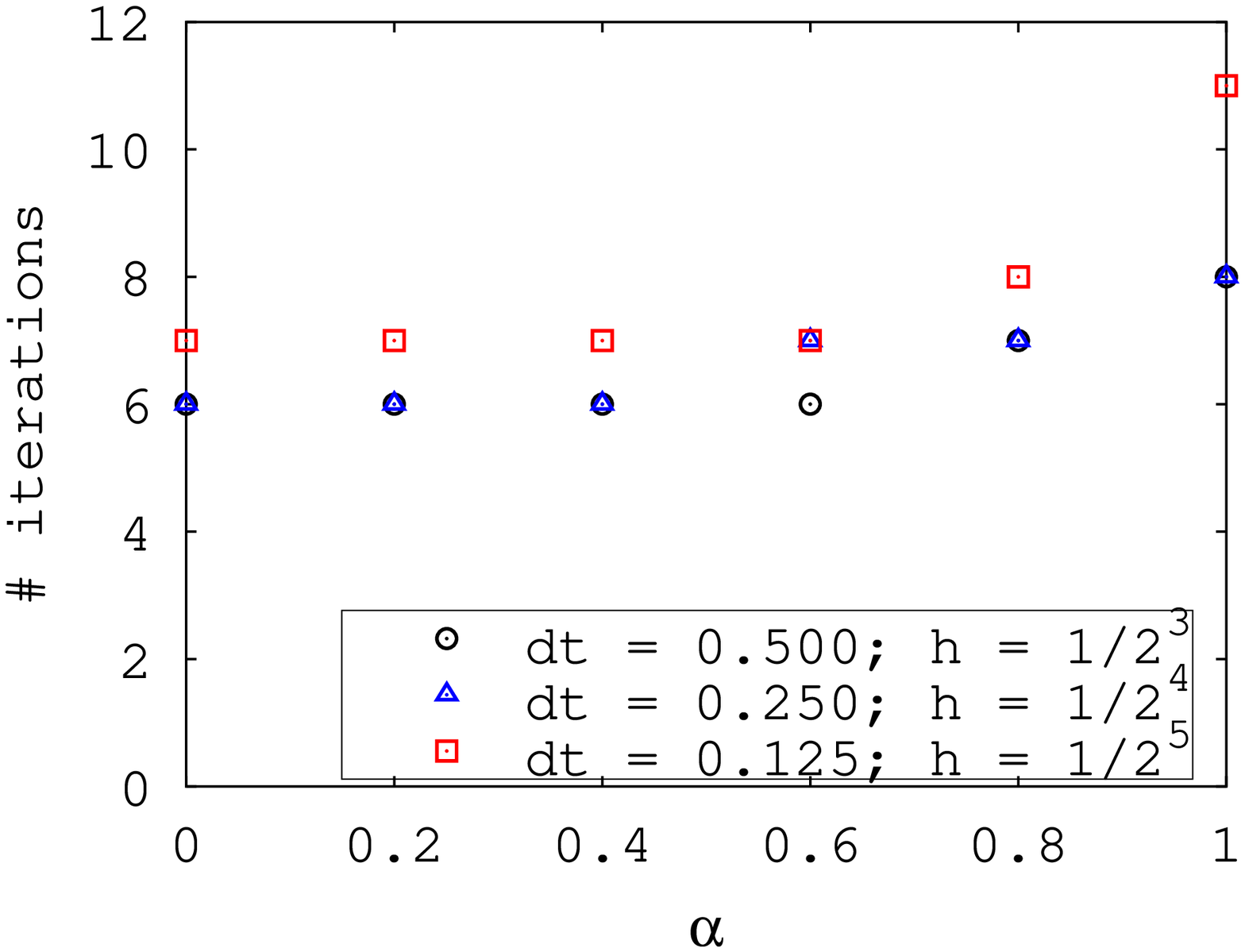} 
        \caption{Splitting}
       \label{MFL1L2_SimpleNoLineal1}
    \end{subfigure}%
       ~ 
    \begin{subfigure}{0.5\textwidth}
    \centering
\includegraphics[scale=0.25]{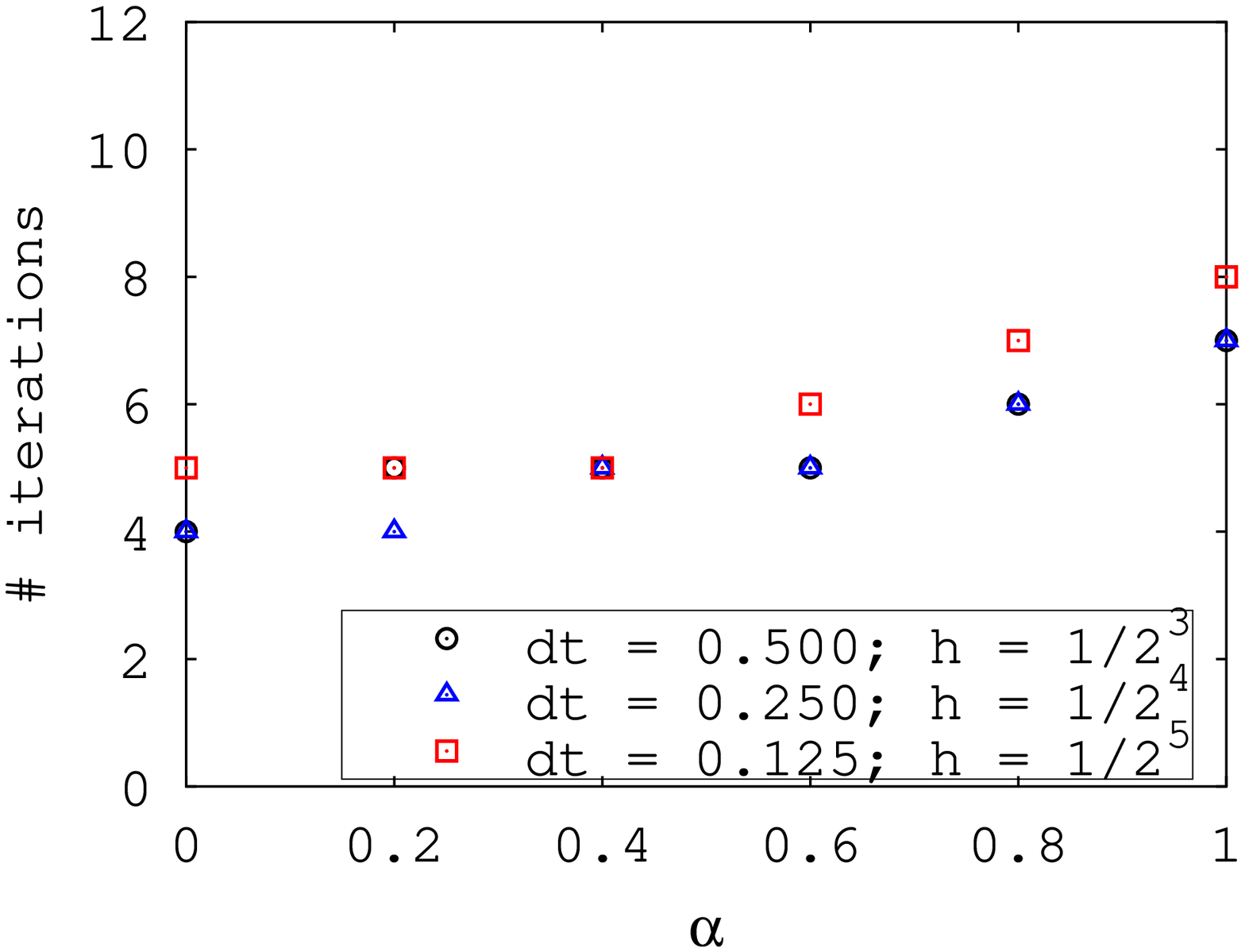}      
   \caption{Monolithic}
      \label{MFL1L2_SimpleNoLineal2}
    \end{subfigure}    
    \caption{Number of iterations for different mesh sizes and different values of $\Delta t$, $\alpha$  for test problem 2, case 2: $b(p)=\frac{p+\sqrt[3]{p}}{M};;\ $ $h(\divv \uu)=\lambda\divv \uu+\lambda \sqrt[3]{(\divv \uu)^5}$.}
    	    \label{AFL1L2_iterations3}
\end{figure}

\begin{figure}[h]
\centering
 \begin{subfigure}{0.5\textwidth}
 \centering
\includegraphics[scale=0.25,trim={0 0 0 0},clip]{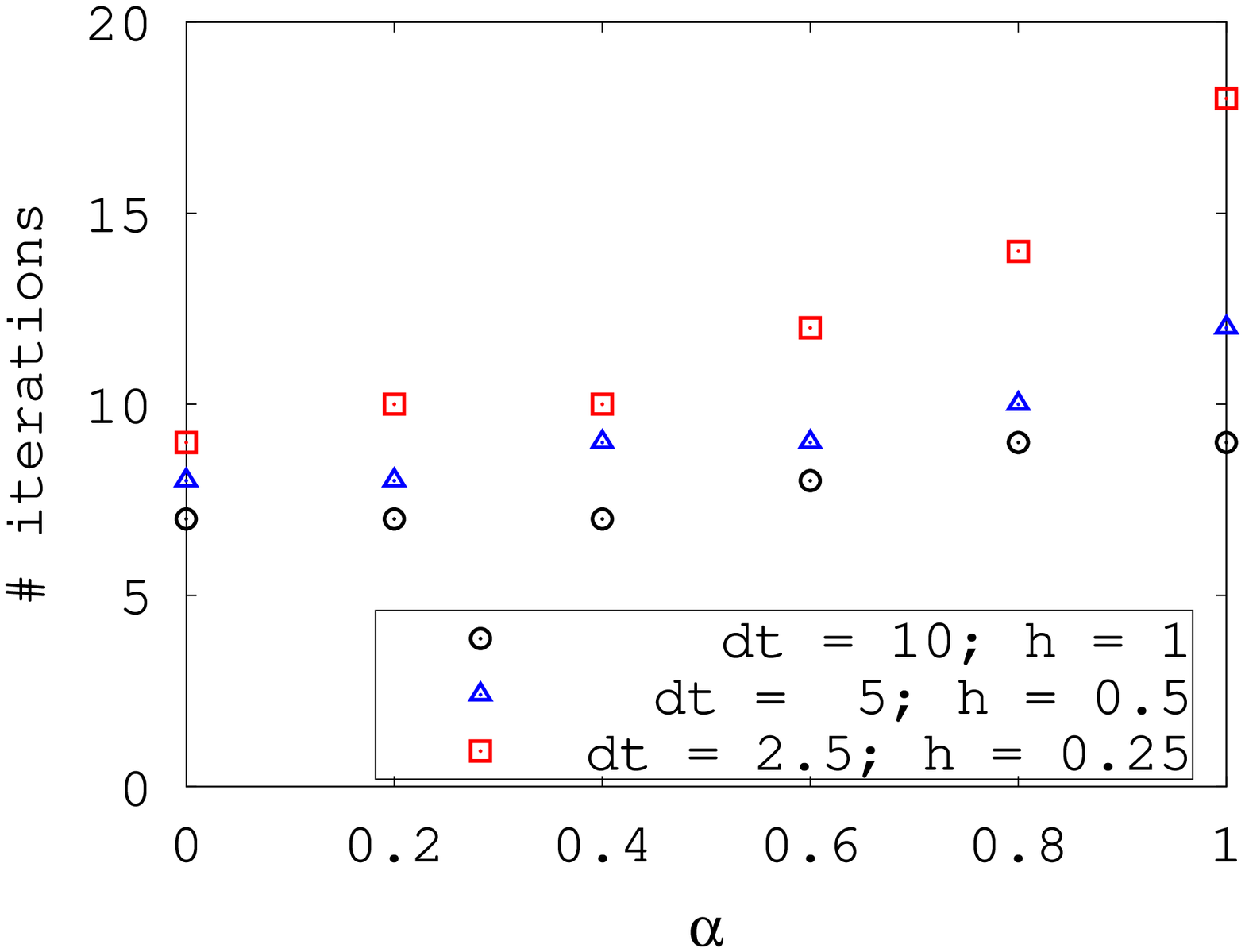} 
        \caption{Splitting}
       \label{FL1L2_SimpleNoLineal1}
    \end{subfigure}%
       ~ 
    \begin{subfigure}{0.5\textwidth}
    \centering
\includegraphics[scale=0.25]{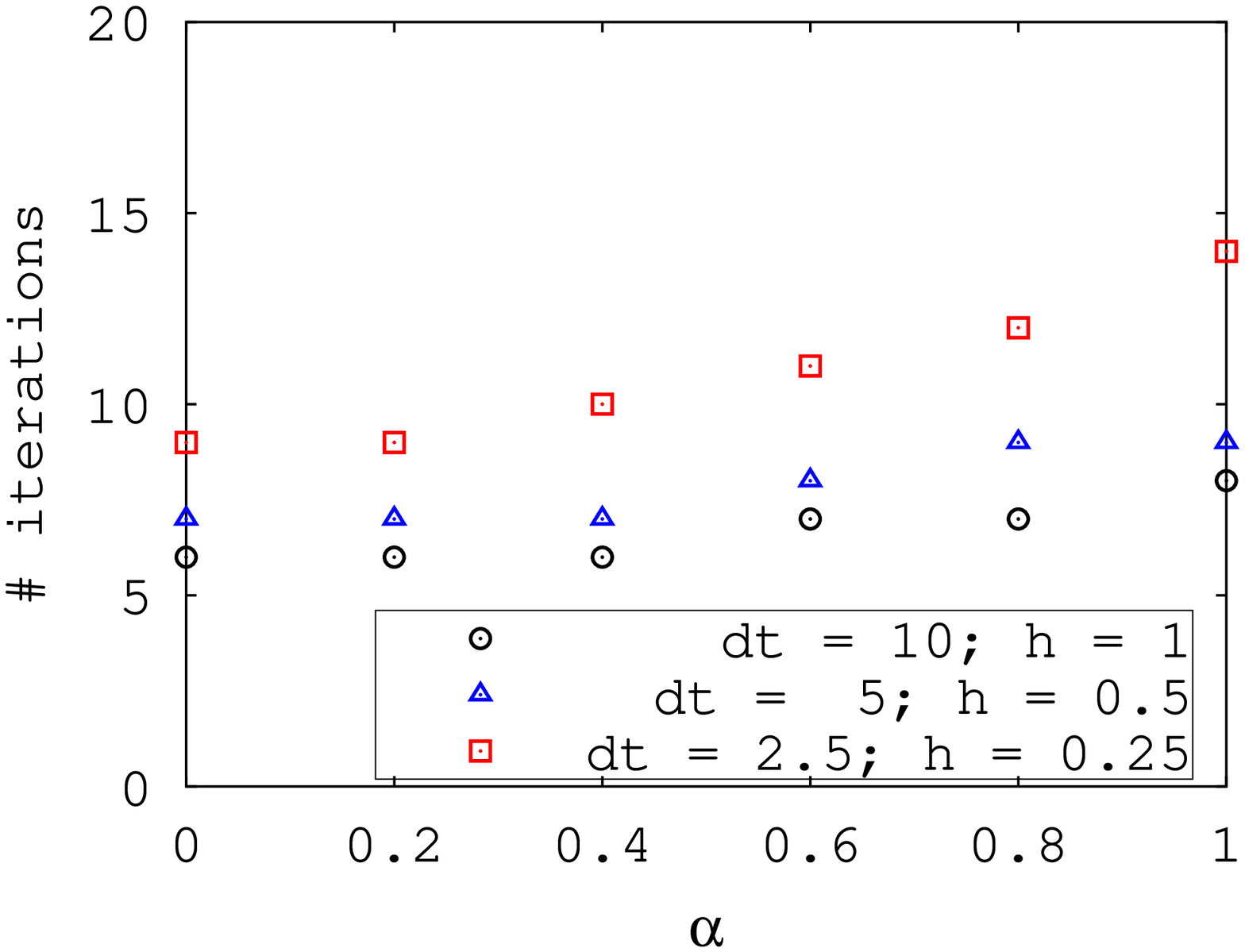}      
   \caption{Monolithic}
      \label{AFL1L2_SimpleNoLineal2}
    \end{subfigure}    
    \caption{Number of iterations for different mesh sizes and different values of $\Delta t$, $\alpha$  for test problem 2, case 3: $b(p)=\frac{e^p}{M};\ $ $h(\divv \uu)=\lambda \divv \uu+\lambda \sqrt[3]{(\divv \uu)^5}$.}
    	    \label{AFL1L2_iterations10}
\end{figure}	 


 \section{Conclusions}
 \label{sec:conclusions}

We have proposed two linearisation schemes for the non-linear Biot model, a monolithic L-scheme and a splitting L-scheme. The convergence of both schemes has been rigorously shown, similar techniques as in \cite{RaduPopKnabner2004,RaduNPK15,List2016} and \cite{Mikelic2} being involved. The schemes are linearly, but global convergent and they are not involving the computations on any derivatives (as in the case of the Newton method). Two illustrative numerical examples, an academic one and a non-linear extension of Mandel's problem were implemented for testing the performance of the schemes. To summarise, we make the following remarks:
\begin{itemize}
\item[$\bullet$]{Both schemes are very robust with respect to the choice of the tuning parameter, the mesh size and time step size.}
\item[$\bullet$]{The tuning parameters $L_1, L_2$ have a strong influence on the speed of the convergence, with $L_2$ being the dominant one.}
\item[$\bullet$]{The two schemes (splitting and monolithic) performed similarly with respect to parameters $L_1, L_2$, CPU time and number of iterations when no preconditioning was applied. When preconditioned, the splitting L-scheme is faster. }
\item[$\bullet$]{The splitting L-scheme can be used both as a robust solver or as a preconditioner to improve the performance of a monolithic solver}
\item[$\bullet$]{The convergence of the schemes is faster for higher permeability. }
\item[$\bullet$]{The convergence of the schemes is almost independent of the mesh size, and varies only slightly with the Biot coupling parameter.}
\end{itemize}

According to the results in Section \ref{sec:convergence}, we should observe a faster convergence for the pressure variable when we increase the time step size.
Nevertheless, there is no indication on how the time step size is affecting the coupled problem. Numerically, we observed decreasing the time step only for the Mandel's problem. It was not observed a clear tendency for example one.




\nocite{Explicit1}
\nocite{Explicit4}
\nocite{Fullycoupled3}
\nocite{Jha2007}
\nocite{Fullycoupled1}
\nocite{Explicit2}
\nocite{Biotsttlement}
\nocite{Chin}
\nocite{Coussy1989}
\nocite{Fung}
\nocite{Girault}
\nocite{Kim20111591}
\nocite{Kim20111591}
\nocite{mikelic2013phase}
\nocite{SHOWALTER_2016}
\nocite{Settari}


\nocite{Settari_1998}
\nocite{Tchelepi_2015}
\nocite{Ferronato_2016}
\nocite{Kundan_2016}


\nocite{Jan_2016}
\nocite{Carmen_2016}
\nocite{SHOWALTER_2016}
\nocite{RaduPopKnabner2004}
\nocite{List2016}
\nocite{RaduNPK15}
 	%
\nocite{Tchelepi_2016_Preconditionet}
\nocite{Naga_2010_preconditioner}

\nocite{Uwe3}

\section*{Acknowledgement}
The research was supported by the University of Bergen in cooperation with the FME-SUCCESS center (grant 193825/S60) funded by the Research Council of Norway.  The work has also been partly supported by the NFR-DAADppp grant 255715 and the NFR-Toppforsk project 250223.

%
\bibliographystyle{spmpsci}      
\bibliographystyle{plain}

\begin{thebibliography}{10}
\providecommand{\url}[1]{{#1}}
\providecommand{\urlprefix}{URL }
\expandafter\ifx\csname urlstyle\endcsname\relax
  \providecommand{\doi}[1]{DOI~\discretionary{}{}{}#1}\else
  \providecommand{\doi}{DOI~\discretionary{}{}{}\begingroup
  \urlstyle{rm}\Url}\fi
\bibitem{Abousleiman}
Abousleiman, Y., Cheng, A.H.D., Cui, L., Detournay, E., Roegiers, J.C.:
  {Mandel's problem revisited}.
\newblock G{\'e}otechnique \textbf{46}(2), 187--195 (1996).

\bibitem{Kundan_2016}
Almani, T., Kumar, K., Dogru, A.H., Singh, G., Wheeler, M.F.: {Convergence
  Analysis of Multirate Fixed-Stress Split Iterative Schemes for Coupling Flow
  with Geomechanics}.
\newblock Comput. Methods. Appl. Mech. Eng. \textbf{311},
  180--207 (2016).

\bibitem{Explicit1}
Armero, F.: {Formulation and finite element implementation of a multiplicative
  model of coupled poro-plasticity at finite strains under fully saturated
  conditions}.
\newblock Comput. Methods. Appl. Mech. Eng. \textbf{171}(3), 205--241 (1999).

\bibitem{Explicit4}
{Armero F.}, {Simo J. C.}: {A new unconditionally stable fractional step method
  for non-linear coupled thermomechanical problems}.
\newblock Int. J. Numer. Meth. Eng. \textbf{35}(4), 737--766 (1992).


\bibitem{DealII}
Bangerth, W., Kanschat, G., Heister, T.: {deal.II Differential equations
  analysis library} (2014).

\bibitem{Uwe3}
Bause, M., Radu, F.A., Kocher, U.: {Space-time finite element approximation of
  the Biot poroelasticity system with iterative coupling}.
\newblock ArXiv:1611.06335 (2016).

\bibitem{Biotsttlement}
Biot, M.A.: {Consolidation Settlement under a rectangular load distribution}.
\newblock J. Appl. Phys. \textbf{12}(5), 426--430 (1941).

\bibitem{biot1941general}
Biot, M.A.: {General theory of three-dimensional consolidation}.
\newblock J. Appl. Phys. \textbf{12}(2), 155--164 (1941).

\bibitem{biot1954}
Biot, M.A.: {Theory of Elasticity and Consolidation for a Porous Anisotropic
  Solid}.
\newblock J. Appl. Phys. \textbf{26}(2), 182--185 (1955).

\bibitem{Jakub_2016}
Both, J.W., Borregales, M., Nordbotten, J.M., Kumar, K., Radu, F.A.: {Robust
  fixed stress splitting for Biot{\rq}s equations in heterogeneous media}.
\newblock Appl. Math. Letters \textbf{68}, 101--108 (2017).

\bibitem{brezzi2012mixed}
Brezzi, F., Fortin, M.: {Mixed and hybrid finite element methods},
  \emph{{Springer Ser. Comput. Math}}, vol.~15.
\newblock Springer-Verlag New York (2012).

\bibitem{Ferronato_2016}
Castelletto, N., White, J.A., Ferronato, M.: {Scalable algorithms for
  three-field mixed finite element coupled poromechanics}.
\newblock J. Comput. Phys. \textbf{327}, 894--918 (2016).

\bibitem{Tchelepi_2015}
Castelletto, N., White, J.A., Tchelepi, H.A.: {Accuracy and convergence
  properties of the fixed-stress iterative solution of two-way coupled
  poromechanics}.
\newblock Int. J. Numer. Anal. Meth. Geomech. \textbf{39}(14), 1593--1618
  (2015).


\bibitem{Chin}
Chin, L.Y., Thomas, L.K., Sylte, J.E., Pierson, R.G.: {Iterative Coupled
  Analysis of Geomechanics and Fluid Flow for Rock Compaction in Reservoir
  Simulation}.
\newblock Oil \& Gas Sci. Technol. \textbf{57}(5), 485--497 (2002).

\bibitem{Coussy1989}
Coussy, O.: {A general theory of thermoporoelastoplasticity for saturated
  porous materials}.
\newblock Trans. Por. Med. \textbf{4}(3), 281--293 (1989).

\bibitem{coussy1995mechanics}
Coussy, O.: {Mechanics of Porous Continua}.
\newblock Wiley, New York (1995).

\bibitem{detournay1993fundamentals}
Detournay, E., Cheng, A.H.D.: {Fundamentals of Poroelasticity}, vol.~2.
\newblock Pergamon Press (1993).

\bibitem{FlorianFCP}
Doster, F., Nordbotten, J.M.: {Full Pressure Coupling for Geo-mechanical
  Multi-phase Multi-component Flow Simulations}, paper 
  \newblock SPE 173232 presented at the SPE Reservoir  Simulation Symposium, Houston  (2015).

\bibitem{Fung}
Fung, L.S.K., Buchanan, L., Wan, R.G.: {Couplled Geomechanical-thermal
  Simulation For Deforming Heavy-oil Reservoirs}.
\newblock J. Can. Pet. Technol. \textbf{33}(04) (1994).

\bibitem{GayX}
Gai, X., Dean, R.H., Wheeler, M.F., Liu, R.: {Coupled Geomechanical and
  Reservoir Modeling on Parallel Computers}, paper 
  \newblock SPE 79700 presented at the SPE Reservoir  Simulation
Symposium, Houston (2003).

\bibitem{GayX2}
Gai, X., Wheeler, M.F.: {Iteratively coupled mixed and Galerkin finite element
  methods for poro-elasticity}.
\newblock Numer. Methods. Partial. Diff. Equations \textbf{23}(4), 785--797
  (2007).

\bibitem{Girault}
Girault, V., Kumar, K., Wheeler, M.F.: {Convergence of iterative coupling of
  geomechanics with flow in a fractured poroelastic medium}.
\newblock Comput. Geosci. \textbf{20}(5), 997--1011 (2016).


\bibitem{Naga_2010_preconditioner}
Haga, J.B., Osnes, H., Langtangen, H.P.: {Efficient block preconditioners for
  the coupled equations of pressure and deformation in highly discontinuous
  media}.
\newblock Int. J. Numer. Anal. Meth. Geomech. \textbf{35}(13), 1466--1482
  (2011).


\bibitem{Fullycoupled3}
Jeannin, L., Mainguy, M., Masson, R., Vidal-Gilbert, S.: {Accelerating the
  convergence of coupled geomechanical-reservoir simulations}.
\newblock Int. J. Numer. Anal. Meth. Geomech. \textbf{31}(10), 1163--1181
  (2007).

\bibitem{Jha2007}
Jha, B., Juanes, R.: {A locally conservative finite element framework for the
  simulation of coupled flow and reservoir geomechanics}.
\newblock Acta Geotechnica \textbf{2}(3), 139--153 (2007).

\bibitem{Kim20112094}
Kim, J., Tchelepi, H., Juanes, R.: {Stability and convergence of sequential
  methods for coupled flow and geomechanics: Drained and undrained splits}.
\newblock Comput. Methods. Appl. Mech. Eng. \textbf{200}(23--24), 2094--2116
  (2011).

\bibitem{Kim20111591}
Kim, J., Tchelepi, H., Juanes, R.: {Stability and convergence of sequential
  methods for coupled flow and geomechanics: Fixed-stress and fixed-strain
  splits}.
\newblock Comput. Methods. Appl. Mech. Eng. \textbf{200}(13--16), 1591--1606
  (2011).

\bibitem{RIS_0}
Kim, J., Tchelepi, H.A., Juanes, R.: {Stability, Accuracy, and Efficiency of
  Sequential Methods for Coupled Flow and Geomechanics}.
\newblock SPE J.  (2011).

\bibitem{mikelic2013phase}
Lee, S., Mikelic, A., Wheeler, M.F., Wick, T.: {Phase-field modeling of
  proppant-filled fractures in a poroelastic medium}.
\newblock Comput. Methods. Appl. Mech. Eng. \textbf{312}, 509--541 (2016)

\bibitem{Book_Lewis}
Lewis, R.W., Schrefler, B.A.: { The Finite Element Method in the Static and
  Dynamic Deformation and Consolidation of Porous Media, second edition}.
\newblock Wiley (1998).

\bibitem{Fullycouple4}
Lewis, R.W., Sukirman, Y.: {Finite element modelling of three-phase flow in
  deforming saturated oil reservoirs}.
\newblock Int. J. Numer. Anal. Meth. Geomech. \textbf{17}(8), 577--598 (1993).

\bibitem{List2016}
List, F., Radu, F.A.: {A study on iterative methods for solving Richards'
  equation}.
\newblock Comput. Geosci. \textbf{20}(2), 341--353 (2016).

\bibitem{Mandel}
Mandel, J.: {Consolidation Des Sols ({\'E}tude Math{\'e}matique)}.
\newblock G{\'e}otechnique \textbf{3}(7), 287--299 (1953).

\bibitem{Mikelic}
Mikeli{\'c}, A., Wang, B., Wheeler, M.F.: {Numerical convergence study of
  iterative coupling for coupled flow and geomechanics}.
\newblock Comput. Geosci. \textbf{18}(3-4), 325--341 (2014).

\bibitem{Mikelic3}
Mikeli{\'c}, A., Wheeler, M.F.: {Theory of the dynamic Biot-Allard equations
  and their link to the quasi-static Biot system}.
\newblock J. Math. Phys. \textbf{53}(12), 123702 (2012).

\bibitem{Mikelic2}
Mikeli{\'c}, A., Wheeler, M.F.: {Convergence of iterative coupling for coupled
  flow and geomechanics}.
\newblock Comput. Geosci. \textbf{18}(3-4), 325--341 (2013).

\bibitem{Jan_2016}
Nordbotten, J.M.: {Stable Cell-Centered Finite Volume Discretization for Biot
  Equations}.
\newblock SIAM J. Numer. Anal. \textbf{54}(2), 942--968 (2016).

\bibitem{Petersen2012}
Pettersen, O.: { Coupled Flow and Rock Mechanics Simulation Optimizing the
  coupling term for faster and accurate computation}.
\newblock nt. J. Numer. Anal. Model. \textbf{9}(3), 628--643 (2012).

\bibitem{wheeler}
Phillips, P.J., Wheeler, M.F.: {A coupling of mixed and continuous Galerkin
  finite element methods for poroelasticity I: the continuous in time case}.
\newblock Comput. Geosci. \textbf{11}(2), 131--144 (2007).

\bibitem{RaduPopKnabner2004}
Pop, I., Radu, F., Knabner, P.: {Mixed finite elements for the Richards{\rq}
  equation: linearization procedure}.
\newblock J. Comput. Appl. Math. \textbf{168}(1--2), 365--373 (2004).

\bibitem{Prevost2013}
Prevost, J.H.: {One-Way versus Two-Way Coupling in Reservoir-Geomechanical
  Models}, pp. 517--526.
\newblock American Society of Civil Engineers (2013).

\bibitem{RaduNPK15}
Radu, F.A., Nordbotten, J.M., Pop, I.S., Kumar, K.: {A robust linearization
  scheme for finite volume based discretizations for simulation of two-phase
  flow in porous media.}
\newblock J. Comput. Appl. Math. \textbf{289}, 134--141 (2015).

\bibitem{Carmen_2016}
Rodrigo, C., Gaspar, F., Hu, X., Zikatanov, L.: {Stability and monotonicity for
  some discretizations of the Biot{\rq}s consolidation model}.
\newblock Comput. Methods. Appl. Mech. Eng. \textbf{298}, 183--204 (2016).

\bibitem{Settari}
Settari, A., Mourits, F.M.: {Coupling of geomechanics and reservoir simulations
  models.}
\newblock Computational Methods and Advances in Geomechanics. pp. 2151--2158
  (1994).

\bibitem{Settari_1998}
Settari, A., Mourits, F.M.: {A Coupled Reservoir and Geomechanical Simulation
  System}.
\newblock SPE J.  (1998).

\bibitem{Settari_2001}
Settari, A., Walters, D.A.: {Advances in Coupled Geomechanical and Reservoir
  Modeling With Applications to Reservoir Compaction}.
\newblock SPE J.  (2001).

\bibitem{SHOWALTER_2016}
Showalter, R.E.: {Diffusion in Poro-Elastic Media}.
\newblock J. Math Anal. Appl. \textbf{251}(1), 310--340 (2000).

\bibitem{Thomas}
Thomas, J.:  {Sur l'analyse numerique des methodes d'elements finis hybrides et
  mixtes.} \newblock Univ. Pierre et Marie Curie, thèse (1977).

\bibitem{Fullycoupled1}
Wan, J., Durlofsky, L., Hughes, T., Aziz, K.: {Stabilized finite element
  methods for coupled geomechanics -reservoir flow simulations}, paper 
  \newblock SPE 79694 presented at the SPE Reservoir  Simulation
Symposium, Houston  (2003).

\bibitem{Tchelepi_2016_Preconditionet}
White, J.A., Castelletto, N., Tchelepi, H.A.: {Block-partitioned solvers for
  coupled poromechanics: A unified framework}.
\newblock Comput. Methods. Appl. Mech. Eng. \textbf{303}, 55--74 (2016).

\bibitem{Explicit2}
Zienkiewicz, O.C., Paul, D.K., Chan, A.H.C.: {Unconditionally stable staggered
  solution procedure for soil-pore fluid interaction problems}.
\newblock Int. J. Numer. Meth. Engng. \textbf{26}(5), 1039--1055 (1988).

\end{thebibliography}


\end{document}